\documentclass[11pt,oneside]{amsart}
\usepackage[numbers]{natbib}
\usepackage[foot]{amsaddr}
\usepackage{listings}
\usepackage[left=0.8in]{geometry}
\usepackage{graphicx}
\usepackage{xcolor}
\usepackage{enumitem}
\usepackage{constants}
\usepackage{mathtools}
\usepackage{amsmath}
\usepackage{amssymb}
\usepackage{amsthm}
\usepackage{amsfonts}
\usepackage{mathrsfs} 
\usepackage{bm}
\usepackage{bbm}
\usepackage{hyperref}
\usepackage{todonotes}

\hypersetup{
    colorlinks=true,
    citecolor=blue,
    linkcolor=purple,
    filecolor=magenta,      
    urlcolor=cyan
}

\usepackage{cleveref}
\usepackage[toc,page]{appendix}
\usepackage{caption}
\usepackage{subcaption}

\numberwithin{equation}{section}

\newtheorem{theorem}{Theorem}[section]
\newtheorem{corollary}[theorem]{Corollary}
\newtheorem{prop}[theorem]{Proposition}
\newtheorem{lemma}[theorem]{Lemma}

\newtheorem{exam}[theorem]{Example}
\newtheorem{assumption}{Assumption}

\DeclareMathOperator{\statdim}{stat.dim}
\DeclareMathOperator{\E}{\mathbb{E}}
\DeclareMathOperator{\PP}{\mathbb{P}}
\DeclareMathOperator{\iid}{\overset{\mathrm{iid}}{\sim}}

\DeclareMathOperator*{\lip}{lip}
\DeclareMathOperator{\prox}{\mathsf{prox}}

\DeclareMathOperator{\N}{\mathcal{N}}
\DeclareMathOperator{\dist}{dist}
\DeclareMathOperator{\law}{law}
\DeclareMathOperator*{\argmin}{argmin}
\DeclareMathOperator{\proj}{\Pi_{G}^\perp}

\DeclareMathOperator{\cone}{cone}

\newcommand{\inner}[2]{\langle #1, #2 \rangle}
\newcommand{\R}{\mathbb{R}}

\newcommand{\bx}{\bm{x}}
\newcommand{\be}{\bm{e}}
\newcommand{\bg}{\bm{g}}
\newcommand{\bA}{\bm{A}}
\newcommand{\bw}{\bm{w}}

\newcommand{\noop}[1]{ }
\newcommand{\ma}{\mathcal{A}}
\newcommand{\ml}{\mathcal{L}}
\newcommand{\mg}{\mathcal{G}}
\newcommand{\mh}{\mathcal{H}}
\newcommand{\mt}{\mathcal{T}}
\newcommand{\ind}{\perp\!\!\!\!\perp} 

\newcommand{\loss}{\mathsf{loss}}
\newcommand{\reg}{\mathsf{reg}}
\newcommand{\env}{\mathsf{env}}

\begin{document}
\title[Existence of solutions to nonlinear equations
characterizing M-estimation error]{
Existence of solutions to the nonlinear equations characterizing
the precise error of M-estimators
}
\date{\today}

\author{Pierre C. Bellec and Takuya Koriyama}
\address{P.C. Bellec:  Rutgers University. 
T. Koriyama: University of Chicago.}
\email{pierre.bellec@rutgers.edu, tkoriyam@uchicago.edu}

\begin{abstract}
Major progress has been made in the previous decade
to characterize the asymptotic behavior of regularized M-estimators
in high-dimensional regression problems
in the proportional asymptotic regime where the sample size $n$ and the number of features $p$ are increasing simultaneously such that $n/p\to \delta \in(0,\infty)$, using powerful tools
such as Approximate Message Passing or
the Convex Gaussian Min-Max Theorem (CGMT).
The asymptotic error and behavior of the regularized M-estimator
is then typically described by a system of nonlinear equations
with a few scalar unknowns, and the solution to this system
precisely characterizes the asymptotic error.
Application of the CGMT and related machinery requires
the existence and uniqueness of a solution to this low-dimensional system
of equations or to a related scalar convex minimization problem.

This paper resolves the question of existence and uniqueness
of solution to this low-dimensional system for the case of linear models
with independent additive noise,
when both the data-fitting loss function and 
regularizer are separable and convex. Such existence result
was previously known under strong convexity or for specific estimators
such as the Lasso.
The main idea behind this existence result is inspired
by an argument developed by
\citet{montanari2019generalization, celentano2020lasso} in different contexts:
By constructing an ad-hoc convex minimization problem in an
infinite dimensional Hilbert space, the existence of
the Lagrange multiplier for this optimization problem
makes it possible to construct explicit solutions to the
low-dimensional system of interest.

The conditions under which we derive this existence result
exactly correspond to the side of the phase transition
where perfect recovery \( \bm{\hat x}=\bm x_0 \) fails,
so that these conditions are optimal.
\end{abstract}

\maketitle
\section{Introduction}
\subsection{Proportional asymptotics: unregularized M-estimation}
Consider the linear model
$$
\bm{y} = \bm{A}\bm{x}_0 + \bm{z}.
$$
where $\bm{A}\in\R^{n\times p}$ is the design matrix, 
$\bm{x}_0\in\R^p$ is the signal of interest, and $\bm{z}\in\R^n$ is the noise vector. To estimate the signal $\bm{x}_0$ from the observed data $(\bm{y}, \bA)$
in high dimensions,
\citet{huber1964robust} proposed to use the robust M-estimator
defined as 
\begin{equation}\label{eq:intro_estimator_noreg}
    \hat{\bm{x}} \in \argmin_{\bm{x}\in\R^p} \sum_{i=1}^n \loss(y_i-\bm{e}_i^\top\bm{A}\bm{x}),
\end{equation}
where $\loss:\R\to\R$ is a loss function, that throughout the paper will be assumed convex. Above, $\bm e_i\in\R^n$ is the $i$-th canonical basis vector.

\subsubsection{A nonlinear system of equations}
Since the seminal work of \cite{el2013robust,donoho2016high},
the asymptotic behavior of the M-estimator \eqref{eq:intro_estimator_reg}
in this proportional regime has received considerable attention.
The works \cite{el2013robust,donoho2016high} established that  in
the high-dimensional regime where the sample size $n$ and the number of features $p$ are increasing simultaneously such that 
$$
n/p \to \delta \in (1, +\infty),
$$
the design $\bm A$ has iid $N(0,\frac1p)$ entries independent of the noise $\bm z$, and {the noise vector}
$\bm z$ are iid components of a random variable $Z$,
the asymptotic
behavior of the robust M-estimator $\hat{\bm x}$
is governed by the system of two equations
\begin{align}\label{eq:intro_system_reg_noreg}
    \begin{split}
        \alpha^2  &= \delta \E\bigl[
        ((\alpha G + Z) - \prox[\kappa \loss] (\alpha G+ Z))^2\bigr] \\
        \alpha &= \delta \E \bigl[((\alpha G + Z) - \prox[\kappa \loss](\alpha G+ Z))G\bigr]
    \end{split}
    \quad \text{with positive unknown} \quad (\alpha, \kappa)
\end{align}
where $G\sim  N(0,1)$ and $Z$ are independent
and $\prox[\kappa\loss](u) = \argmin_{v\in\R} \kappa \loss(v) + (u-v)^2/2$ is the proximal operator of the convex function $\kappa\loss$.

In particular, \cite{el2013robust,donoho2016high,karoui2013asymptotic}
conjectured and established under suitable assumptions that
if $(\alpha_*,\kappa_*)$ is a unique solution to the above system
then $\frac1p\|\hat{\bm x} - \bm x_0\|^2\to\alpha_*^2$ in probability.
This striking result captures the subtle variations in the risk
$\frac1p\|\hat{\bm x} - \bm x_0\|^2$ as a function of $\loss$,
the law of $Z$ and the oversampling ratio $\delta=\lim n/p$.
\citet{donoho2016high} shows the uniqueness of the solution to \eqref{eq:system_noreg} for a strongly convex \( \loss \).
Strong convexity, however,
excludes a large family of robust losses such as L1 loss $|x|$, the Huber loss: $\loss(x)=\int_{0}^{|x|} \min(1, u) du$, and its smooth variants such as the pseudo-Huber loss: $\loss(x)=\sqrt{1+x^2}-1$.
The uniqueness and existence of solution to \eqref{eq:intro_system_reg_noreg}
is also showed in \cite{sur2019likelihood} for the logistic loss and
other twice-differentiable increasing loss functions provided
$\delta>2$.

\subsubsection{The case of $\alpha_*=0$}
\citet{thrampoulidis2018precise} later relaxed the conditions
necessary to provably establish this correspondence between 
the risk of $\hat{\bm x}$ and the scalar solutions to \eqref{eq:intro_system_reg_noreg}, and
gave an alternative statement of this result that we now describe.
If $\env_g:\R\times \R_{>0}\to \R$ is the Moreau envelope defined as
$\env_{g}(x;\tau)=\min_{u\in\R} (x-u)^2/(2\tau) + g(u)$ for any convex function $g$, define the function $\mathsf{L}:\R\times \R_{>0}\to \R$ by
\begin{equation}
    \label{eq:def_noreg_L}
    \forall c\in\R, \ \tau>0,
    \qquad
    \mathsf{L} (c, \tau) := \E[\env_\loss(cG + Z; \tau)-\loss(Z)]
\end{equation}
where $G\sim N(0,1)$ is independent of $Z$ as in \eqref{eq:intro_system_reg_noreg}.
Consider the minimization problem with objective $\mathsf{M}:\R_{\ge 0}\to\R$
defined by
\begin{align}\label{eq:potential_noreg_intro}
    \min_{\alpha\ge 0} \mathsf{M}(\alpha) \quad \text{with}\quad \mathsf{M}(\alpha) := \sup_{\beta {>} 0} \inf_{\tau_g>0} \frac{\beta\tau_g}{2\delta } + \mathsf{L}(\alpha, \frac{\tau_g}{\beta}) - \frac{\alpha\beta}{\delta}.
\end{align}
The definition above of $\mathsf{M}(\alpha)$ is taken from \cite{thrampoulidis2018precise} in the special case where the penalty is 0.
A simpler but equivalent definition of $\mathsf{M}(\alpha)$ using only a
supremum is
$\mathsf{M}(0) = 0$ and $\mathsf{M}(\alpha) = \sup_{b>0} ( \mathsf{L}(\alpha, \frac \alpha b) - \frac{\alpha b}{2\delta})$, as we will prove in \Cref{lm:N_check_N}.
\citet{thrampoulidis2018precise} showed that $\mathsf{M}$ is convex,
and that if $\mathsf{M}(\alpha)$ admits a unique minimizer $\alpha_*\in[0,+\infty)$
then $\frac{1}{p}\|\hat{\bm x} - \bm x_0\|^2\to \alpha_*^2$ in probability.
This result statement is more general than that of the previous paragraph
due to allowing for $\alpha_*=0$ to be solution.
If $\alpha_*$ minimizes $\mathsf{M}$ in $(0,+\infty)$ then one can find $\kappa^*$
such that $(\alpha_*,\kappa_*)$ is solution to \eqref{eq:intro_system_reg_noreg}.
The converse is also true: if $(\alpha_*,\kappa_*)$ is a solution to
\eqref{eq:intro_system_reg_noreg} then $\alpha_*>0$ is a minimizer of $\mathsf{M}$.
This is because, as explained in \cite{thrampoulidis2018precise},
the system \eqref{eq:intro_system_reg_noreg} corresponds to the critical points
of the min-max problem \eqref{eq:potential_noreg_intro},
hence there is a one-to-one correspondence between the solutions to
\eqref{eq:intro_system_reg_noreg} and positive minimizers of $\alpha\mapsto \mathsf{M}(\alpha)$ if $\alpha_*>0$.
This viewpoint with the potential $\mathsf{M}(\alpha)$ allows {the degenerate case $\alpha_*=0$}. In this case the system \eqref{eq:intro_system_reg_noreg}
and the critical point equations 
may not hold with equality since a convex function $\mathsf{M}:\R_{\ge 0}\to \R$
is minimized at 0 if and only if 
the right-derivative of $\mathsf{M}$ at 0 is non-negative (but not necessarily 0).

\subsubsection{On existence of minimizers}
At the other extreme, if all is known is that $\mathsf{M}:\R_{\ge 0}\to \R$ is convex,
the existence of minimizers of $\mathsf{M}$ is not guaranteed.
This corresponds to the phenomenon observed 
for the logistic loss: $\loss(u)=\log(1+\exp(u))$
in \cite{sur2019likelihood}: if $\delta < 2$ then the system
\eqref{eq:intro_system_reg_noreg} has no solution and the minimization problem
\eqref{eq:intro_estimator_noreg} in $\R^p$ has no minimizer;
if {$\delta > 2$} then the system \eqref{eq:intro_system_reg_noreg} has a unique
solution and \eqref{eq:intro_estimator_noreg} admits a unique
minimizer in $\R^p$ satisfying $\frac1p\|\hat{\bm x} - \bm x_0\|^2\to \alpha_*^2$.
Here, this phase transition at $\delta=2$ intuitively holds because the
logistic loss is increasing. By the separability result of
\cite{sur2019likelihood}, any increasing $\loss$ (other than the logistic loss)
also satisfies that \eqref{eq:intro_estimator_noreg} has no minimizer when
$\delta < 2$.

However, for loss functions that are first decreasing and eventually increasing (this excludes the logistic loss and other increasing losses),
it is still unknown whether $\mathsf{M}(\alpha)$ admits a minimizer
and whether the results of \cite{karoui2013asymptotic,donoho2016high,thrampoulidis2018precise} are applicable.

\subsubsection{Summary for unregularized M-estimation}

The picture drawn by the previous discussions is that three
qualitatively different phenomena may occur.

\begin{itemize}
    \item
        $\|\hat{\bm x} - \bm x_0\|^2/p \to 0$ in probability
        (weak perfect recovery) or
        $\PP(\hat{\bm x} = \bm x_0)\to 1$ (perect recovery). 
        Thanks to \cite{thrampoulidis2018precise},
        a sufficient          condition
        $\|\hat{\bm x} - \bm x_0\|^2/p \to^P 0$ (weak perfect recovery)
        is that 0 is a unique
        minimizer of $\mathsf{M}:\R_{\ge 0}\to\R$
        is minimized at 0 only (equivalently, that $\mathsf{M}$ is increasing).
    \item
        $\|\hat{\bm x} - \bm x_0\|^2/p \to \alpha_*^2$ in probability
        for some limit $\alpha_*>0$.
        Thanks to \cite{donoho2016high,thrampoulidis2018precise},
        a sufficient solution condition for
        $\|\hat{\bm x} - \bm x_0\|^2/p \to^P \alpha_*^2$
        is that {$\alpha_*$} is the unique
        minimizer of $\mathsf{M}:\R_{\ge 0}\to\R$. 
        
    \item $\hat{\bm x}$ does not exist in the sense that \eqref{eq:intro_estimator_noreg} admits no minimizer, or the minimizer satisfies that
        $\|\hat{\bm x} - \bm x_0\|^2/p$ is unbounded.
        A sufficient condition for this phenomenon is $\delta=\lim n/p < 2$
        for the logistic loss or any increasing $\loss$.
\end{itemize}

Even though the behavior of such M-estimators
in the proportional regime has
received considerable attention in recent years, and even though
\eqref{eq:intro_estimator_noreg} is perhaps the simplest and most
studied estimator in high-dimensional statistics,
there is still no rigorous results establishing, for a given
loss function, whether the nonlinear system has a solution
or whether $\mathsf{M}(\alpha)$ is minimized at 0 or in $(0,+\infty)$.

Beyond the question of existence of solutions to the nonlinear system
\eqref{eq:intro_estimator_noreg} or of existence of minimizers 
of the potential $\mathsf{M}$, the above three qualitatively different 
phenomena raise the question of whether a clear phase transition
can be established between the first and second bullet points
(transition between perfect recovery and positive limiting risk),
and between the second and the third
(transition between finite limiting risk and no minimizer in \eqref{eq:intro_estimator_noreg}).
The transition between finite limiting risk and no minimizer in \eqref{eq:intro_estimator_noreg} is solved by \cite{sur2019likelihood} under assumptions
on $\loss$, and the transition happens when $\delta$ crosses $2$.
We are not aware of results describing the phase transition
between perfect recovery and positive limiting risk prior to the present paper.

One goal of the paper is to fill these gaps, both regarding
the phase transition between the first two bullet points above,
and regarding whether the nonlinear system \eqref{eq:intro_estimator_noreg}
admits a solution.
The paper focuses on Lipschitz regression
functions $\loss$ minimized at $0$,
which encompass all typical loss functions
used in robust regression, such as the Huber loss and its variants.
In particular,
\Cref{sec:noreg} proves the following concise result:
if $\delta>1$, $\PP(Z\ne 0)>0$ and
$\loss:\R\to\R$ is convex, Lipschitz and $\argmin\loss = \{0\}$ then
\begin{itemize}
    \item $\mathsf{M}:\R_{\ge 0}\to\R$ always admits a unique minimizer $\alpha_*\ge 0$.
    \item The phase transition between positive limiting risk and perfect recovery is located at $\delta_{\mathsf{perfect}}\in(0,+\infty]$ defined by
        \begin{equation}
            \label{eq:delta_threshold_intro_unreg}
    \delta_{\mathsf{perfect}}:= \frac{1}{(1-\inf_{t>0}\E\bigl[\dist(G, t\partial\loss(Z))^2\bigr])_+} \in (0,\infty], 
    \end{equation}
    in the following sense:
    \begin{itemize}
        \item
            If $\delta < \delta_{\mathsf{perfect}}$, then the minimizer $\alpha_*$ is strictly positive, and there exists a unique positive scalar $\kappa_*$ such that $(\alpha_*,\kappa_*)$ is the unique solution to the nonlinear system of equations \eqref{eq:intro_system_reg_noreg}
        \item If $\delta > \delta_{\mathsf{perfect}}$ then $\PP(\hat{\bm x} = \bm x_0) \to 1$.
    \end{itemize}
\end{itemize}
In the definition of $\delta_{\mathsf{perfect}}$, we use the notation
$\dist(\cdot, S)^2=\inf_{x\in S} (\cdot-x)^2 $ is the Euclidean distance to a set $S$ and $\partial g \subset \R$ is the subdifferential of the convex function $g$.

The techniques used in the present paper to establish the existence of 
a minimizer to
$\min_{\alpha\ge 0}\mathsf{M}(\alpha)$ are inspired by the recent works of
\citet{montanari2019generalization,celentano2020lasso},
which establish the existence of solutions to nonlinear systems
in different contexts:
\cite{montanari2019generalization} studies max-margin classifiers
in regression and \cite{celentano2020lasso} studies the Lasso,
i.e., square loss and L1 regularizer.
The gist of these techniques is that while studying the system
\eqref{eq:intro_system_reg_noreg} or the minimization $\min_{\alpha\ge 0}\mathsf{M}(\alpha)$
directly
is difficult due to the min-max structure of $\min_{\alpha\ge 0}\mathsf{M}(\alpha)$,
one can find a convex optimization problem in an infinite-dimensional
Hilbert space associated to $\min_{\alpha}\mathsf{M}(\alpha)$.
Here, lifting the minimization problem to this infinite-dimensional Hilbert space
makes the min-max structure disappear, resulting in a convex optimization
problem with convex constraint. Once this infinite-dimensional minimization
problem is found, tools from convex optimization
in Hilbert spaces (e.g., \cite{bauschke2017correction}) then help
establish the existence of minimizers, although now even simple questions
such as coercivity require careful arguments due to the infinite-dimensional
nature of the problem.

The infinite dimensional minimization problem is intimately related to
the minimization problem $\min_{\alpha\ge 0} \mathsf{M}(\alpha)$ in several ways:
The minimal objective value of this infinite dimensional convex minimization
problem is the same as $\min_{\alpha\ge 0}\mathsf{M}(\alpha)$,
and the Hilbert norm of the infinite-dimensional minimizer $v_*$
is related to the minimizer $\alpha_*$ of $\mathsf{M}(\alpha)$ by
$\alpha_* = \|v_*\|/\sqrt{1-\delta^{-1}}$. 
Furthermore, if the solution $v_*$ to the infinite dimensional optimization problem is {nonzero}, an explicit solution $(\alpha_*, \kappa_*)$ to the nonlinear system \eqref{eq:intro_system_reg_noreg} can be constructed from $v_*$.
{Reciprocally, from a positive solution $(\alpha_*,\kappa_*)$ to \eqref{eq:intro_estimator_noreg} can be constructed a nonzero minimizer $v_*$ to the infinite dimensional problem}. For unregularized M-estimation, this infinite-dimensional
minimization problem is constructed in \Cref{subsec:proof_outline_noreg} (see \Cref{fig:proof_sketch_noreg} for a quick overview). For regularized M-estimation discussed in the next subsection,
the infinite-dimensional minimization problem
is described in \Cref{subsec:proof_outline_reg}.

\begin{figure}
    \includegraphics[width=0.47\linewidth]{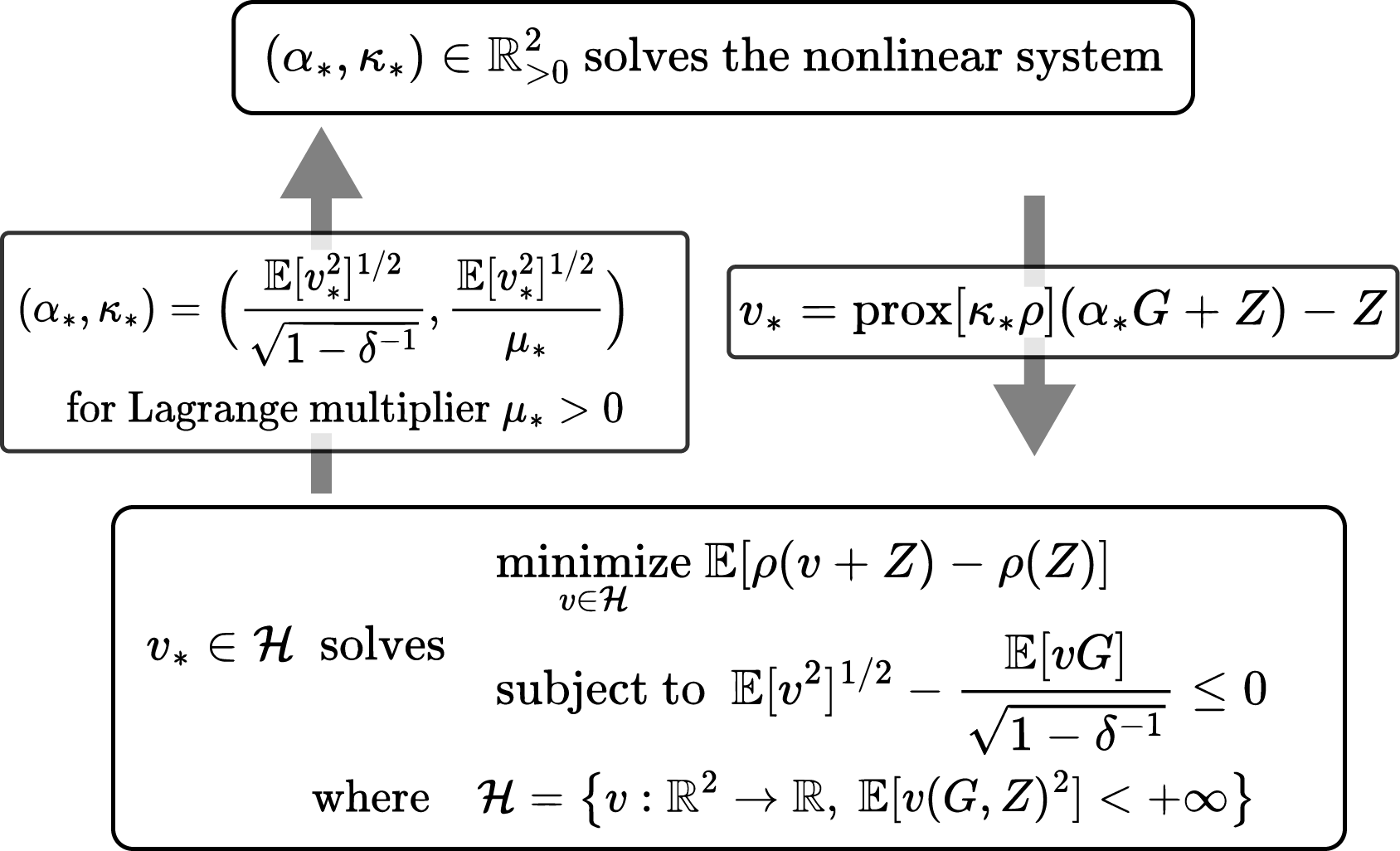}
    \caption{{Correspondence between the infinite dimensional optimization problem and solutions $(\alpha_*,\kappa_*)$ to the nonlinear system in the unregularized case \eqref{eq:intro_estimator_noreg}. Here} $\rho=\loss$, $\delta=\lim{n/p}>1$, and $(Z, G)$ are as in \eqref{eq:intro_system_reg_noreg}. }
    \label{fig:proof_sketch_noreg}
\end{figure}  

\subsection{Regularized M-estimation}
We are also interested in the asymptotic behavior of the
regularized M-estimator $\hat{\bm{x}}$
\begin{equation}\label{eq:intro_estimator_reg}
    \hat{\bm{x}} \in \argmin_{\bm{x}\in\R^p} \sum_{i=1}^n \loss(y_i-\bm{e}_i^\top\bm{A}\bm{x}) + \sum_{j=1}^p \reg(x_j),    
\end{equation}
where $\loss:\R\to\R$ is a convex loss function and $\reg:\R\to\R$ is a convex regularizer, and $n,p$ are increasing simultaneously such that
$$
n/p \to \delta \in (0, +\infty). 
$$
A representative landmark result for this regularized M-estimator from
\cite{thrampoulidis2018precise} is the following:
{
Assume that the design matrix $\bm{A}$ has iid entries
with distribution $\N(0,1/p)$, while the noise $\bm{z}\in\R^n$ and the signal $\bm{x}_0\in\R^p$
are independent with iid entries
\begin{equation}
    \label{iid}
    \bm z = (z_1,...,z_n) \sim^{\text{iid}} Z,
    \qquad
    \bm x_0 = (x_{01},...,x_{0p}) \sim^{\text{iid}} X, 
\end{equation}
for some scalar random variables $(Z, X)$. 
Given $\loss$, the regularizer $\reg$, and these marginal distributions $(Z, X)$ of noise and signal vectors, we define functions $\mathsf{L}, \mathsf{R}:\R\times \R_{>0}\to \R$ by
\begin{align}\label{eq:def_LF}
    \forall c\in\R, \ \tau>0, \quad \begin{split}
        \mathsf{L}(c, \tau) &:= \E[\env_\loss(cG + Z; \tau)-\loss(Z)]\\
        \mathsf{R}(c, \tau) &:= \E[\env_\reg(cH + X; \tau)-\reg(X)]
    \end{split}
    \quad 
    \text{for} \quad 
    \begin{cases}
        G \sim \N(0,1) \ind Z\\
        H \sim N(0,1) \ind X
    \end{cases}, 
\end{align}
Then, the asymptotic squared risk of the regularized M-estimator $\hat{\bm{x}}$ in \eqref{eq:intro_estimator_reg} is known to be characterized by
$
p^{-1}\|\hat{\bm{x}}-\bm{x}_0\|_2^2\to^p \alpha_*^2,
$
where $\alpha_*$ is a solution to the following optimization problem
 \begin{align}\label{eq:convex-concave}
  \min_{\alpha\ge 0} \mathsf{M}(\alpha) \quad \text{with}\quad \mathsf{M}(\alpha) := \sup_{\beta> 0, \tau_h>0} \inf_{\tau_g>0} \frac{\beta\tau_g}{2} + \delta \mathsf{L}(\alpha, \frac{\tau_g}{\beta}) - \frac{\alpha\tau_h}{2} - \frac{\alpha\beta^2}{2\tau_h} +  \mathsf{R}\Bigl(\frac{\alpha\beta}{\tau_h},\frac{\alpha}{\tau_h}\Bigr).
\end{align}
This characterization of the limit of $\|\hat{\bm{x}}-\bm{x_0}\|_2^2/p$ 
 captures 
the interplay between $\loss$, $\reg$, the marginal distribution $(Z, X)$ of noise and signal, and the oversampling ratio $\delta = \lim n/p$. 
However, as in the unregularized case,
the convergence result $p^{-1}\|\hat{\bm{x}}-\bm{x}_0\|_2^2\to^p\alpha_*^2$ by \cite{thrampoulidis2018precise} relies on the important assumption that 
\begin{equation}\label{as:intro_unique_exist}
    \textit{
        the minization problem in \eqref{eq:convex-concave} admits a unique solution $\alpha_*\in [0,+\infty)$.
    }
\end{equation}
This is necessary to apply
the main result of \cite{thrampoulidis2018precise}
and its generalizations (e.g., \citep{loureiro2021learning, koriyama2024precise}).

One approach to verify \eqref{as:intro_unique_exist} is to 
look at the first-order optimality condition of \eqref{eq:convex-concave}; 
using the differentiability of $(\mathsf{L}, \mathsf{R})$  \cite[Lemma 4.3]{thrampoulidis2018precise}, letting $\kappa = \tau_g/\beta$ and $\nu = \tau_h/\alpha$, one can verify that \eqref{as:intro_unique_exist} is satisfied for some $\alpha_*>0$ if there exists associated positive scalars $(\beta_*, \kappa_*, \nu_*)$ such that the pair $(\alpha_*, \beta_*, \kappa_*, \nu_*)$ is the unique solution to the nonlinear system of equations 
\begin{align}\label{eq:intro_system_reg}
    \begin{split}
        \alpha^2 &= \E \bigl[
            \bigl(
            \prox[\nu^{-1} \reg] (\nu^{-1}\beta H + X) - X
            \bigr)^2
          \bigr] \\
          \delta^{-1}\beta^2 \kappa^2 &=  \E\bigl[\bigl(\alpha G + Z - \prox[\kappa\loss](\alpha G + Z)\bigr)^2\bigr]
          \\
          \delta^{-1} \nu\alpha\kappa &= \E\bigl[G\cdot \bigl(\alpha G + Z - \prox[\kappa\loss](\alpha G + Z)\bigr)\bigr]\\
          \kappa\beta &= \E\bigl[
          H \cdot \bigl(\prox[\nu^{-1}\reg](\nu^{-1}\beta H + X) - X\bigr)
          \bigr],
    \end{split}
\end{align}
However, the uniqueness and existence of the solution to this nonlinear system are only known for a limited pair of loss and regularizer as explained below. 
}

When the loss is the least square $\loss(x)=x^2/2$, then thanks to the formula $\env_\loss(x;\kappa)=\tfrac{x^2}{2(1+\kappa)}$ 
{the first and the second equations in \eqref{eq:intro_system_reg} can be solved with respect to $\kappa$ and $\nu$ explicitly} \citep[Section 5.4]{thrampoulidis2018precise}. Substituting these to the remaining equations results in
two equations with unknowns \( (\alpha,\beta) \) and recovers
the system found in works studying the square loss
\citep{bayati2011lasso,miolane2021distribution}.
If \( \loss \) is the square loss and
$\reg$ is the Ridge regularizer, i.e., $\reg(x)={\loss(x)=x^2/2}$, a solution to \eqref{eq:intro_system_reg} uniquely exists
provided that $X$ and $Z$ have finite second moments. Indeed, random matrix theory results, such as
the Marchenko--Pastur law, are sufficient to characterize the limit of
\( p^{-1}\|\bm{\hat x}-\bm x_0\|^2 \) and the system \eqref{eq:intro_system_reg}
is equivalent to the limiting equation satisfied by the Stieltjes transform
of \( \bm A^\top \bm A \).
For regularizer other than {Ridge}, the existence and the uniqueness were established by \cite{miolane2021distribution}
by careful univariate calculus
analysis \citep[Appendix A]{miolane2021distribution},
and by \cite{celentano2020lasso} for the Lasso regularizer $\reg(x)={|x|}$ and its
smoothed variants
using techniques that inspired the argument used in the present paper.

On the other hand, for loss functions \( \loss \) other than the square loss, 
previous results on the existence and uniqueness of the solution to the nonlinear system \eqref{eq:intro_system_reg} are limited, even in the unregularized case
with regularizer $\reg=0$ {as discussed in previous sections}. 

Moreover, since  $\alpha=0$ is the boundary of the range of $\inf_{\alpha\ge 0}$ in \eqref{eq:convex-concave}, 
the approach of verifying \eqref{as:intro_unique_exist} by showing the existence and uniqueness of the nonlinear system \eqref{eq:intro_system_reg} misses the important case of the optimization problem in \eqref{eq:convex-concave} achieving the minimum at $\alpha=0$. By \cite{thrampoulidis2018precise},
if $\alpha_*=0$ is the unique minimizer then \textit{weak perfect recovery} holds in the sense that $p^{-1}\|\hat{\bm{x}}-\bm{x}_0\|_2^2\to^p \alpha_* = 0$. {
As in the case of unregularized estimation in \eqref{eq:delta_threshold_intro_unreg},
it is of interest to determine the
phase transition for perfect recovery
(or weak perfect recovery) in the context
of regularized M-estimators 
of the form \eqref{eq:intro_estimator_reg}.
}

Our contribution for the regularized M-estimation problem is summarized as follows:
\begin{itemize}
    \item $\mathsf{M}:\R_{\ge 0}\to\R$ always admits a unique minimizer $\alpha_*\ge 0$.
    \item The phase transition between positive limiting risk and perfect recovery is located at $\delta_{\mathsf{perfect}}\in(0,+\infty]$ defined by
        \begin{equation}
            \label{eq:delta_threshold_intro_reg}
    {\delta_{\mathsf{perfect}}:= \frac{\inf_{t>0} \E[\dist(G, t\partial \reg(X))^2]}{(1-\inf_{t>0}\E\bigl[\dist(G, t\partial\loss(Z))^2\bigr])_+} \in (0,\infty], }
    \end{equation}
    in the following sense:
    \begin{itemize}
        \item
            If $\delta < \delta_{\mathsf{perfect}}$, then the minimizer $\alpha_*$ is strictly positive, and there exists a unique positive scalar $\beta_*, \kappa_*, \nu_*$ such that $(\alpha_*,\beta_*, \kappa_*, \nu_*)$ is the unique solution to the nonlinear system of equations \eqref{eq:intro_system_reg}
        \item If $\delta > \delta_{\mathsf{perfect}}$ then $\PP(\exists\lambda>0 \text{ s.t. } 
            \hat{\bm x}_\lambda = \bm x_0) \to 1$ where $\hat\bx_\lambda$ is the regularized M-estimator \eqref{eq:intro_estimator_reg} computed by $\loss$ and regularizer $\lambda \cdot \reg$. 
    \end{itemize}
\end{itemize}



\subsection{Organization}
In \Cref{sec:noreg}, we discuss the unregularized case $\reg=0$ with
a robust Lipschitz loss function \( \loss \),
characterize the phase transition threshold \( \delta_{\mathsf{perfect}} \). 
In \Cref{sec:reg}, we extend the analysis to the regularized case
with \( \reg\ne 0 \),
with the characterization of \( \delta_{\mathsf{perfect}} \). 
Most proofs are delayed to the appendix.

\subsection{Notation}
For any convex function $f$, let $\partial f$ be the subdifferential, and 
let $\prox[f](x) = \argmin_{u} f(u) + (x-u)^2/2$ be the proximal operator. Given a random variable $v$, we denote by $\|v\|$ the L2 norm $\E[v^2]^{1/2}$. For any set $S\subset \R^n$, let $\dist(\cdot, S):\R^n\to\R$ be the map $\bm{x}\mapsto \inf\{
\|\bm{x}-\bm{y}\|_2
: \bm{y}\in S\}$, i.e., the Euclidean distance to the set $S$.  
For any real \( x \), denote \( x_+=\max(0,x) \).

\section{Unregularized {M-estimation}}\label{sec:noreg}
Let us start with the unregularized case ($\reg=0$). We are interested in the squared risk of the unregularized M-estimator
\begin{align}\label{eq:estimator_noreg}
    \hat{\bx} \in \argmin_{\bx \in\R^p} \sum_{i=1}^n \loss(y_i - \be_i^\top \bA \bx). 
\end{align}
Substituting $\reg=0$ to \eqref{eq:convex-concave} and dividing the objective function by $\delta$, the minimization problem characterizing the risk of unregularized M-estimator is given by
\begin{align}\label{eq:potential_noreg}
    \min_{\alpha\ge 0} \mathsf{M}(\alpha) \quad \text{with}\quad \mathsf{M}(\alpha) := \sup_{\beta > 0} \inf_{\tau_g>0} \frac{\beta\tau_g}{2\delta } + \mathsf{L}(\alpha, \frac{\tau_g}{\beta}) - \frac{\alpha\beta}{\delta}    
\end{align}
in the sense that if this optimization problem admits a unique minimizer $\alpha_*\ge 0$, then the convergence in probability 
$p^{-1} \|\hat{\bx}-\bx_0\|^2 \to^p \alpha_*^2$ holds. 
It was already observed in \cite{thrampoulidis2018precise} that
$(\alpha,\beta,\tau_g)\mapsto
\frac{\beta\tau_g}{2\delta } + \mathsf{L}(\alpha, \frac{\tau_g}{\beta}) - \frac{\alpha\beta}{\delta}$ is convex in $(\alpha,\tau_g)$ and concave in $\beta$
thanks to the Moreau envelope being jointly convex in both arguments.
As mentioned in the introduction, 
we prove in \Cref{lm:N_check_N} below that an equivalent definition of $\mathsf{M}(\alpha)$  is
$$\mathsf{M}(0) = 0,\qquad  \mathsf{M}(\alpha) = \sup_{b>0}  \mathsf{L}\Bigl(\alpha, \frac \alpha b\Bigr) - \frac{\alpha b}{2\delta} \quad\text{ for } \alpha >0.
$$
This alternative definition using just a supremum is new as far as we know.
It better captures the dimensionality of the corresponding system of equations
in the unregularized case which has dimension 2, whereas
writing the stationary equations of \eqref{eq:potential_noreg}
give three equations in $(\beta,\tau_g,\alpha)$ (one of which reveals
that a stationary point must satisfy $\alpha=\tau_g$).
The proof in \Cref{lm:N_check_N} uses the infinite-dimensional
optimization constructed in \Cref{subsec:proof_outline_noreg} below;
we could not find a proof of this equivalence of the two expressions using only simple calculus arguments.

Our aim in this section is to show that the  optimization problem
\eqref{eq:potential_noreg} admits a unique minimizer $\alpha_*\ge 0$ and characterize the condition under which the minimizer $\alpha_*$ is $0$, i.e., perfect recovery holds. 
Below we state our working assumption. 
\begin{assumption}\label{as:noreg}
    Assume that the $\loss:\R\to\R$, the oversampling ratio $\delta=\lim_{n\to+\infty} {n/p}$, and the marginal distribution $Z$ of noise vectors satisfy the conditions below. 
      \begin{itemize}
          \item $\loss:\R\to\R$ is convex, Lipschitz, and $\argmin_x\loss(x)=\{0\}$.
          \item $\delta>1$. 
          \item $\PP(Z\ne 0)>0$. 
      \end{itemize}
  \end{assumption}
  We assume $\argmin_{x}\loss(x)=\{0\}$ for simplify. It also guarantees that $\loss$ is coercive (see \Cref{lm:convex_coercive}). The condition $\delta>1$ is necessary for the unregularized M-estimator \eqref{eq:estimator_noreg} to be well-defined. 
   The condition $\PP(Z\ne 0) > 0$ is assumed to avoid the trivial case where the perfect recovery $\hat{\bx}=\bx_0$ holds with probability $1$. Indeed, if $\PP(Z\ne 0)=0$, then combined with $\{0\}=\argmin_x\loss(x)$, this gives 
    $\partial \loss(\bm{z})=\{\bm{0}_n\}$ with probability $1$, so that  $\bA^\top \partial \loss(\bm{z})=\{\bm{0}_p\}$ for the design matrix $\bA\in\R^{n\times p}$ with probability $1$. By the KKT condition for the unregularized M-estimator \eqref{eq:estimator_noreg}, this means $\hat{\bx} = \bx_0$ with probability $1$. 
  The Lipschitz condition of $\loss$ is necessary for the solution $(\alpha_*,\kappa_*)$ to the nonlinear system \eqref{eq:system_noreg} below to be finite for any noise distribution $Z$ (see \Cref{prop:lipschitz_necessary} for details). 

  With this condition, we have the following claim. 
  \begin{theorem}\label{th:main_noreg}
    Let \Cref{as:noreg} be satisfied. 
    \begin{enumerate}
        \item The minimization problem $\min_{\alpha\ge 0}\mathsf{M}(\alpha)$ \eqref{eq:potential_noreg} admits a unique minimizer $\alpha_{*}\ge 0$.
        \item 
        If the oversampling ratio $\delta=\lim n/p$ is smaller than the threshold $\delta_{\mathsf{perfect}}$ as 
        \begin{align}\label{eq:threshold_noreg}
           \delta <  \delta_{\mathsf{perfect}} := \frac{1}{\bigl(1-\inf_{t>0}\E[\dist(G, t \partial\loss(Z))^2]\bigr)_+} \in (0, +\infty],
        \end{align}
        then the minimizer $\alpha_*$ is strictly positive. In this case, there exists a unique positive scalar $\kappa_*$ such that $(\alpha_*,\kappa_*)\in\R_{>0}^2$ is a unique solution to the nonlinear system of equations
        \begin{align}\label{eq:system_noreg}
            \begin{split}
                \alpha^2  &= \delta \E\bigl[
                ((\alpha G + Z) - \prox[\kappa \loss] (\alpha G+ Z))^2\bigr] \\
                \alpha &= \delta \E \bigl[((\alpha G + Z) - \prox[\kappa \loss](\alpha G+ Z))G\bigr]
            \end{split}
            \quad \text{with positive unknown} \quad (\alpha, \kappa)
        \end{align}
        where $G\sim  N(0,1)$ and $Z$ are independent. 
    \end{enumerate}
  \end{theorem}
  \Cref{th:main_noreg} solves the existence and uniqueness problem of the minimization problem $\min_{\alpha\ge 0} \mathsf{M}(\alpha)$, which is necessary for the main theorem in \cite{thrampoulidis2018precise} to be applied. 
  
  We emphasize that the condition $\delta<\delta_{\mathsf{perfect}}$ in \eqref{eq:threshold_noreg} is optimal in the sense below. 
  \begin{prop}\label{prop:phase_transition_noreg}
    Let $\hat{\bm{x}}$ be the unregularized M-estimator in \eqref{eq:estimator_noreg} computed by a convex $\loss$ with $\argmin_x\loss(x)=\{0\}$, and  
    assume that the marginal distribution $Z$ of the noise $\bm{z}\in\R^{n}$ satisfies $\PP(Z\ne 0)>0$. Then, as $n, p\to+\infty$ with $n/p \to \delta$, we have 
    \begin{align*}
    \PP\Bigl(\bx_0 \in  \argmin_{\bm x\in\R^p}\sum_{i=1}^n \loss(y_i - \be_i^\top \bA \bx)\Bigr) \to \begin{cases}
            1 & \delta>\delta_{\mathsf{perfect}}\\
            0 & \delta < \delta_{\mathsf{perfect}}
        \end{cases}
    \end{align*}
    where $\delta_{\mathsf{perfect}}\in (0,+\infty]$ is the threshold defined by \eqref{eq:threshold_noreg}. 
    \end{prop}
We emphasize that the threshold $\delta_{\mathsf{perfect}}$  can be $+\infty$; in this case, the minimizer $\alpha_*$ is always strictly positive no matter how large the oversampling ratio $\delta=\lim n/p$ is. By the definition \eqref{eq:threshold_noreg} of $\delta_{\mathsf{perfect}}$, the threshold $\delta_{\mathsf{perfect}}$ is $+\infty$ if and only if
  $\inf_{t>0} \E[\dist(G, t\partial\loss(Z))^2] \ge 1$. A simple sufficient condition for this is
  \begin{align} \label{eq:nonzero_sufficient_noreg}
    \PP(\text{$\loss:\R\to\R$ is differentiable at $Z$})=1. 
  \end{align}
  Indeed, if \eqref{eq:nonzero_sufficient_noreg} is true, then $t\partial\loss(Z)$ is the singleton $\{t\loss'(Z)\}$ for all $t>0$ with probability $1$. Combined with the independence of $(G, Z)$, we have 
  \begin{align*}
    \E[\dist(G, t\partial\loss(Z))^2] = \E[(G-t\loss'(Z))^2]
    = 1 + t^2 \E[\loss'(Z)^2],
  \end{align*}
so taking $\inf_{t>0}$ on the both sides gives  $\inf_{t>0}\E[\dist(G, t\partial\loss(Z))^2]=1$. The condition \eqref{eq:nonzero_sufficient_noreg} is always satisfied if $\loss:\R\to\R$ is differentiable, or if $Z$ has no point masses since the set of non-differentiable point of Lipschitz function is countable. We will verify the phase transition in \Cref{prop:phase_transition_noreg} by numerical simulation in \Cref{subsec:simulation_noreg}. 

\subsection{Construction of solutions from an infinite-dimensional optimization problem}\label{subsec:proof_outline_noreg}
To show \Cref{th:main_noreg}, 
we consider an infinite-dimensional convex optimization problem. To begin with, let us consider the almost sure equivalent classes $\mh$
of squared integrable measurable functions of $(G, Z)$
$$
\mh=\{v:\R^2\to\R: \E[v(G,Z)^2]<+\infty\}, 
$$
with independent standard normal $G\sim  N(0,1)$ and the noise distribution $Z$. {Almost sure equivalence classes of $\mh$ form} a Hilbert space 
equipped with the usual inner product $\langle u,v\rangle := \E[u(G,Z)v(G,Z)]$.
{We will sometimes refer to $\mh$ itself as the Hilbert space,
in this case we implicitly identify random variables
$v(G,Z)$ that are equal almost surely.
Now we define the real-valued functions $\ml$ and $\mg$ over $\mh$ as  
\begin{equation}
    \label{noreg_ml_mg}
\begin{aligned}
   &\ml:\mh\to\R, \quad  v \mapsto \E[\loss(v(G,Z)+Z)-\loss(Z)], \\
    &\mg: \mh\to\R, \quad v \mapsto \E[v(G, Z)^2]^{1/2} - \E[v(G,Z)G]/\sqrt{1-\delta^{-1}},
\end{aligned}
\end{equation}
where $\loss:\R\to\R$ is the Lipschitz and convex loss function used for the unregularized M-estimator \eqref{eq:estimator_noreg} and $\delta>1$ is the oversampling ratio $\lim_{n\to+\infty} n/p$.
 Note that the above functions $(\ml,\mg)$ are Lipschitz and finite-valued convex functions for any noise distribution $Z$
(see \Cref{lm:lg_basic_noreg}).

{
For notational brevity, we identify the function $v\in\mh$ with the random variable $v(G,Z)$, and we write for instance $\E[v^2]$ for $\E[v(G,Z)^2]$.
Any $v$ inside an expectation denotes the random variable $v(G,Z)$,
while $v$ outside an expectation sign, for instance in $\mg(v)$, $\ml(v)$ or $\|v\|$,
refers to the element of $\mh$.
The norm $\|v\|$ refers to the Hilbert norm
$\E[v(G,Z)^2]^{1/2}$.
With this convention, the inequality $\mg(v)\le 0$ can be rewritten
$\|v\| \le \E[Gv]/\sqrt{1-\delta^{-1}}$ where the arguments $(G,Z)$ of $v(G,Z)$
are omitted inside the expectation sign. Similarly, we write
$\mathcal L(v)=\E[\loss(v+Z) - \loss(Z)]$.
}

With these functions $(\ml, \mg)$, we claim that the minimization problem $\min_{\alpha\ge 0}\mathsf{M}(\alpha)$ in \eqref{eq:potential_noreg} admits a unique minimizer $\alpha_*\ge 0$ if and only if the following infinite-dimensional convex optimization problem over the Hilbert space $\mh$
\begin{align}\label{eq:const_optim_noreg}
    \min_{v\in\mh} \ml(v) \quad \text{subject to}\quad \mg(v) \le 0          
\end{align}
admits a unique minimizer $v_*\in \mh$.

\begin{lemma}\label{lm:equivalence_potential_noreg}
  Let \Cref{as:noreg} be fulfilled. 
    \begin{enumerate}
        \item For any $\alpha\ge 0$, the potential $\mathsf{M}(\alpha)$ defined in \eqref{eq:potential_noreg} equals {the minimal objective value}
        \begin{align}\label{eq:equivalence_potential_noreg}
            \min_{v\in\mh} \ml(v) \quad \text{subject to}\quad \|v-\alpha G\|\le \delta^{-1/2}\alpha. 
        \end{align}
        \item The potential $\mathsf{M}(\alpha)$ is minimized at $\alpha=\alpha_*$ if and only if there exists a minimizer $v_*$ of the  infinite-dimensional optimization problem \eqref{eq:const_optim_noreg}
        with $\alpha_* = \|v_*\|/\sqrt{1-\delta^{-1}}$. 
    \end{enumerate}
\end{lemma}
See \Cref{proof:lm:equivalence_potential_noreg} for the proof. 
\Cref{lm:equivalence_potential_noreg}-(1) implies that the map $[0,+\infty)\ni \alpha\mapsto \mathsf{M}(\alpha)$ is finite valued since the objective value of \eqref{eq:equivalence_potential_noreg} is finite for any $\alpha\ge 0$. Indeed, the set of $v\in\mh$ satisfying the condition $\|v-\alpha G\| \le \delta^{-1/2}\alpha$ is a closed, convex and bounded subset of $\mh$, and for any continuous functions $f:\mh\to \R$ and any convex, closed, and bounded subset $\mathcal{D}\subset \mh$, the optimization problem $\min_{v\in \mathcal{D}} f(v)$ admits a minimizer
(see
\cite[proposition 11.15]{bauschke2017correction}). 
\Cref{lm:equivalence_potential_noreg}-(2) implies that the minimization problem $\min_{\alpha\ge 0}\mathsf{M}(\alpha)$ admits a unique minimizer if the infinite-dimensional optimization problem $\min_{\mg(v)\le 0}\ml(v)$ admits a unique minimizer $v_*$.

Note that \Cref{lm:equivalence_potential_noreg}-(2) easily follows from \Cref{lm:equivalence_potential_noreg}-(1). Indeed, \Cref{lm:equivalence_potential_noreg}-(1) implies 
$$
\min_{\alpha\ge 0} \mathsf{M}(\alpha) = \min_{\alpha\ge 0} \min_{v\in \mh: \|v-\alpha G\|\le \delta^{-1/2}\alpha } \ml(v).
$$
Letting $\proj(v):=v-\E[vG]G$ be the projection of $v$ onto the orthogonal complement of the linear span of $G$, the constraint $\|v-\alpha G\|\le  \delta^{-1/2}\alpha$ can be rewritten as the following norm constraint of the projection $\proj(v)$:
$$
\|\proj(v)\|^2 \le -(1-\delta^{-1})\alpha^2 + 2\alpha \E[vG] - \E[vG]^2.
$$ 
Since the right-hand side is maximized at $\alpha=\E[vG]/(1-\delta^{-1})$ for fixed $\E[vG]$, we have $\alpha_* = \E[v_*G]/(1-\delta^{-1})$. Furthermore, we will argue that the constraint is binding at the minimizer $v_*$, i.e., $\mg(v_*)= 0$ (see \Cref{lm:lagrange_noreg}), which gives $\alpha_* = \E[v_*G]/(1-\delta^{-1})=\|v_*\|/\sqrt{1-\delta^{-1}}$. 

Next, we will establish the relationship between the optimization problem $\min_{\mg(v)\le 0}\ml(v)$ and the nonlinear system of equations \eqref{eq:system_noreg}. The key to this is the existence of the Lagrange multiplier associated with the constraint $\mg(v)\le 0$. For any $v_*\in\mh$, $v_*$ solves the constrained optimization problem $\min_{\mg(v)\le0}\ml(v)$ if and only if there exists a  Lagrange multiplier $\mu_*\ge 0$ such that the KKT condition 
\begin{align}\label{eq:KKT_noreg_main}
    -\mu_* \partial \mg(v_*) \cap \partial \ml(v_*) \ne \emptyset, \quad \mu_*\mg(v_*)=0, \quad \text{and} \quad \mg(v_*) \le 0 
\end{align}
is satisfied, where $\partial \mg, \partial\ml \subset \mh$  are subdifferentials of the convex functions $\mg, \ml:\mh\to\R$. Importantly, we will argue that the Lagrange multiplier $\mu_*$ is always strictly positive  if $\PP(Z\ne 0)>0$ and $\{0\}=\argmin_x \loss(x)$ (see \Cref{lm:lagrange_noreg}). Then, combined with $\mu_*\mg(v_*)=0$ in the KKT condition \eqref{eq:KKT_noreg_main}, we have always $\mg(v_*)=0$, i.e., the constraint $\mg(v)\le 0$ is binding. 
With this positive Lagrange multiplier $\mu_*>0$ and the binding condition $\mg(v_*)=0$, we claim the following equivalence between the minimizer of $\min_{\mg(v)\le 0}\ml(v)$ and the solution to the nonlinear system of equations \eqref{eq:system_noreg}. 
\begin{lemma}\label{lm:equivalence_system_noreg}
    Let \Cref{as:noreg} be fulfilled.
    \begin{enumerate}
        \item If $v_*\in \mh$ solves the optimization problem $\min_{\mg(v)\le 0}\ml(v)$ with $\|v_*\|>0$, then for any Lagrange multiplier $\mu_*>0$ satisfying the KKT condition \eqref{eq:KKT_noreg_main}, 
        the positive scalars $(\alpha_*, \kappa_*)$ defined by 
        $$
        \alpha_* = \|v_*\|/\sqrt{1-\delta^{-1}} > 0, \quad \kappa_* = \|v_*\|/\mu_* > 0,
        $$
        solve the nonlinear system of equations \eqref{eq:system_noreg}. 
        \item         
        If $(\alpha_*, \kappa_*)\in\R_{>0}^2$ is a solution to the nonlinear system of equations \eqref{eq:system_noreg}, then $v_*\in\mh$ defined by 
        $$
            v_* = \prox[\kappa_* \loss](\alpha_* G + Z)-Z 
        $$
        solves the constrained optimization problem \eqref{eq:const_optim_noreg} with $\|v_*\|=\alpha_*\sqrt{1-\delta^{-1}}>0$, and the KKT condition \eqref{eq:KKT_noreg_main} is satisfied by the Lagrange multiplier $\mu_*=\alpha_* \sqrt{1-\delta^{-1}}/\kappa_*>0$. 
    \end{enumerate}
\end{lemma}
See \Cref{proof:lm:equivalence_system_noreg} for the proof. 
\Cref{lm:equivalence_system_noreg} implies that the nonlinear system of equations \eqref{eq:system_noreg} admits a unique solution $(\alpha_*, \kappa_*)\in\R_{>0}^2$ if and only if the optimization problem $\min_{\mg(v)\le0}\ml(v)$ admits a unique nonzero solution $v_*\in\mh$ with $v_*\ne 0$ and a unique Lagrange multiplier $\mu_*>0$ satisfying the KKT condition \eqref{eq:KKT_noreg_main}. 
The lemma below gives the uniqueness of the minimizer $v_*$ and the Lagrange multiplier $\mu_*$. 
\begin{lemma}[Uniqueness]\label{lm:unique_noreg}
    Let $v_{*}\in\mh$ be a minimizer of the optimization problem $\min_{\mg(v)\le 0}\ml(v)$. 
     \begin{enumerate}
        \item $v_{*}$ is unique. 
        \item If $v_{*}\ne 0$, then the Lagrange multiplier $\mu_*>0$ satisfying the KKT condition \eqref{eq:KKT_noreg_main} is also unique. 
     \end{enumerate}
\end{lemma}
By \Cref{lm:unique_noreg}-(1) and \Cref{lm:equivalence_potential_noreg}, the minimizer of $\min_{\alpha\ge 0}\mathsf{M}(\alpha)$ is unique if it exists. Furthermore, \Cref{lm:unique_noreg}-(2) and \Cref{lm:equivalence_system_noreg}-(2) imply that the solution to the nonlinear system of equations is also unique if $v_*$ is not degenerate, i.e., $v_*\ne 0$. The next lemma gives us a sufficient condition for $v_*\ne 0$. 
\begin{lemma}[Non-degeneracy]\label{lm:nonzero_noreg}
    Let \Cref{as:noreg} be fulfilled and 
    let $\delta_{\mathsf{perfect}}$ be the threshold defined by \eqref{eq:threshold_noreg}. 
    If $\delta < \delta_{\mathsf{perfect}}$ then $0\in\mh$ is not a minimizer of the optimization problem $\min_{\mg(v)\le 0}\ml(v)$. 
\end{lemma}
See \Cref{proof:lm:nonzero_noreg} for the proof. \Cref{lm:nonzero_noreg} claims that when the oversampling ratio $\delta=\lim n/p$ is smaller than the threshold $\delta_{\mathsf{perfect}}$, then the degenerate case $v_*=0$ cannot happen.  

Finally, we prove {in the next lemma} the existence of the optimization problem $\min_{v\in\mh: \mg(v)\le 0}\ml(v)$ and thereby conclude the proof of \Cref{th:main_noreg}. The key to this existence is the coercivity of the objective function $\mh\ni v\mapsto \ml(v)$ over the subset of constraint $\{v\in\mh:\mg(v)\le 0\}\subset \mh$, as coercivity implies the existence of minimizers from standard results
in convex optimization in Hilbert spaces.
We recall for convenience the definition of coercivity:
a function $f:\mh\to\R$ is coercive over a convex set $C$ if its level sets are bounded, in the sense that
$\sup \{\|v\|, v\in C: f(v) \le t\}$ is bounded for all
real $t$.
Here, even though the function $\loss:\R\to\R$ is coercive
thanks to $\argmin_x\loss(x)=\{0\}$,
elements of $\mh$ are random variables (more precisely, squared integrable
measurable functions of $(G,Z)$),
and the function $\ml:\mh\to\R$ in \eqref{noreg_ml_mg} is 
not necessarily coercive over $\mh$: if $\PP(Z=0)=1$ and $\loss(x)=|x|$
the function 
$v_n(G,Z)=\min(n, |G|^{-1/2})$ has $\ml(v_n) = \E[|v_n|]\le \E[|G|^{-1/2}]<+\infty$ 
but $\E[v_n^2]\to+\infty$ as $n\to+\infty$.
In our case, coercivity holds over the set $\{v\in\mh: \mg(v)\le 0\}$
thanks to the constraint
$\mg(v)\le 0$ and the following argument:
with $t>0$ deterministic, $\mg(v)\le 0$ gives
by the Cauchy--Schwarz inequality
$$
\sqrt{1-\delta^{-1}}\|v\| \le 
\E\Bigl[G v\Bigr]
=
\E\Bigl[G v I\{|G|>t\}\Bigr] + \E\Bigl[G v I\{|G|\le t\}\Bigr]
\le \E\Bigl[G^2 I\{|G|>t\}\Bigr]^{1/2} \|v\| + t \E\Bigl[|v|\Bigr].
$$
Choosing $t>0$ large enough gives $\|v\| \le C(\delta) \E[|v|]$
for some constant depending on $\delta$ only.
By the Paley--Zygmund inequality, this gives a small-ball property of the form
$\PP(v(G,Z) > c(\delta) \|v\|) \ge c(\delta)$ for some constant $c(\delta)>0$
depending only on $\delta$.
As detailed in the proof of the next lemma,
this small ball property allows to deduce the coercivity of $\ml$ over the set $\{v\in\mh: \mg(v)\le 0\}$ 
from the coercivity of $\loss:\R\to\R$.

\begin{lemma}[Coercivity]\label{lm:coercive_noreg}
    Let \Cref{as:noreg} be fulfilled. Then
    $$
    \ml(v) \ge C_1 \|v\| - C_2 \quad \text{for all $v\in\mh$ such that $\mg(v)\le 0$}
    $$ 
    where $(C_1, C_2)$ are positive constants that only depends on $(\loss, \law(Z), \delta)$. 
\end{lemma}
See \Cref{proof:lm:coercive_noreg} for the proof. 
The existence of a minimizer of $\min_{\mg(v)\le 0}\ml(v)$ immediately follows from \Cref{lm:coercive_noreg}. Indeed, by $\mg(0)=0$ and $\ml(0)=0$, the objective value of $\min_{\mg(v)\le 0}\ml(v)$ will not change even if the constraint set $\{\mg(v)\le 0\}$ is restricted to the intersection $\mathcal{D}=\{\mg(v)\le 0\} \cap \{\ml(v)\le 0\}$. Here, the subset $\mathsf{D}\subset \mh$ is closed and convex by the convexity and Lipschitz condition of $\ml$ and $\mg$ (see \Cref{lm:lg_basic_noreg}), while the set $\mathcal{D}$ is bounded since \Cref{lm:coercive_noreg} implies that the level set $\{\ml(v)\le 0\}$ is included in the closed ball of radius $C_2/C_1 <+\infty$. 
Thus, by \cite[Proposition 11.15]{bauschke2017correction}, $\min_{\mg(v)\le 0} \ml(v)=\min_{v\in \mathcal{D}}\ml(v)$ admits a minimizer.  

By the representation of the solution  $\alpha_*=\|v_*\|/\sqrt{1-\delta^{-1}}$ established in \Cref{lm:equivalence_potential_noreg} and the bound of $L_2$ norm $\|v_*\|\le C_2/C_1$ as we discussed in the above paragraph, it holds that $\alpha_* \le C_2/(C_1 \sqrt{1-\delta^{-1}})$, so 
the explicit expression for the constants $(C_1, C_2)$ in \Cref{lm:coercive_noreg} gives an explicit upper bound of $\alpha_*$ (see \Cref{lm:coercive_noreg_detail} and \Cref{subsec:upper_bound_alpha_noreg}). 


\subsection{Numerical simulation}\label{subsec:simulation_noreg}
In this section, we compare the risk behaviors of the unregularized M-estimators computed by the L1 loss and the Huber loss
\begin{align*}
    \rho_{\text{L1}}(x) = |x|, \quad \rho_{\text{Huber}}(x) = \left\{
        \begin{array}{ll}
            x^2/2 & |x| \le 1\\
            |x| - 1/2 & |x| \ge 1
        \end{array}
    \right., 
\end{align*}
as we change the oversampling ratio $\delta=n/p$. Here, the marginal distribution $Z$ of the noise vector $\bm{z}$ is fixed to the mixture of point mass $\delta_{0}$ and standard normal $N(0,1)$:
$$
\bm{z}  \iid Z = 0.9 \delta_{0} + 0.1 N(0,1)
$$
In this case, the threshold $\delta_{0}$ defined in \eqref{eq:threshold_noreg} is $+\infty$ for the Huber loss since the Huber loss is differentiable (see the discussion around \eqref{eq:nonzero_sufficient_noreg}), which means that the perfect recovery is impossible for the Huber loss. 
We see from \Cref{fig:risk_compare_noreg} that the risk of the Huber is strictly positive for all $\delta$. On the other hand, the risk of L1 loss achieves $0$ for sufficiently large $\delta$, and the point at which the risk achieves $0$ coincides with the theoretical threshold $\delta_{\mathsf{perfect}}$. 

Next, we focus on the L1 loss and consider the  following family of noise distributions
$$
\bm{z} \iid (1-s) \delta_{0} + s \cdot \text{Cauchy}(0,1) \quad \text{for} \quad s\in (0,1). 
$$
where the parameter $s=\PP(z_i\ne 0)$ controls the probability that the noise entry $z_i$ takes a nonzero value. In \Cref{fig:perfect_recovery_noreg}, we plot the 2D heatmap of the empirical probability of the perfect recovery $\hat{\bx}=\bx_0$ for different $s=\PP(z_i\ne 0)$ and oversampling ratio $\delta=n/p$ with $n$ fixed to $100$. We also compute the theoretical threshold $\delta_{\mathsf{perfect}}$ for each $s\in (0,1)$ and plot the curve $s\mapsto \delta_{\mathsf{perfect}}$ on the same figure. We observe that the curve $s\mapsto \delta_{\mathsf{perfect}}$ perfectly separates the region where the perfect recovery holds and the other region in which the perfect recovery is impossible. 
\begin{figure}
    \begin{subfigure}{0.48\textwidth}
        \includegraphics[width=\textwidth]{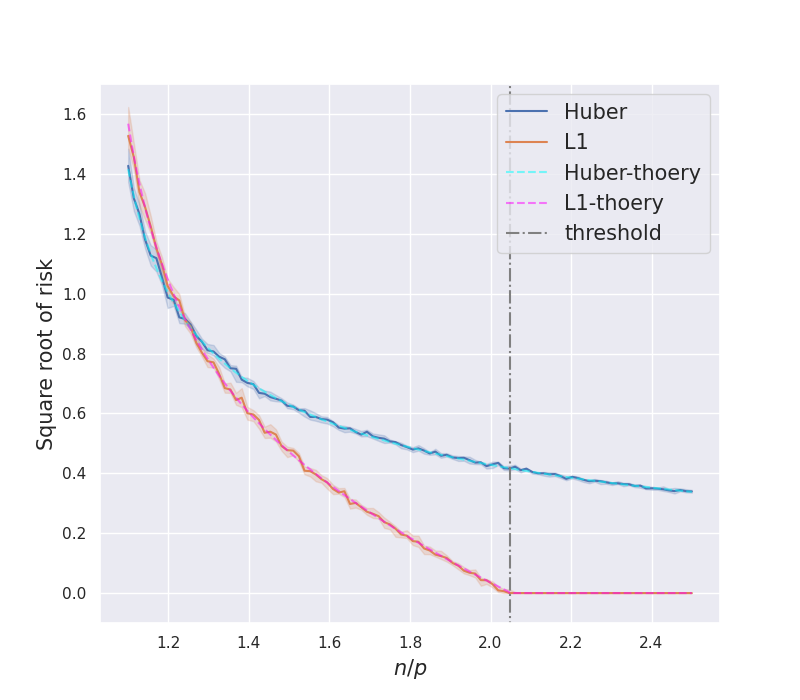}        
    \end{subfigure}
    \begin{subfigure}{0.48\textwidth}
        \includegraphics[width=\textwidth]{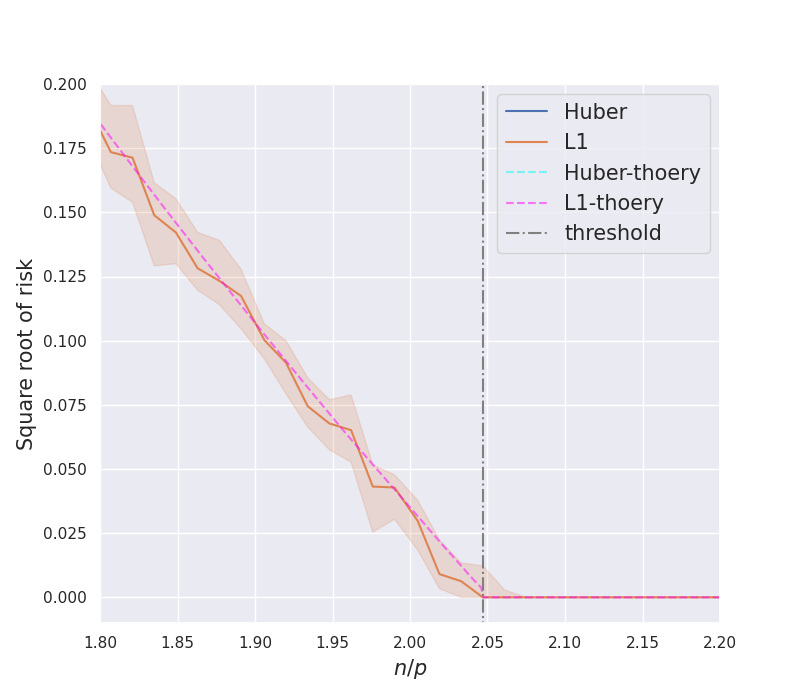}        
    \end{subfigure}
    \caption{Risk behavior of the unregularized M-estimator computed by L1 loss and Huber loss. The vertical dashed line is the threshold given by \Cref{as:reg}-(3) for the L1 loss. 
    The noise distribution is fixed to $0.9 {\delta_0} + 0.1 N(0,1)$, $p=500$ and $20$ repetitions. }
     \label{fig:risk_compare_noreg}
\end{figure}

\begin{figure}
    \includegraphics[width=\linewidth]{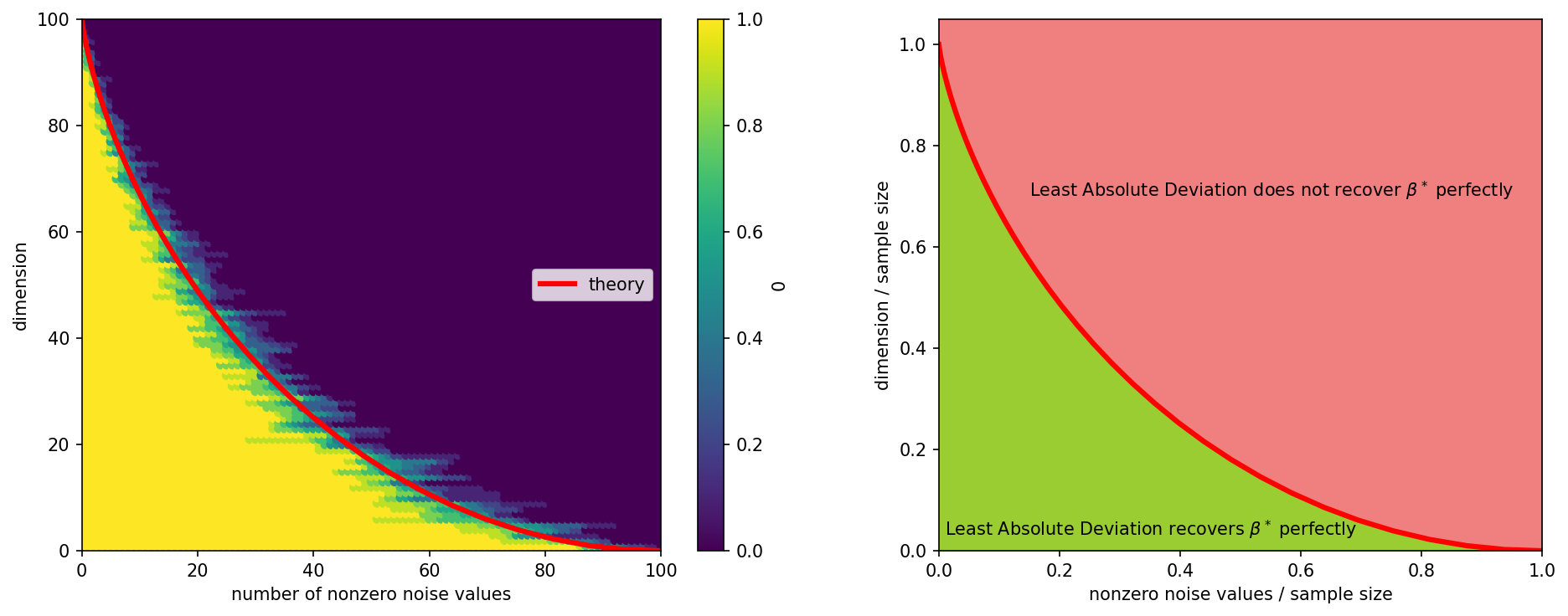}
    \caption{Empirical probabilities for perfect recovery and theoretical phase transition for the unregularized M-estimator computed by the L1 loss. }
    \label{fig:perfect_recovery_noreg}
\end{figure}

\section{Regularized {M-estimation}}\label{sec:reg}
We now extend the previous results with no regularizer ($\reg=0$)
to the regularized case $(\reg\ne 0)$. We consider the squared risk of the regularized M-estimator: 
\begin{align}\label{eq:estimator_reg}
    \hat{\bm{x}} \in \argmin_{\bx\in\R^p} \sum_{i=1}^n \loss(y_i-\bm{e}_i^\top \bA \bx) + \sum_{j=1}^p \reg(x_j).     
\end{align}
Let us define $\mathsf{M}(\alpha)$ by
\begin{align}\label{eq:potential_reg}
  \mathsf{M}(\alpha) =  \sup_{\beta > 0, \tau_h>0} \inf_{{\tau_g}>0} \frac{\beta\tau_g}{2} + \delta \mathsf{L}(\alpha, \frac{\tau_g}{\beta}) - \frac{\alpha\tau_h}{2} - \frac{\alpha\beta^2}{2\tau_h} + \mathsf{R}\Bigl(\frac{\alpha\beta}{\tau_h},\frac{\alpha}{\tau_h}\Bigr) \quad \text{for all $\alpha \ge 0$}. 
\end{align}
Recall that if $\min_{\alpha\ge 0}\mathsf{M}(\alpha)$ admits a unique minimizer $\alpha_*\ge 0$, this minimizer $\alpha_*$ characterizes the squared risk of the regularized M-estimator $\hat{\bm{x}}$ as $p^{-1}\|\hat{\bm{x}}-\hat{\bm{x}}_0\|^2\to \alpha_*^2$. Below, we state our working assumption. Our goal in this section is to show the uniqueness and existence of the minimizer $\alpha_*$. 
\begin{assumption}\label{as:reg}
    Assume that $\loss$ and $\reg$ are convex and Lipschitz with $\{0\}=\argmin_x\loss(x)=\argmin_x \reg(x)$. We also assume that the conditions below are fulfilled. 
    \begin{enumerate}
        \item $\PP(Z\ne 0)>0$. 
        \item Either one of the following conditions holds:
        \begin{enumerate}
            \item  $\reg$ is differentiable,
            \item  $\reg$ has a finite number of nondifferentiable points
            and $X$ is unbounded, i.e., $\PP(|X|>M)>0$ for all $M>0$. 
        \end{enumerate}
    \end{enumerate}
\end{assumption}

\begin{theorem}\label{th:main_reg}
Let \Cref{as:reg} be satisfied. Then, we have the following. 
\begin{enumerate}
    \item The optimization problem $\min_{\alpha\ge 0} \mathsf{M}(\alpha)$  admits a unique minimizer $\alpha_{*}\ge 0$.
    \item The minimizer $\alpha_{*}$ is strictly positive if 
    \begin{equation}\label{eq:threshold_reg}
        \delta < \delta_{\mathsf{perfect}} := \frac{\inf_{t>0} \E[\dist(H, t \partial \reg(X))^2]}{(1-\inf_{t>0}\E[\dist(G, t\partial\loss(Z))^2])_+} \in (0,\infty].
    \end{equation}
    In this case, there exists some positive scalars $(\beta_*, \kappa_*, \nu_*)$ such that $(\alpha_*, \beta_*, \kappa_*, \nu_*) \in \R_{>0}^4$ is the unique solution to the nonlinear system of equations
    \begin{equation}\label{eq:system_reg}
      \begin{split}
          \alpha^2 &= \E \bigl[
              \bigl(
                \prox[\nu^{-1} \reg] (\nu^{-1}\beta H + X) - X
              \bigr)^2
            \bigr] \\
            \delta^{-1}\beta^2 \kappa^2 &=  \E\bigl[\bigl(\alpha G + Z - \prox[\kappa\loss](\alpha G + Z)\bigr)^2\bigr]
            \\
            \delta^{-1} \nu\alpha\kappa &= \E\bigl[G\cdot \bigl(\alpha G + Z - \prox[\kappa\loss](\alpha G + Z)\bigr)\bigr]\\
            \kappa\beta &= \E\bigl[
              H \cdot \bigl(\prox[\nu^{-1}\reg](\nu^{-1}\beta H + X) - X\bigr)
            \bigr]
      \end{split} \quad \text{for unknown $(\alpha, \beta, \kappa, \nu)\in\R_{>0}^4$}. 
    \end{equation} 
\end{enumerate}
\end{theorem}

The next proposition discusses the optimality of the threshold $\delta_{\mathsf{perfect}}$ in \eqref{eq:threshold_reg}. 
\begin{prop}\label{prop:phase_transition_reg}
    Suppose $\PP(Z\ne 0)>0$, $\PP(X\ne 0)>0$, and $\argmin_{x}\loss(x)=\argmin_x \reg(x)=\{0\}$. For any $\lambda>0$, we denote by $\hat{\bm{x}}_\lambda$ the regularized M-estimator  \eqref{eq:estimator_reg} computed by $\loss$ and regularizer $\lambda \reg$.  Then, as $n/p\to \delta$, we have
    \begin{equation}\label{eq:phase_transition}
      \PP\bigl(
      \text{There exists some $\lambda>0$ such that }  
      \hat{\bm{x}}_\lambda =\bm{x}_0\bigr) \to \begin{cases}
        1 & \delta > \delta_{\mathsf{perfect}}\\
        0 & \delta < \delta_{\mathsf{perfect}}
      \end{cases}
    \end{equation}
    where $\delta_{\mathsf{perfect}}\in (0, +\infty]$ is the threshold  defined by \eqref{eq:threshold_reg}. 
\end{prop}
\Cref{prop:phase_transition_reg} implies that when the oversampling ratio $\delta=\lim n/p$ is larger than the threshold $\delta_{\mathsf{perfect}}$ then there exists some regularization parameter $\lambda>0$ such that the perfect recovery holds with high probability.

\subsection{Construction of solutions from an infinite-dimensional optimization problem}\label{subsec:proof_outline_reg}
We use the same technique as in \Cref{subsec:proof_outline_noreg} to show \Cref{th:main_reg}; we derive a \textit{dual} optimization problem on a Hilbert space such that the analysis of the potential $\mathsf{M}(\alpha)$ in \eqref{eq:potential_reg} and the nonlinear system \eqref{eq:system_reg} can be reduced to the optimization problem over the Hilbert space. 

Now we consider the product of Hilbert spaces
\begin{align*}
  \mh := \mh_Z\times \mh_X, \quad 
  \left\{
  \begin{array}{l}
    \mh_Z :=  \{v: \R^2\to \R, \ \E[v(G, Z)^2] <+\infty\}\\
    \mh_X := \{w: \R^2\to \R, \ \E[w(H, X)^2] <+\infty\}
  \end{array}
  \right.
\end{align*}
where $\mh_{Z}$ and $\mh_X$ are the almost sure equivalent classes of squared integrable measurable functions of $(G, Z)$ and $(H, X)$, respectively. Note that $\mh$ is also a Hilbert space equipped with the inner product $\inner{(v,w)}{(v',w')}_{\mh}=\E[vv']+\E[ww']$. Now we define the functions $\ml$ and $\mg$ on this Hilbert space $\mh$ by
\begin{align*}
& \ml: \mh \to \R, \quad (v, w) \mapsto \delta \E[ \loss(v+Z)-\loss(Z)] + \E[\reg(w+X)-\reg(X)],\\
  &\mg: \mh \to \R, \quad (v, w) \mapsto  \mt(v,w) - \delta^{-1/2} \E[Hw], 
\end{align*}
where $\loss$ and $\reg$ are data-fitting loss and regularizer, and $\mt:\mh\to\R$ is the convex function
$$
\mt:\mh\to\R, \quad  (v,w) \mapsto \sqrt{\bigl(\E[w^2]^{1/2}-\E[vG]\bigr)_+^2 + \E[\proj(v)^2]}
$$
where $\proj(v) := v - \E[vG]G$ is the projection onto the orthogonal complement of the linear span of $G$. With the above notation, we claim that the minimization problem $\min_{\alpha\ge 0} \mathsf{M}(\alpha)$ admits a solution if and only if the infinite-dimensional convex optimization problem
\begin{align}\label{eq:const_optim_reg}
    \min_{v,w\in \mh} \ml(v,w) \quad \text{subject to}\quad  \mg(v, w)\le 0 
\end{align}
admits a minimizer $(v_*, w_*)$. 

\begin{lemma}\label{lm:equivalence_potential_reg}
  Let \Cref{as:reg} be satisfied. 
  \begin{enumerate}
    \item  For all $\alpha\ge 0$, the potential $\mathsf{M}(\alpha)$ defined in \eqref{eq:potential_reg} is equal to the objective value of 
    \begin{align*}
      \min_{v,w} \ml(v,w) \quad \text{subject to} \quad \begin{cases}
        \|w\| \le \alpha\\
        \|v-\alpha G\| \le \E[Hw]/\sqrt{\delta}. 
      \end{cases}
    \end{align*}
    \item The solution to this optimization problem also solves the optimization problem with the relaxed condition below:
    \begin{align*}
        \min_{v,w} \ml(v,w) \quad \text{subject to} \quad \begin{cases}
            \|w\| \le \alpha\\
           \{(\alpha-\E[vG])_+^2 + \|\proj(v)\|^2\}^{1/2} \le \E[Hw]/\sqrt{\delta}
          \end{cases} 
    \end{align*}
    \item $\mathsf{M}(\alpha)$ is minimized at $\alpha_*\ge 0$ if and only if there exists a minimizer $(v_*,w_*)$ to the constrained optimization problem \eqref{eq:const_optim_reg} with $\|w_*\|=\alpha_*$. 
\end{enumerate} 
\end{lemma}
See \Cref{proof:lm:equivalence_potential_reg} for the proof. Note in passing that \Cref{lm:equivalence_potential_reg}-(1) implies that 
the map $[0, +\infty)\ni \alpha \mapsto \mathsf{M}(\alpha)$ is finite valued since the set of $(v,w)$ satisfying the condition in (1) is closed, convex, and bounded.  
\Cref{lm:equivalence_potential_reg}-(3) immediately follows from \Cref{lm:equivalence_potential_reg}-(2) since for fixed $w\in\mh_X$, the subset 
$\mathcal{V}_\alpha \subset \mh_Z$ indexed by $\alpha\ge 0$ that consists of $v$ satisfying the second condition in (2), i.e., 
$$
\mathcal{V}_\alpha
:= \Bigl\{
    v\in \mh_Z:  \bigl\{(\alpha-\E[vG])_+^2 + \|\proj(v)\|^2\bigr\}^{1/2} \le \E[Hw]/\sqrt{\delta}
\Bigr\} \quad \text{for $\alpha\ge 0$}, 
$$ 
is decreasing in $\alpha\ge 0$ in the sense of $\mathcal{V}_\alpha \supset  \mathcal{V}_{\alpha'}$ for all $\alpha\le \alpha'$. 

The infinite-dimensional optimization problem \eqref{eq:const_optim_reg} can also be related to the nonlinear system of equations \eqref{eq:system_reg} as well. The key is again the existence of Lagrange multiplier; $(v_*,w_*)\in \mh$ solves the constrained optimization problem \eqref{eq:const_optim_reg} if and only if there exists $\mu_*\ge0$ such that
\begin{align}\label{eq:KKT_reg_main}
    -\mu_* \partial\mg (v_*, w_*)\cap \partial \ml(v_*, w_*)\ne \emptyset, \quad  \mu_* \mg(v_*, w_*)=0 \quad  \text{and} \quad \mg(v_*, w_*)\le 0.     
\end{align} 
Furthermore, we will argue that $\mu_*$ is always strictly positive so that the constraint is binding, i.e., $\mg(v_*, w_*)=0$ (see \Cref{lm:lagrange_reg}). This existence of the positive Lagrange multiplier $\mu_*>0$ and the binding condition $\mg(v_*, w_*)=0$ lead to an explicit construction of the solution to the nonlinear system of equations \eqref{eq:system_reg}. 
\begin{lemma}\label{lm:equivalence_system_reg}
    Let \Cref{as:reg} be fulfilled. Then
    \begin{enumerate}
        \item If $(v_*, w_*)$ is a solution to \eqref{eq:const_optim_reg} with $w_*\ne 0$, then we must have $\mt(v_*, w_*)>0$ and $\E[w_*^2]< \E[v_*G]$. For any Lagrange multiplier $\mu_*>0$ satisfying \eqref{eq:KKT_reg_main}, the pair of positive scalars $(\alpha_*,\beta_*,\kappa_*,\nu_*)$ defined by 
        \begin{align*}
          \alpha_* = \E[w_*^2]^{1/2}, \quad \beta_*= \frac{\mu_*}{\sqrt{\delta}},\quad \kappa_* = \frac{\delta}{\mu_*}\mt(v_*, w_*), 
          \quad \nu_* = \mu_* \frac{1- \E[w_*^2]^{-1/2}\E[v_* G]}{\mt(v_*, w_*)},
        \end{align*}
        solves  the nonlinear system of equations \eqref{eq:system_reg}. 
        \item If $(\alpha_*, \beta_*, \kappa_*, \nu_*)\in\R_{>0}^4$ is a solution to the nonlinear system of equations \eqref{eq:system_reg}, then $(v_*, w_*)\in\mh$ defined by
        \begin{align*}
            v_* = \prox[\kappa_*\loss](\alpha_* G+Z)-Z, \quad w_* = \prox[\nu_*^{-1} \reg](\nu_*^{-1}\beta_* H + X)-X
        \end{align*}
        solves the constrained optimization problems \eqref{eq:const_optim_reg} with $\|w_*\|=\alpha_{*}>0$ and Lagrange multiplier $\mu_{*}=\beta_* \sqrt{\delta}$.  
    \end{enumerate}
\end{lemma}
See \Cref{proof:lm:equivalence_system_reg} for the proof. 
\Cref{lm:equivalence_system_reg} implies that the nonlinear system of equations \eqref{eq:system_reg} admits a unique positive solution $(\alpha_*,\beta_*,\kappa_*,\nu_*)$ if and only if the infinite-dimensional optimization problem \eqref{eq:const_optim_reg}  admits a unique solution $(w_*, v_*)\in\mh$ with $w_*\ne 0$, and the Lagrange multiplier $\mu_*>0$ satisfying the KKT condition \eqref{eq:KKT_reg_main} is unique. 

The lemma below gives the uniqueness. 
\begin{lemma}[Uniqueness]\label{lm:unique_reg}
   Suppose $(v_*, w_*)$ solves the optimization problem $\min_{\mg(v,w)\le 0}\ml(v,w)$. Then, 
    \begin{enumerate}
        \item $w_{*}$ is unique. 
        \item If $w_*$ is non-degenerate $w_*\ne 0$, then $v_*$ and the Lagrange multiplier $\mu_*>0$ satisfying the KKT condition \eqref{eq:KKT_reg_main} are also unique. 
    \end{enumerate}
\end{lemma}
See \Cref{proof:lm:unique_reg} for the proof. 
From  \Cref{lm:unique_reg}-(1) and \Cref{lm:equivalence_potential_reg}, we see that the minimizer of $\min_{\alpha\ge 0}\mathsf{M}(\alpha)$ is unique if it exists. Furthermore, \Cref{lm:unique_reg}-(2) and \Cref{lm:equivalence_system_reg} imply that the solution to the nonlinear system of equations \eqref{eq:system_reg} is also unique if $w_*$ is always non-degenerate $w_*\ne 0$. 
The next lemma gives a sufficient condition for $w_*\ne 0$. 
\begin{lemma}[Non-degeneracy]\label{lm:nonzero_reg}
    Let \Cref{as:reg} be satisfied and let $\delta_{\mathsf{perfect}}$ be the threshold defined by \eqref{eq:threshold_reg}. Then, if $\delta < \delta_{\mathsf{perfect}}$, we must have $w_*\ne 0$. 
\end{lemma}
See \Cref{proof:lm:nonzero_reg} for the proof. \Cref{lm:nonzero_reg} implies that if $\delta$ is smaller than the threshold $\delta_{\mathsf{perfect}}$ 
then the degenerate case $w_*=0$ cannot occur. 

Finally, we claim the existence of the minimizer of the infinite-dimensional optimization problem $\min_{\mg(v,w)\le 0}\ml(v,w)$, and thereby conclude the proof of \Cref{th:main_reg}. The key to this existence is the coercivity of the objective function $\ml(v,w)$ over the constraint $\mg(v,w)\le 0$. 
\begin{lemma}[Coercivity]\label{lm:coercive_reg}
    If $\loss$ and $\reg$ are coercive, then the map $\ml:\mh\to\R$ is coercive when restricted to the subset of constraint $\mg(v,w)\le 0$, in the sense that 
      $$
      \ml(v,w) \ge C_1(\|v\|+\|w\|) - C_2 \quad \text{for all} \quad (v,w)\in \{(v,w)\in \mh: \mg(v,w)\le 0\},
      $$
      where $C_1$ and $C_2$ are some positive constants that only depend on $(\delta, \reg, \loss, \law(W), \law(X)).$
\end{lemma}
See \Cref{proof:lm:coercive_reg} for the proof. 
The existence of the minimizer $(v_*, w_*)$ immediately follows from \Cref{lm:coercive_reg} by the same argument developed right after \Cref{lm:coercive_noreg}. One can also get an explicit upper bound of $\alpha_*$ by calculating the constant $(C_1, C_2)$ explicitly (see \Cref{lm:coercive_reg_detail} and \Cref{subsec:upper_bound_alpha_reg}). 

\subsection{Numerical simulation}
In this section, we simulate the risk behavior of the regularized M-estimator
with the L1 loss L1 regularizer. Let $\hat{\bm{x}}_\lambda$ be the L1-penalized LAD estimator
$$
\hat{\bm{x}}_\lambda \in \argmin_{\bm{x}\in\R^p} \|\bm y-\bA \bm{x}\|_1  + \lambda \|\bm{x}\|_1
$$
with the regularization parameter $\lambda>0$. 
The marginal distribution of the noise vector and the signal are taken as 
$$
\bm{x}_0 \overset{\text{iid}}{\sim} 0.9 \delta_{0} + 0.1 N(0,1), \quad \bm{z} \overset{\text{iid}}{\sim} 0.7 \delta_{0} + 0.3 N(0,1)
$$
so that they have some point mass at the nondifferentiable point $0$ of the map $\R\ni x\mapsto |x|$. 
As we can see from \Cref{fig:reg_phase_transition}, the risk curve $\lambda\mapsto \|\hat{\bm{x}}_\lambda-\bm{x}_0\|^2$ is bounded from away from $0$ for $\delta>\delta_{\mathsf{perfect}}$, tangent for $\delta=\delta_{\mathsf{perfect}}$, and it takes $0$ in a certain region of $\lambda$ for $\delta > \delta_{\mathsf{perfect}}$.

Next, we plot the theoretical threshold $\delta_{\mathsf{perfect}}$ in \eqref{eq:threshold_reg} for different sparsity of noise $\PP(z_i = 0)$ and sparsity of signal $\PP(x_i =  0)$. We change the law of noise and signal as
$$
\bm{x}_{0}\overset{\text{iid}}{\sim} (1-s)\delta_{0} + s N(0,1), \quad \bm{z} \overset{\text{iid}}{\sim} (1-t) \delta_{0} + tN(0,1), \quad s, t \in (0,1),
$$
where $s=\PP(x_{0j}\ne 0)$ and $t=\PP(z_{i}\ne 0)$ are the probability that signal and noise take nonzero value, respectively. For each $t\in[0.2, 0.3, 0.5, 0.7, 1.0]$, we plot the inverse of threshold $1/\delta_{\mathsf{perfect}}$ as a function of $s$ in \Cref{fig:reg_threshold}. We observe that the curve $s\mapsto \delta_{\mathsf{perfect}}^{-1}$ shifts down as $t$ increases, which means that as the signal becomes dense, the oversampling ratio $\delta=n/p$ needed for perfect recovery increases. 
\begin{figure}
        \includegraphics[width=0.5\textwidth]{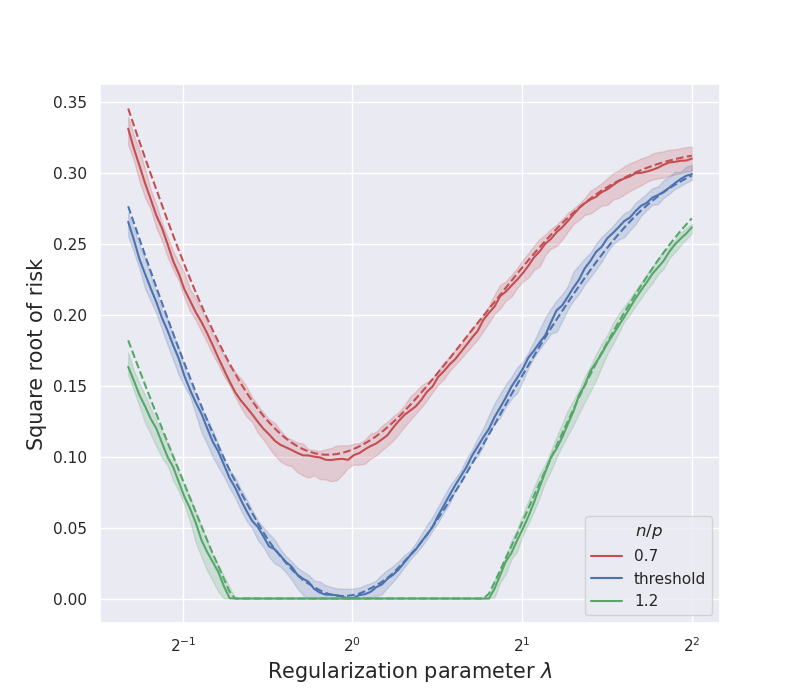}        
    \caption{Risk behavior of the regularized M-estimator computed by L1 loss and L1 regularizer for different ratio $n/p$ with $p$ fixed to $1200$. }
    \label{fig:reg_phase_transition}
\end{figure}

\begin{figure}
    \includegraphics[width=0.5\linewidth
    ]{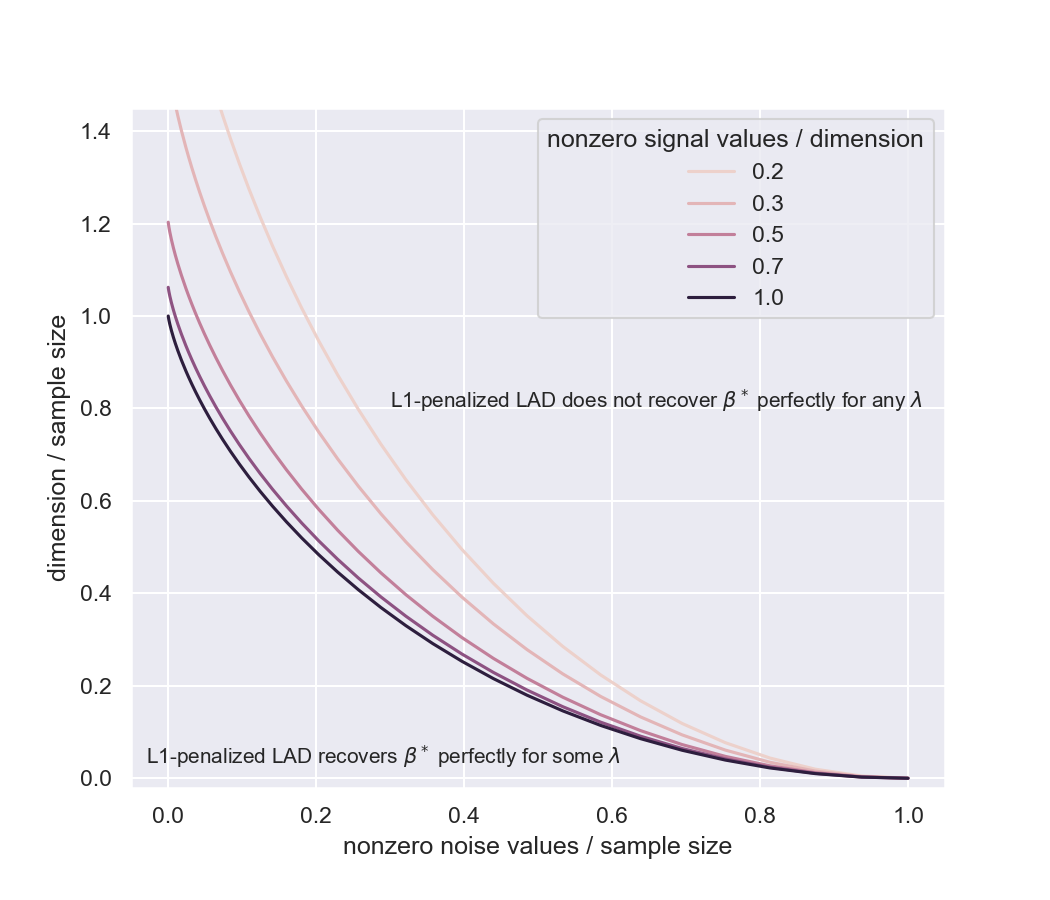}
    \caption{The theoretical threshold $1/\delta_{\mathsf{perfect}}$ for perfect recovery.}
    \label{fig:reg_threshold}
\end{figure}

\bibliographystyle{plainnat}
\bibliography{reference}

\appendix

\section{Proof for \Cref{sec:noreg}}
\subsection{Set up for infinite-dimensional optimization problem}
\begin{lemma}\label{lm:lg_basic_noreg}
    Let $
    \mh=\{v:\R^2\to\R: \E[v(G,Z)^2]<+\infty\}
    $  be the almost sure equivalent classes 
of squared integrable measurable functions of $(G, Z)$ equipped with the usual $L^2$ scalar product $\langle u,v\rangle_{\mathcal H}= \E[u(G,Z)v(G,Z)]$. 
    Suppose $\loss:\R\to\R$ is convex and Lipschitz, and $\delta (=\lim n/p)$ satisfies $\delta>1$. 
   Then
    the maps $\ml$ and $\mg$ defined by 
    \begin{align*}
        \ml: \mh \to \R, \quad &v\mapsto \E[\loss(v+Z)-\loss(Z)]\\
        \mg: \mh\to \R, \quad &v\mapsto \|v\|-\E[vG]/\sqrt{1-\delta^{-1}}
    \end{align*}
    are both convex, Lipschitz, and finite valued. 
    Furthermore, 
    $\mg$ is Fr\'echet differentiable at $\mh \setminus \{0\}$ in the sense that 
    $$
    \mg(v+h) = \mg(v) + \E[\nabla\mg(v) h] + o(\|h\|) \quad \text{for all} \quad \|v\|>0, 
    $$ 
    where the gradient $\nabla \mg:\mh\to\mh$ is  given by 
    $$
    \nabla \mg: v\mapsto {v}/{\|v\|} - {G}/{\sqrt{1-\delta^{-1}}}.
    $$
\end{lemma}
\begin{proof}
    $\ml$ is convex since $\loss$ is convex, while $\mg$ is convex since the L2 norm $\|v\|=\E[v^2]^{1/2}$ is convex and $\E[vZ]/\sqrt{1-\delta^{-1}}$ is linear in $v$. For all $v, \tilde{v} \in \mh$, Jensen's inequality, the triangle inequality, and Cauchy--Schwarz inequality yield the inequalities 
\begin{align*}
    |\ml(v)-\ml(\tilde{v})| \le  \|\loss\|_{\lip} \|v-\tilde{v}\| \quad \text{and}\quad 
    |\mg(v)-\mg(v')| \le \|v-v'\| + \|v-v'\|/\sqrt{1-\delta^{-1}}, 
\end{align*}
so combined with $\ml(0)=\mg(0)=0$, the above display implies that the maps $\mg$ and $\ml$ are finite valued and Lipschitz. The Fr\'echet differentiability of $\mg$ over the set $\mh/\{0\}$ immediately follows from the first order approximation
\( \|v+h\|=\|v\|+\E[vh]/\|v\| + o(\|h\|) \) for any $v\ne 0$. Note that this approximation follows from the inequality $|\sqrt{1+x}-(1+x/2)| \le x^2/8$ for all \( x\ge 0 \)
applied with \( x=(2\E[vh]+\|h\|^2)/\|v\|^2 \).
\end{proof}

\begin{lemma}\label{lm:lagrange_noreg}
    Suppose that $\loss$ is convex and Lipschitz. Then, $v_*\in\mh$ is a solution to the constrained optimization problem 
    $$\min_{v\in \mh} \ml(v) \quad 
    \text{subject to} \quad \mg(v)\le 0, 
    $$
    if and only if there exists a nonnegative Lagrange multiplier $\mu_*\ge 0$ such that the KKT condition
    \begin{align}\label{eq:KKT_noreg}
    -\mu_* \partial \mg(v_*) \cap \partial \ml(v_*) \ne \emptyset, \quad \mg(v_*)\le 0, \quad \mu_*\mg(v_*) = 0
    \end{align}
    is satisfied. Furthermore, if $Z$ and $\loss$ satisfy $\PP(Z\ne 0) > 0$ and $\{0\} = \argmin_x \loss(x)$, then the Lagrange multiplier $\mu_*$ must be strictly positive so that the constraint is binding, i.e., $\mg(v_*) = 0$. 
\end{lemma}
\begin{proof}
Note that the relative interior of the set of constraint $\{v\in\mh: \mg(v)\le 0\}$ is not empty since $\mg(G)= 1 - (1-\delta^{-1})^{-1/2} < 0$. Thus, the if and only if part follows from \Cref{lm:Lagrange}. 

It remains to show the positiveness $\mu_*>0$. We proceed by contradiction. If  $\mu_* = 0$, then the minimizer $v_*$ must satisfy $v_*\in \argmin_{v\in \mh} \ml(v)$ and $\mg(v_*) \le 0$.
Let us take $v_n=-Z I\{\loss(Z)\le n\}$. Note that $v_n$ belongs to the Hilbert space $\mh$, i.e., $\E[v_n^2]^{1/2}<+\infty$. Indeed, the assumption $\{0\}=\argmin_x\loss(x)$ implies $\loss$ is coercive in the sense that $\loss(x)\ge a |x| - b$ for some constant $a>0$ and $b\ge 0$ (see \Cref{lm:convex_coercive}). This gives $|v_n|\le (n+b)/a$ with probability $1$ so in particular $\E[|v_n|^2]^{1/2} <+\infty$. Now, the objective function $\ml(v)=\E[\loss(v+Z)-\loss(Z)]$ evaluated at $v=v_n$ can be written as 
\begin{align*}
    \ml(v_n) =- \E[I\{\loss(Z)\le n\} (\loss(Z)-\loss(0))]. 
\end{align*}
Now we consider the two cases: (1) $\E[\loss(Z)]=+\infty$ and (2) $\E[\loss(Z)]<+\infty$. 
If $\E[\loss(Z)]=+\infty$, the monotone convergence yields $\ml(v_n) \to  - \infty$ as $n\to\infty$, a contradiction since by the optimality condition for $v_*$ we must have $\inf_n \ml(v_n)\ge \ml(v_*) >-\infty$. 
If $\E[\loss(Z)]<+\infty$, then again the monotone convergence theorem gives $\lim_{n\to\infty} \ml(v_n) = \E[\loss(0)-\loss(Z)]$.  Since $v_*$ minimizes $\ml(v)$, we must have $\ml(v_*) = \E[\loss(v_*+Z)-\loss(Z)] \le \E[\loss(0)-\loss(Z)]$, or equivalently, 
$$
\E
[\loss(v_*+Z)-\loss(0)]\le 0.
$$
Since $\loss(v_*+Z)-\loss(0)$ is always non-negative by $\argmin_x\loss(x)=\{0\}$, the above inequality implies  $\loss(v_*+Z)= \loss(0)$ and $v_* = -Z$. 
Substituting $v_*=-Z$ to the constraint $\mg(v_*)\le 0$ and using the independence of $Z$ and $G$, we are left with
$$
0 \ge \mg(v_*) = \E[v_*^2]^{1/2}-\E[v_*G]/\sqrt{1-\delta^{-1}} = \E[Z^2]^{1/2},
$$
and hence $\PP(Z=0)=1$. This is a contradiction with $\PP(Z\neq 0)>0$. Therefore, the Lagrange multiplier $\mu_*$ must be strictly positive. 
\end{proof}

\subsection{Proof of \Cref{lm:equivalence_potential_noreg}}\label{proof:lm:equivalence_potential_noreg}
For each $\alpha\ge 0$, let $\mathsf{M}(\alpha)$ be the potential in \eqref{eq:potential_noreg}. We redefine it here for convenience:
\begin{align*}
    \forall\alpha\ge 0, \quad 
 \mathsf{M}(\alpha) = \sup_{\beta >  0} \inf_{\tau_g>0}  \mathsf{D}(\alpha, \beta, \tau_g),  \quad \mathsf{D}(\alpha, \beta, \tau_g) := 
    \frac{\beta\tau_g}{2\delta } + \mathsf{L}(\alpha, \frac{\tau_g}{\beta}) - \frac{\alpha\beta}{\delta}
\end{align*}
where $\mathsf{L}(c, \tau) = \E[\env_\loss(cG+Z;\tau)-\loss(Z)]$. We also define 
{
$\bar{\mathsf{M}}(\alpha)$ for $\alpha \ge 0$ as
$$
\bar{\mathsf{M}}(0)=0,
\qquad
\text{and for }\alpha>0,
\qquad
\bar{\mathsf{M}}(\alpha)=\sup_{b>0}
\Bigl(\mathsf{L}(\alpha, \frac{\alpha}{b}) - \frac{\alpha b}{2\delta}\Bigr).
$$
}
Next, we define $\check{\mathsf{M}}(\alpha)$ as the optimal value of the following constrained optimization problem over the Hilbert space $\mathcal{H}$:
$$
\forall \alpha\ge0, \quad 
\check{\mathsf{M}}(\alpha) := \Bigl(
\min_{v\in \mh} \ml(v) \quad \text{subject to}\quad \|v-\alpha G\|\le \alpha/\sqrt{\delta} \Bigr)
$$
with $\ml(v)=\E[\loss(v+Z)-\loss(Z)]$ defined in  \Cref{lm:lg_basic_noreg}. The above optimization problem admits a minimizer $v_\alpha\in\mh$ for all $\alpha\ge 0$ since the set of constraint $\{v\in \mh:\|v-\alpha G\|\le \alpha/\sqrt{\delta}\}$ is convex, closed, and bounded.
\begin{lemma}\label{lm:N_check_N}
    Let \Cref{as:noreg} be fulfilled. Then, the following holds:
    \begin{enumerate}
        \item $\mathsf{M}(0)=0$, i.e., 
        $$
        \mathsf{M}(0) = \sup_{\beta > 0}
            \inf_{\tau_g>0} \frac{\beta\tau_g}{2\delta } + \mathsf{L}(0, \frac{\tau_g}{\beta}) =  0.
        $$
        \item $\mathsf{M}(\alpha) = {\bar{\mathsf{M}}(\alpha)} = \check{\mathsf{M}}(\alpha)$ for all $\alpha \ge 0$. 
    \end{enumerate}
\end{lemma}

\begin{proof}
Using $\lim_{\tau \to 0+} \mathsf{L}(0, \tau) = 0$ (see \cite[Lemma 4.1]{thrampoulidis2018precise}), we have 
\begin{align*}
    \mathsf{M}(0) \le \sup_{\beta > 0} 
        \lim_{\tau_g\to 0+}  \frac{\beta\tau_g}{2\delta } + \mathsf{L}(0, \frac{\tau_g}{\beta})  = 0. 
\end{align*}
On the other hand, for all $\beta>0$ and $\tau_g>0$, the derivative of objective function with respect to $\tau_g$ is bounded from below as 
$$
\frac{\partial}{\partial \tau_g} \Bigl(
    \frac{\beta\tau_g}{2\delta} +\mathsf{L}(0, \frac{\tau_g}{\beta})
\Bigr) = 
\frac{\beta}{2\delta} - \frac{1}{2\beta}\E[\env_\loss'(0; \frac{\tau_g}{\beta})^2]\ge \frac{\beta}{2\delta } - \frac{1}{2\beta}\|\loss\|_{\lip}^2 = \frac{\beta^2 - \delta \|\loss\|_{\lip}^2}{2\beta\delta}
$$
thanks to $\env_\loss'(x, \tau) \in \partial \loss(\prox[\tau\loss](x))$. This lower bound implies that the map $\tau_g\mapsto 
    \frac{\beta_0\tau_g}{2\delta} +  \mathsf{L}(0, \frac{\tau_g}{\beta_0})
$ over $\R_{>0}$ is strictly increasing when $\beta=\beta_0 = \sqrt{2\delta} \|\loss\|_{\lip}>0$. Therefore, $\mathsf{M}(0)$ is bounded from below as 
\begin{align*}
    \mathsf{M}(0) \ge \inf_{\tau_g>0} \Bigl(
    \frac{\beta_0\tau_g}{2\delta} + \mathsf{L}(0, \frac{\tau_g}{\beta_0})\Bigr) = \lim_{\tau_g\to 0+} \Bigl(
        \frac{\beta_0\tau_g}{2\delta} + \mathsf{L}(0, \frac{\tau_g}{\beta_0})\Bigr) = 0 + \lim_{\tau\to 0+} \mathsf{L}(0, \tau) = 0.
\end{align*}
This gives the reverse inequality $\mathsf{M}(0)\ge 0$ and finishes the proof of $(1)$. 

Next, we show the equality $\mathsf{M}(\alpha)= {\bar{\mathsf{M}}(\alpha)} =\check{\mathsf{M}}(\alpha)$ in $(2)$. {Note that this equality is obvious for $\alpha=0$ since we have shown $\mathsf{M}(0)=0$ in the previous paragraph  and $\bar{\mathsf{M}}(0)=0$ by definition. As for $\check{\mathsf{M}}$, the constraint $\|v-\alpha G\| \le \alpha/\sqrt{\delta}$ with $\alpha=0$ gives $v=0$ and $\ml(0)= \E[\loss(0+Z)-\loss(Z)]=0$ so that $\check{\mathsf{M}}(0)=0$.} Below, we will prove the equality for $\alpha>0$ by showing the three inequalities:
$$
\bar{\mathsf{M}}(\alpha) \le \mathsf{M}(\alpha) \le \check{\mathsf{M}}(\alpha) \le \bar{\mathsf{M}}(\alpha), \quad \forall \alpha >0
$$
Note that the first inequality $\bar{\mathsf{M}}(\alpha) \le \mathsf{M}(\alpha)$ immediately follows by taking $\tau_g=\alpha$ for $\inf_{\tau_g}$ in the objective function $\mathsf{D}$ of $\mathsf{M}$. 

Proof of $\mathsf{M}(\alpha) \le \check{\mathsf{M}} (\alpha)$. 
For all $\beta >  0$ and $\tau_g>0$, if $v\in\mh$ satisfies the constraint $\|v-\alpha G\|\le \alpha/\sqrt{\delta}$, the objective function $\mathsf{D}(\alpha, \beta, \tau_g)$ in $\mathsf{M}$ is bounded from above by 
   \begin{align*}
    \mathsf{D}(\alpha, \beta, \tau_g)
    &= \frac{\beta\tau_g}{2\delta} + \E\Bigl[\min_{x\in\R} \Bigl(\frac{\beta}{2\tau_g}(x-\alpha G)^2 + \loss(x+Z)-\loss(Z)\Bigr)\Bigr] - \frac{\alpha\beta}{\delta}\\
    &\le \beta\Bigl(\frac{\tau_g}{2\delta } + \frac{\|v-\alpha G\|^2}{2 \tau_g}\Bigr) +  \E\Bigl[\loss(v+Z)-\loss(Z)\Bigr] - \frac{\alpha\beta}{\delta}\\
    &\le \frac{\beta}{\delta} \Bigl(\frac{\tau_g}{2} + \frac{\alpha^2}{2\tau_g}-\alpha \Bigr) +  \E\Bigl[\loss(v+Z)-\loss(Z)\Bigr]  && \text{by } \|v-\alpha G\|\le \frac{\alpha}{\sqrt{\delta}}.
\end{align*}
Using the identity $\inf_{\tau_g>0} \frac{\tau_g}{2} + \frac{\alpha^2}{2\tau_g}=\alpha$ for the last line, we obtain $\sup_{\beta > 0} \inf_{\tau_g>0}\mathsf{D}(\alpha, \beta, \tau_g) \le \E[\loss(v+Z)-\loss(Z)]$ for all $v\in\mh$ satisfying the constraint $\|v-\alpha G\|\le \alpha/\sqrt{\delta}$. This means $\mathsf{M}(\alpha) \le \check{\mathsf{M}}(\alpha)$.

Proof of $\check{\mathsf{M}}(\alpha)\le \bar{\mathsf{M}}(\alpha)$. 
{
The two optimization problems
$$\min_{v\in\mh: \|v-\alpha G\|\le \alpha/\sqrt{\delta}} \ml(v),
\qquad
\min_{v\in\mh: \frac12\|v-\alpha G\|^2/\alpha\le \frac12 \alpha/\delta} \ml(v)$$
are the same, with the constraint being convex in both cases,
and $v_\alpha$ minimizing the former also minimizes the latter.
Let us work with the latter problem where the norm is squared.
By the same argument in the proof of \Cref{lm:lagrange_noreg}, there exists an associated Lagrange multiplier {$b_* >0$} such that the minimizer $v_\alpha$ of the constrained optimization problem also solves the unconstrained optimization problem as in 
$$v_\alpha \in \argmin_{v\in \mh} \ml(v) + b_* \Bigl(\frac{\|v-\alpha G\|^2}{2\alpha}-\frac{\alpha}{2\delta}\Bigr)$$
and 
the constraint is binding, that is, $\|v_\alpha-\alpha G\| = \alpha/\sqrt{\delta}$. 
We thus have 
$$\check{\mathsf{M}}(\alpha) = \ml(v_\alpha)
=
\min_{v\in \mh} 
\Bigl\{
\ml(v) + b_* \Bigl(\frac{\|v-\alpha G\|^2}{2\alpha}-\frac{\alpha}{2\delta}\Bigr) \Bigr\}
\le
\mathsf{L}(\alpha, \frac{\alpha}{b_*}) - \frac{b_*\alpha}{2\delta}
$$
by choosing
$v=\prox[\tfrac{\alpha}{b_*} \loss](\alpha G+Z)-Z \in \mh$
for the rightmost upper bound.
Note that $v_*\in \mh$ follows from the inequality $|v_*| \le |\alpha G| + \frac{\alpha}{b_*}\|\loss\|_{\lip}$ and $b_*>0$. Taking $\sup_{b > 0}$ of the right-hand side, we obtain $\check{\mathsf{M}}(\alpha)\le \bar{\mathsf{M}}(\alpha)$.
}
\end{proof}

\begin{lemma}
    $\mathsf{M}(\alpha)$ is minimized at $\alpha_*$ if and only if $\min_{v\in \mh: \mg(v)\le 0}\ml(v)$ admits a minimizer $v_*$ with $\|v_*\|=\alpha_* \sqrt{1-\delta^{-1}}$. 
\end{lemma}
\begin{proof}
    The Hilbert space $\mh=\{v:\R^2\to\R, \ \E[v(G, Z)^2]<+\infty\}$ is one-to-one to the product sets $\R\times \mh^\perp$ where $\mh^\perp =\{v\in\mh: \E[vG]=0\}$ is the orthogonal complement of $G$: the bijection map is given by 
    $$
    \mh\ni v\mapsto (\E[vG], \proj(v))\in \R\times \mh^\perp, \quad \R\times \mh^\perp \ni (t, u)\mapsto tG+u\in \mh
    $$
    where $\proj(v)=v-\E[vG] G$ is the projection onto the orthogonal complement $\mh^\perp$. Using the Pythagorean theorem
    $\E[v^2] = \E[vG]^2 + \E[\proj(v)^2]$, the constraint $\|v-\alpha G\|\le \alpha/\sqrt{\delta}$ is equivalent to
    $$
     \|\proj(v)\|^2 \le - (1-\delta^{-1})\alpha^2 + 2\alpha \E[vG] - \E[vG]^2. 
    $$
    Using the identity $\mathsf{M}(\alpha)= \check{\mathsf{M}}(\alpha) := \min_{v\in\mh: \|v-\alpha G\| \le \alpha/\sqrt{\delta}} \ml(v)$ by \Cref{lm:N_check_N}-(2), it holds that
    \begin{align*}
        \min_{\alpha\ge 0} \mathsf{M}(\alpha) 
        &= \min_{\alpha\ge 0} \min_{t\in\R} \min_{u\in \mh^\perp: \|u\|^2 \le - (1-\delta^{-1})\alpha^2 + 2\alpha t - t^2} \ml(tG + u)\\
        &= \min_{t\in\R} \min_{\alpha\ge 0} \min_{u\in \mh^\perp: \|u\|^2 \le - (1-\delta^{-1})\alpha^2 + 2\alpha t - t^2} \ml(tG + u)
    \end{align*}
    Since the upper bound $- (1-\delta^{-1})\alpha^2 + 2\alpha t - t^2$ of the norm of $u\in \mh^\perp $ is maximized at $\alpha = t/(1-\delta^{-1})$ for each $t\in\R$, we obtain
    \begin{align*}
        \min_{\alpha\ge 0} \mathsf{M}(\alpha) &= \min_{t\in\R} \min_{u\in \mh^\perp: \|u\|^2 + t^2 \le t^2/(1-\delta^{-1})} \ml(tG + u)
        = \min_{v\in \mh: \|v\|^2 \le \E[vG]^2/(1-\delta^{-1})} \ml(v)
    \end{align*}
    so that $\min_{\alpha \ge 0} \mathsf{M}(\alpha)$ is minimized at $\alpha=\alpha_*$ if and only if $\min_{v\in\mh: \mg(v)\le 0}\ml(v)$ is minimized at $v=v_*$ with $\alpha_* = \E[v_* G]/(1-\delta^{-1})$. Since the constraint $\mg(v)\le 0$ is binding (see \Cref{lm:lagrange_noreg}), $\mg(v_*)=\|v_*\| - \E[v_*G]/\sqrt{1-\delta^{-1}}=0$ always holds and this gives $\|v_*\|= \alpha_*\sqrt{1-\delta^{-1}}$. 
\end{proof}

\subsection{Proof of \Cref{lm:equivalence_system_noreg}}\label{proof:lm:equivalence_system_noreg}

\begin{lemma}\label{lm:optimization_to_system_noreg}
    Suppose $v_*\in \mh$ solves $\min_{v\in \mh: \mg(v)\le 0} \ml(v)$ with $v_*\ne 0$. Let us take a Lagrange multiplier $\mu_*>0$ satisfying the KKT condition in \Cref{lm:lagrange_noreg}. Define the positive scalars $(\alpha_*, \kappa_*)\in\R_{>0}^2$ by 
    \begin{equation*}
        \alpha_* =  \|v_*\|/\sqrt{1-\delta^{-1}}, \qquad 
    \kappa_* = \|v_*\| / \mu_*. 
    \end{equation*}
    Then $v_*$ takes the form of 
    $$
    v_* = \prox[\kappa_*\loss](\alpha_* G + Z)-Z
    $$
    and 
    $(\alpha_*, \kappa_*)$ solves the nonlinear system of equations:
    \begin{align}
        \alpha^2  &= \delta \E\bigl[
        ((\alpha G + Z) - \prox[\kappa \loss] (\alpha G+ Z))^2\bigr], \label{eq:system_noreg_1}\\
        \alpha &= \delta \E \bigl[((\alpha G + Z) - \prox[\kappa \loss](\alpha G+ Z))G\bigr]. \label{eq:system_noreg_2}
\end{align}
\end{lemma}
\begin{proof}
Since $\mg$ is Fr\'echet differentiable at $v=v_*\ne 0$ (see \Cref{lm:lg_basic_noreg}) and the constraint $\mg(v)\le 0$ is binding at $v=v_*$ (see \Cref{lm:lagrange_noreg})
 we have 
    $$
    - \mu_* \nabla \mg(v_*) \in \partial \ml(v_*) \quad  \text{and} \quad \mg(v_*) = 0 \quad \text{where} \quad  \nabla\mg(v_*) = {v_*}/{\|v_*\|} - {G}/{\sqrt{1-\delta^{-1}}} 
    $$
    By the definition of $(\alpha_*, \kappa_*)$, the condition $- \mu_* \nabla \mg(v_*) \in \partial \ml(v_*)$ reads to 
    $$
    \partial \ml(v_*) \ni -\mu_* \nabla \mg(v_*) = \kappa_*^{-1} (v_* - \alpha_* G). 
    $$
    This means that $v_*$ also minimizes the convex function:
    $$
    \mh \ni v \mapsto 
        \ml(v) + \E\Bigl[\frac{(v-\alpha_*G)^2}{2\kappa_*}\Bigr] = \E\Bigl[\loss(v+Z)-\loss(Z) + \frac{(v-\alpha_*G)^2}{2\kappa_*}\Bigr].
    $$
    Since the minimization can be performed inside the expectation, we have 
    $$
    v_* = \prox[\kappa_*\loss](\alpha_* G + Z) - Z \in \mh. 
    $$
    Note that the definition $\alpha_* = \|v_*\|_2/ \sqrt{1-\delta^{-1}} 
    $ and the binding condition $\mg(v_*)=\|v_*\|_2- \E[v_* G]/\sqrt{1-\delta^{-1}} = 0$ yield
    $$
    \E[v_* G] = \sqrt{1-\delta^{-1}} \|v_*\|_2 =  \alpha_*(1-\delta^{-1}), \quad \E[v_*^2] = \alpha_*^2(1-\delta^{-1}).
    $$ 
    Thus, we have
    \begin{align*}
    \E[(\alpha_*G-v_*)G] &= \alpha_* -\E[v_* G] = \alpha_* -(1-\delta^{-1})\alpha_* = \alpha_* \delta^{-1} \\
    \E[(\alpha_* G- v_*)^2] 
    &=
    \|v_*\|_2^2 + \alpha_*^2 - 2\alpha_* \E[v_*G]
    =
    \alpha_*^2 (1-\delta^{-1}) + \alpha_*^2 - 2\alpha_*^2 (1-\delta^{-1})
    = \alpha_*^2 \delta^{-1}.
    \end{align*}
    Since $v_*$ $= \prox[\kappa_*\loss](\alpha_* G + Z) - Z$, the above display means that $(\alpha_*, \kappa_*)$  satisfies \eqref{eq:system_noreg_1}-\eqref{eq:system_noreg_2}. This completes the proof. 
\end{proof}

\begin{lemma}\label{lm:system_to_optimization_noreg} If $(\alpha_*, \kappa_*)\in \R_{>0}^2$ solves the nonlinear system of equations \eqref{eq:system_noreg_1}-\eqref{eq:system_noreg_2}, then $v_*\in\mh$ defined by 
    $$
    v_* = \prox[\kappa_*\loss](\alpha_* G + Z)-\loss(Z)
    $$
    satisfies $\|v_*\|= \alpha_*\sqrt{1-\delta^{-1}}>0$ and 
     the KKT condition \eqref{eq:KKT_noreg} with the Lagrange multiplier $\mu_* =\alpha_* \sqrt{1-\delta^{-1}}/\kappa_*>0$. Thus, $v_*$ solves $\min_{v\in\mh:\mg(v)\le 0}\ml(v)$ with $v_*\ne 0$.  
\end{lemma}
\begin{proof}
By the definition of $v_*$, the nonlinear system of equations \eqref{eq:system_noreg_1}-\eqref{eq:system_noreg_2} reads to 
$$
\E[(\alpha_* G - v_*)^ 2] = \alpha_*^2 \delta^{-1}, \quad \E[(\alpha_* G - v_*)G] = \alpha_* \delta^{-1}.
$$
Expanding the square inside the expectation, we are left with 
$$
\E[v_* G] = \alpha_* (1-\delta^{-1}), \quad
\E[v_*^2] = (1-\delta^{-1}) \alpha_*^2>0.
$$
Thus, the constraint is binding $\mg(v_*) = 0$ and $v_*\ne 0$. Using the differentiability of the map $\mh\ni v \mapsto \mg(v)$ at $v\ne 0$ (see \Cref{lm:lg_basic_noreg}), it holds that 
$$
-\frac{\alpha_* \sqrt{1-\delta^{-1}}}{\kappa_*} \cdot \nabla\mg(v_*) = -\frac{\|v_*\|}{\kappa_*} \bigl(
\frac{v_*}{\|v_*\|} - \frac{G}{\sqrt{1-\delta^{-1}}} 
\bigr) = \kappa_{*}^{-1} (\alpha_* G - v_*).
$$
Noting 
$\kappa_*^{-1}  (\alpha_* G-v_*) \in \partial \loss(v_*+Z)$ and $\partial \loss(v_*+Z)\subset \partial \ml(v_*)$ by the definition of $v_*$ and $\ml(v)=\E[\loss(v+Z)-\loss(Z)]$, the above display gives $-\mu_*\nabla\mg(v_*)\in \partial\ml(v_*)$ with $\mu_*=\alpha_*\sqrt{1-\delta^{-1}}/\kappa_*>0$. This means that the KKT condition \eqref{eq:KKT_noreg} is satisfied with this $\mu_*>0$ and $v_*$ solves the constrained optimization problem $\min_{\mg(v)\le 0}\ml(v)$.  
\end{proof}

\subsection{Proof of \Cref{lm:unique_noreg}}\label{proof:lm:unique_noreg}

\begin{lemma}\label{lm:propotional_noreg}
    If $v_{*}$ and $v_{**}$ solve $\min_{\mg(v)\le 0}\ml(v)$, then there exists a positive scalar $t>0$ such that $v_{*}=tv_{**}$.
\end{lemma}

\begin{proof}
If $v_*=v_{**}=0$ then the claim is obvious, so we consider the case where at least $v_*$ or $v_{**}$ is nonzero and assume $v_*\ne 0$ without loss of generality. Since $v_*$ is nonzero, \Cref{lm:optimization_to_system_noreg} implies that $v_*$ takes the form of $v_*=\prox[\kappa_*\loss](\alpha_* G + Z)-Z$ where $(\alpha_*, \kappa_*)\in\R_{>0}^2$ is the associated solution to the nonlinear system satisfying $\E[(v_*-\alpha_* G)^2]=\alpha_*\delta^{-1}$ and $\E[v_*^2]=\alpha_*^2 (1-\delta^{-1})$. 
By the definition of proximal operator, it holds that $\kappa_*^{-1}(\alpha_* G - v_*)\in \partial\loss(v_*+Z)$, which means that 
pointwise, the scalar realization of
$v_*$ also minimizes the convex function 
$$
x \mapsto  \loss(x+Z) + (2\kappa_*)^{-1}((x-\alpha_* G)^2 - (x-v_*)^2),
$$
since the derivative of $x\mapsto (x-v_*)^2$ is $0$ at $v_*$. Thus,
again pointwise, we have 
$$
\loss(v+Z) + (2\kappa_*)^{-1}\bigl[
    (v-\alpha_* G)^2 - (v-v_*)^2
\bigr] \ge \loss(v_*+Z) + (2\kappa_*)^{-1}(v_*-\alpha_* G)^2,
$$
for all $v\in\mh$.
Taking expectations,
$\ml(v) - \ml(v_*)=\E[\loss(v+Z)-\loss(v_*+Z)]$ is bounded from below as 
$$
\ml(v) - \ml(v_*)
\ge 
(2\kappa_*)^{-1} \E[(v_*-\alpha_* G)^2 + (v-v_*)^2 - (v-\alpha_* G)^2].
$$
Noting that $\E[v^2]$ cancel out,
expanding the square and using $\E[(v_*-\alpha_* G)^2]=\alpha_*\delta^{-1}$ and $\E[v_*^2]=\alpha_*^2 (1-\delta^{-1})$, we are left with 
\begin{align*}
    \ml(v)-\ml(v_*)\ge \kappa_*^{-1} \bigl(-\E[v_* v] + \alpha_* \E[v G]\bigr). 
\end{align*}
Applying this inequality with $v=v_{**}$, noting that $v_*$ and $v_{**}$ are minimizer of  $\min_{\mg(v)\le 0}\ml(v)$ so that $\ml(v_{**})=\ml(v_*)$, 
\begin{align*}
    0 & \ge \kappa_*^{-1} \bigl(-\E[v_* v_{**}] + \alpha_* \E[v_{**} G]\bigr) \\
    &\ge \kappa_*^{-1}\bigl(
        - \|v_*\| \|v_{**}\|
        +\alpha_* \E[v_{**} G]\bigr)
          &&\text{by the Cauchy--Schwarz for $\E[v_*v_{**}]\le \|v_*\|\|v_{**}\|$}\\
    &\ge  \kappa_*^{-1}\bigl(
   -\|v_*\| \|v_{**}\|
      + \alpha_* \|v_{**}\|\sqrt{1-\delta^{-1}}\bigr)
      &&\text{by $\mg(v_{**}) = \|v_{**}\| - \E[v_{**}G]/\sqrt{1-\delta^{-1}}\le 0$,}\\
      &=0
      &&\text{by $\|v_*\| =\alpha_*\sqrt{1-\delta^{-1}}$}
\end{align*}
Thus, the inequality $\E[v_{*}v_{**}]\le \|v_*\|\|v_{**}\|$ must hold with equality. 
The equality case gives $v_{*}= t v_{**}$ for some $t\ge 0$, where $t$ must be strictly positive as $v_*\ne 0$. This finishes the proof. 
\end{proof}

\begin{lemma}[Uniqueness of solution $v_*$]\label{lm:unique_v_*}
    If $v_{*}$ and $v_{**}$ solve $\min_{\mg(v)\le 0}\mg(v)$ then we must have $v_{*}=v_{**}$. 
\end{lemma}
\begin{proof}
    We proceed by contradiction. 
    Suppose that there exist some distinct minimizers $v_*\ne v_{**}$. We may assume $v_*$ and $v_{**}$ are nonzero. Indeed, if one of them is $0$, say $v_{*}=0$, then the other one $v_{**}$ must be nonzero, and by the convexity of the objective function $\ml$ and the constraint set $\{\mg(v)\le 0\}$, the intermediate point $v_{***}=(v_* + v_{**})/2=v_{**}/2$ is also a minimizer with $v_{***}\ne 0$ and $v_{***}\ne v_{*}$.

    Since $v_*$ and $v_{**}$ are nonzero, \Cref{lm:optimization_to_system_noreg} implies that $v_*$ and $v_{**}$ take the form 
    $$
    v_* = \prox[\kappa_*\loss](\alpha_*G +Z)-Z, \quad v_{**} = \prox[\kappa_{**}\loss](\alpha_{**}G +Z)-Z
    $$
    where $(\alpha_*, \kappa_*)$ and $(\alpha_{**}, \kappa_{**})$ are associated positive scalars satisfying \eqref{eq:system_noreg_1}-\eqref{eq:system_noreg_2}. 
     By \Cref{lm:propotional_noreg}, $v_{*}$ and $v_{**}$ are proportional in the sense that $v_{**}=t v_*$ for a positive scalar $t$. 
    Since $\E[(v_{**})^2]=\alpha_{**}^2 (1-\delta^{-1})$ and $\E[v_*^2]=\alpha_*^2 (1-\delta^{-1})$,
    we must have $\alpha_{**}=t \alpha_*$ and 
    \begin{equation}
        \label{eq:prox_rewritten}
        \prox[\kappa_{**} \loss] (t\alpha_{*} G + Z) - Z = t\prox[\kappa_*\loss](\alpha_* G + Z) - tZ, 
    \end{equation}
    or equivalently
    \begin{equation}
        \label{eq:partial_prox_identity}
        \prox[\kappa_{**}\loss](t\alpha_{*} G + Z)-(t\alpha_{*} G + Z) = t \bigl\{\prox[\kappa_{*}\loss](\alpha_* G + Z)- (Z+\alpha_* G) \bigr\}.
    \end{equation}
    Since $f(x)=\prox[\kappa_*\loss](x)$ is firmly non-expansive
    and $f(0)=0$ by $\{0\} = \argmin_x\loss(x)$, we have 
    $$
    f(x) x = \bigl(f(x) - f(0)\bigr) (x-0) \ge (f(x)-f(0))^2 =f(x)^2, 
    $$
    or equivalently $f(x)(x-f(x))\ge 0$ for all $x\in \R$. This gives that $0 \le f(x)\le x$ if $x>0$ and $x\le f(x) \le 0$ if $x<0$. 
    This also applies to $x\mapsto \prox[\kappa_{**}\loss](x)$.
    Thus, in the event
    \begin{equation}
    \{ \alpha_* t G + Z < 0\}\cap \{\alpha_* G + Z > 0 \}.
    \label{event}
    \end{equation}
    the left-hand side of \eqref{eq:partial_prox_identity}
    is non-negative while the right-hand side is non-positive
    so both equal 0, i.e,
    $$
    0\in \partial \loss(\prox[\kappa_*\loss](\alpha_* G+Z)),
    \qquad
    0\in \partial \loss(\prox[\kappa_{**}\loss](t\alpha G+Z)).
    $$
    Since $\argmin\loss=\{0\}$, this implies
    $0=\prox[\kappa_*\loss](\alpha_* G+Z)=\prox[\kappa_{**}\loss](t\alpha_* G+Z)$ under the event \eqref{event}.
    In turn, this means that the left-hand side
    of \eqref{eq:partial_prox_identity} is strictly positive whereas the right-hand side is strictly negative: a contradiction.
    This contradiction means that the event \eqref{event} is empty. Exchanging the role of the LHS and RHS in the above argument also implies that $\{ \alpha_* t G + Z > 0\}\cap\{\alpha_* G + Z < 0 \}$ is empty. Thus, we must have
    \begin{equation}\label{eq:event_null}
        \PP\bigl( \alpha_* t G + Z < 0 \text{ and } \alpha_* G + Z > 0\bigr) = \PP\bigl(\{\alpha_* t G + Z > 0 \text{ and } \alpha_* G + Z < 0\}\bigr) = 0.    
    \end{equation}
    Now we verify that \eqref{eq:event_null} does not hold
    unless $t=1$. We proceed by contradiction. Switching the role of $v_*$ and $v_{**}$,  we assume $t>1$. By \Cref{as:noreg}-(2), at least 
    $\PP(Z>0)$ or $\PP(Z<0)$ is strictly positive. 
    If $\PP(Z>0)>0$, then we can always find 
    $K>0$ such that $\PP(Z\in(K,\sqrt t K))>0$, and the event
    $\{Z\in(K,\sqrt t K),-K< \alpha_* G < -\tfrac{K}{\sqrt t}\}$
    has a positive probability by the independence of $(Z,G)$ and $G\sim N(0,1)$. In the event $\{Z\in(K,\sqrt t K),-K< \alpha_* G < -\tfrac{K}{\sqrt t}\}$, we have
    \begin{align*}
        \alpha_* t G + Z  < \alpha_* t G + \sqrt t K < -\sqrt t K + \sqrt t K =0,
        \qquad
        \alpha_* G + Z> \alpha_* G + K > -K + K =0,
    \end{align*}
    so that $\PP( \alpha_* t G + Z < 0 \text{ and } \alpha_* G + Z > 0) > 0$. By the same argument, if $\PP(Z<0) >  0$, taking $K>0$ such that $\PP(Z \in (-\sqrt{t} K, - K))>0$, we have
    $$
    \PP( \alpha_* t G + Z> 0 \text{ and } \alpha_* G + Z < 0) \ge \PP\bigl(Z \in (-\sqrt{t} K, - K)\bigr)\PP\bigl(G\in (\tfrac{K}{\sqrt{t}\alpha_*}, \tfrac{K}{\alpha_*})\bigr) > 0.
    $$
    Thus, $t\ne 1$ leads to a contradiction with \eqref{eq:event_null}, so $t=1$ must hold. 
\end{proof}

\begin{lemma}[Uniqueness of Lagrange multiplier]\label{lm:unique_lagrange_noreg}
If $v_{*} \ne 0$ then the Lagrange multiplier $\mu_{*}>0$ satisfying the KKT condition \eqref{eq:KKT_noreg}  is also unique.  
\end{lemma}
\begin{proof} 
Let $\mu_{*}$ and $\mu_{**}$ be Lagrange multipliers. By \Cref{lm:optimization_to_system_noreg}, we must have 
$$
v_{*}=\prox[\kappa_{*}\loss](\alpha_{*} G + Z)-Z = \prox[\kappa_{**}\loss](\alpha_{*} G +Z)-Z
$$
where $\alpha_*=\|v_*\|/\sqrt{1-\delta^{-1}}$, $\kappa_{*}=\|v_*\|/\mu_{*}$ and $\kappa_{**}=\|v_*\|/\mu_{**}$. Below we show $\kappa_{*}=\kappa_{**}$. 
Letting $Y=\alpha_* G +Z$, the above display reads to $\prox[\kappa_*\loss](Y) =\prox[\kappa_{**}\loss](Y)$. This means that the intersection  
\begin{align*}
    \kappa_{**} \partial \loss (\prox[\kappa_{**}\loss](Y))
    \cap
    \kappa_* \partial \loss(\prox[\kappa_* \loss](Y))
\end{align*}
is non-empty because it contains the real
$Y-\prox[\kappa_*\loss](Y)=Y-\prox[\kappa_{**}\loss](Y)$.
Thus, we have 
$$
\textstyle
\kappa_{**} \min\partial \loss (\prox[\kappa_{**}\loss](Y))
\le \kappa_* \max\partial \loss (\prox[\kappa_*\loss](Y))
\le \kappa_* \sup_{x>0} \max \partial \loss(x) = \kappa_* L_+
$$
where we have taken $L_+ := \sup_{x>0} \max\partial \loss(x) > 0$. 
Since $\loss$ is convex and Lipschitz, thanks to the monotonicity of the subdifferential
($x<y\Rightarrow \max\partial \loss(x)\le \min\partial\loss(y)$),
$$
\lim_{n\to\infty} \min \partial \loss(n) = \sup_{x>0}\min \partial \loss(x) = \sup_{x>0} \max\loss(x) = L_+>0. 
$$
In the event $\{Y>n+\max\{\kappa_*,\kappa_{**}\}\|\loss\|_{\lip}\}$ which has positive probability
because $Y = \alpha_* G+Z$ has a continuous distribution, we have 
$$
\prox[\kappa_{**}\loss](Y) \ge Y - \kappa_{**}\|\loss\|_{\lip} \ge n + \max\{\kappa_*, \kappa_{**}\} \|\loss\|_{\lip}-\kappa_{**}\|\loss\|_{\lip} \ge n
$$
so that the monotonicity of the subdifferential implies 
$$
\textstyle
\kappa_{**} \min\partial\loss(n) \le \kappa_{**} \min\partial\loss(\prox[\kappa_{**}\loss](Y)) \le \kappa_* L_+
\qquad \text{ for any fixed }n. 
$$ 
Taking $n\to+\infty$ gives $\kappa_{**} L_+\le \kappa_* L_+$ and $\kappa_{**}\le \kappa_*$.
Exchanging the role of $\kappa_*$ and $\kappa_{**}$ in the previous argument
gives the reverse inequality and $\kappa_*=\kappa_{**}$. 
\end{proof}

\subsection{Proof of \Cref{lm:nonzero_noreg}}\label{proof:lm:nonzero_noreg}
We proceed by contradiction. If $v=0\in\mh$ is the minimizer, then \Cref{lm:lagrange_noreg} implies that there exists an associated Language multiplier $\mu>0$ such that $0\in \argmin_{v\in\mh} \ml(v) + \mu \mg(v)$, or equivalently, 
$$
0 \le \E[\loss(v+Z)-\loss(Z)] + \mu \bigl(\E[v^2]^{1/2}-\E[vG]/\sqrt{1-\delta^{-1}}\bigr) \quad \text{for all $v\in\mh$}. 
$$ 
Multiplying the both sides by $\lambda = \sqrt{1-\delta^{-1}}/\mu >0$ and denoting \( f(\cdot)=\lambda (\loss(\cdot + Z)-\loss(Z))\), we have
\begin{equation}\label{eq:ma_minorant}
0 \le \ma(v) : = \E[f(v)] + \E[v^2]^{1/2} \sqrt{1-\delta^{-1}} - \E[vG] \quad\text{for all $v\in \mh$}. 
\end{equation}
Now we parametrize $v\in\mh$ as $v_t = \prox[t f](t G) \in \mh$ for all $t>0$ and show $\ma(v_t) < 0$ for sufficiently small $t>0$. 
Note that $t^{-1} (tG - v_t) \in \partial f(v_t)$ implies
$$
- f(v_t) = f(0) - f(v_t) \ge t^{-1}(tG-v_t) (0-v_t) = -G v_t + t^{-1} v_t^2.
$$
Then, noting that $\E[v_tG]$ is cancelled out, $\ma(v_t)$ is bounded from above as 
\begin{align}\label{eq:ma_upper_bound}
\ma(v_t) &\le \E[Gv_t - t^{-1}v_t^2] + \E[v_t^2]^{1/2} \sqrt{1-\delta^{-1}} - \E[v_tG] \le - t \cdot \|t^{-1} v_t\| \bigl(
\|t^{-1} v_t\| - \sqrt{1-\delta^{-1}} 
\bigr).
\end{align}
Now we identify the limit of $\|v_t/t\|$ as $t\to 0+$. 
Note that the Moreau envelope constructed function
\( t\mapsto \env_f(tG, t) = \frac{1}{2t}(tG-v_t)^2 + f(v_t)\)
has derivative
\[
-\frac{1}{2 t^2}(tG - v_t)^2 + \frac{1}{t}G (tG - v_t)
= \frac 1 2 \Bigl[G^2 - \Bigl(\frac{v_t}{t} \Bigr)^2\Bigr],
\]
which is increasing in \( t \) because the Moreau envelope $\env_f(x,y)$ is jointly convex in $(x, y)\in \R\times \R_{>0}$ (cf.  \cite[Lemma D.1]{thrampoulidis2018precise}). 
This means that \( v_t^2/t^2 \) is non-increasing in \( t \) and hence 
\( v_t^2/t^2 \) must have a non-negative limit as \( t\to 0+ \) by the monotone convergence theorem. Now we identify the limit of $t^{-1}v_t$:
\begin{itemize}
\item 
If \( G\in \partial f(0) \) then \( v_t = 0 \) satisfies
\( v_t/t + \partial f(v_t) \ni G \).
\item
If \( G>\max \partial f(0) \)
then \( v_t/t > 0 \) eventually so
\( \partial f(v_t) \to \max \partial f(0) \).
\item
If \( G<\min\partial f(0) \)
then \( v_t/t < 0 \) eventually so
\( \partial f(v_t) \to \min \partial f(0) \).
\end{itemize}
Thus, the square of $(v_t/t)\in G-\partial f(v_t)$ converges pointwise as follows
\[
(v_t/t)^2
\to
(G-\max {\partial} f(0))_+^2
+
(\min {\partial} f(0)-G)_+^2.
\]
Furthermore,
$(v_t/t)^2$ is uniformly bounded from above by the integrable random variable
$(|G| + \lambda \|\loss\|_{\lip})^2$ since $f(\cdot)=\lambda(\loss(\cdot + Z)-\loss(Z))$ is $\lambda\|\loss\|_{\lip}$-Lipschitz and $G-t^{-1}v_t\in\partial f(v_t)$.   
Thus, the dominated convergence theorem yields
\begin{align*}
\lim_{t\to +0} \E[(v_t/t)^2]& = \E[(G- \max \partial f(0))_+^2 + (\min\partial f(0)-G)_+^2] \\
& = \E[(G- \lambda \max \partial\loss(Z))_+^2 + (\lambda \min\partial \loss(Z)-G)_+^2] && \text{ by $f(\cdot)=\lambda(\loss(\cdot+Z)-\loss(Z))$}
\\
&= \E[\dist(G, \lambda \partial \loss(Z))^2],
\end{align*}
where we have used $\dist(\cdot, S) = \inf_{x\in\R} |\cdot-x|$ for the last equality. 
Combined with the upper bound of $\ma(v_t)$ in \eqref{eq:ma_upper_bound}, we obtain
\begin{align*}
\ma(v_t) &\le - t \|\dist(G, \lambda \partial \loss(Z))\| \bigl(\|\dist(G, \lambda \partial \loss(Z))\| - \sqrt{1-\delta^{-1}}\bigr) + o(t).
\end{align*}
Here, the leading term is strictly positive since 
$\|\dist(G, \lambda \partial \loss(Z))\| \ge \inf_{t>0} \|\dist(G, t \partial \loss(Z))\| > \sqrt{1-\delta^{-1}}$ by the assumption $\delta<\delta_{\mathsf{perfect}}$. Thus, there exists some $t_0 > 0$ such that 
$\ma(v_{t_0}) < 0$, a contradiction with \eqref{eq:ma_minorant}. Therefore, $0\in\mh$ cannot be the minimizer of $\min_{\mg(v)\le 0}\ml(v)$. 

\subsection{Proof of \Cref{lm:coercive_noreg}}\label{proof:lm:coercive_noreg}

\begin{lemma}\label{lm:coercive_noreg_detail}
    If $\delta>1$ and $\loss:\R\to\R$ is coercive in the sense of \eqref{eq:coercive_condition},  
     $\ml$ is coercive on the subset $\{\mg(v)\le 0\}$;
     more precisely, we have
     \begin{equation}\label{eq:l_coercive_noreg}
    \ml(v) \ge C_\delta  a \|v\|- 2 \|\loss\|_{\lip} \mathsf{q}\Bigl({C}_\delta\frac{a^2}{\|\loss\|_{\lip}^2}\Bigr) - b  \quad \text{for all $v\in\mh$ such that $\mg(v)\le 0$},  
     \end{equation}
    where $C_\delta>0$ is a constant depending on $\delta$ only, $(a,b)$ is any constant satisfying \eqref{eq:coercive_condition}, 
     and $\mathsf{q}(x) := \inf \{q>0 : \PP(|Z|\ge q) \le x\}$. 
\end{lemma}
\begin{proof}
Taking $C_\delta>0$ large enough that 
$\PP(|G|>C_\delta) \le (1-\delta^{-1})/4$, we have that for all $v\in \{\mg(v)\le 0\}$, 
\begin{align*}
    \E[v^2]^{1/2} &\le \E[vG]/\sqrt{1-\delta^{-1}} && \text{ by $\mg(v)\le0$}\\
    &\le \E\bigl[|v||G| I\{|G|> C_\delta\} + |v|C_\delta\bigr]/\sqrt{1-\delta^{-1}}\\
    &= \E[v^2]^{1/2}/2 + \E[|v|] C_\delta/\sqrt{1-\delta^{-1}} && \text{ by the Cauchy--Schwarz inequality}. 
\end{align*}
This gives 
$\E[v^2]^{1/2} \le \E[|v|] \tilde C_\delta$ with $\tilde C_\delta = 2C_\delta/\sqrt{1-\delta^{-1}}>0$.  
Combined with the Paley--Zygmund inequality, we have
\begin{align*}
   \PP\bigl(|v| > \E[v^2]^{1/2}/(2\tilde C_\delta)\bigr) &\ge \PP\bigl(|v|> \E[|v|]/2\bigr) && \text{ by $\E[v^2]^{1/2} \le \E[|v|] \tilde C_\delta$}\\
   &\ge (1-2^{-1})^2 \E[|v|]^2/\E[v^2] && \text{ by the Paley--Zygmund inequality}\\
   &\ge 3/4 \cdot  (\tilde C_\delta )^{-2}  && \text{ by $\E[v^2]^{1/2} \le \E[|v|] \tilde C_\delta$} 
\end{align*}
Letting $C_\delta  = \min(1, (2\tilde{C}_\delta)^{-1}, 3/4\cdot \tilde{C}_\delta^{-2}) \le 1$ in the above display, we obtain the claim below:
\begin{equation}\label{eq:l1_l2_noreg}
\exists C_\delta \in (0,1], \quad \forall v\in \{v\in\mh: \mg(v)\le 0\}, \quad \PP(|v|> C_\delta  \E[v^2]^{1/2}) \ge C_\delta.    
\end{equation}
Now we bound $\ml(v)$ from below using \eqref{eq:l1_l2_noreg}. 
For some $\epsilon>0$ to be specified later, we decompose $\ml(v)$ into two terms
$$
\ml(v) = \E\bigl[I\{|Z|\ge \mathsf{q}(\epsilon) \} (\loss(Z+v) - \loss(Z))\bigr] + \E\bigl[I\{|Z|\le \mathsf{q}(\epsilon) \}(\loss(Z+v) - \loss(Z))\bigr],
$$
where $\mathsf{q}(x) := \inf \{t>0 : \PP(|Z|\ge t) \le x\}$. 
By the Cauchy--Schwarz inequality and Lipschitz condition of $\loss$, 
the first term can be bounded from below as 
$$
\E\bigl[I\{|Z|\ge \mathsf{q}(\epsilon) \} (\loss(Z+v) - \loss(Z))\bigr]\ge - \E\bigl[I\{|Z|\ge \mathsf{q}(\epsilon) \} \|\loss\|_{\lip}  |v|\bigr] \ge - \sqrt{\epsilon} \|\loss\|_{\lip} \|v\|.
$$
For the second term, the coercivity $\loss(\cdot) -\loss(0) \ge a |\cdot| -b$ leads to 
\begin{align*}
    &\E\bigl[I\{|Z|\le \mathsf{q}(\epsilon) \}(\loss(Z+v) - \loss(Z))\bigr]\\
    &\ge \E\bigl[I\{|Z|\le \mathsf{q}(\epsilon) \}(\loss(Z+v) -\loss(0))\bigr] - \|\loss\|_{\lip} \mathsf{q}(\epsilon)\\
    &\ge \E\bigl[I\{|Z|\le \mathsf{q}(\epsilon) \}
    (a |v+Z|-b) \bigr]
    -  \|\loss\|_{\lip} \mathsf{q}(\epsilon)  && \text{by $\loss(\cdot)-\loss(0) \ge a |\cdot|-b$}\\
    &\ge  a \E\bigl[I\{|Z|\le \mathsf{q}(\epsilon) \}
    |v|\bigr] - aQ(\epsilon)  - b
    -  \|\loss\|_{\lip}\mathsf{q}(\epsilon)  && \text{by $|v+Z| \ge |v|-|Z|$}\\
    &\ge  a \E\bigl[I\{|Z|\le \mathsf{q}(\epsilon) \}
    |v|\bigr] - 2\|\loss\|_{\lip} \mathsf{q}(\epsilon)  - b && \text{by $a\le \|\loss\|_{\lip}$}.
\end{align*}
Putting the above displays together, we have the following:
\begin{equation}\label{eq:l_bound_rho}
    \forall\epsilon>0, \quad 
\forall v\in \mh, \quad 
\ml(v) \ge -\sqrt{\epsilon} \|\loss\|_{\lip}\|v\| + a \E\bigl[I\{|Z|\le \mathsf{q}(\epsilon)\} |v|\bigr] - 2\|\loss\|_{\lip} \mathsf{q}(\epsilon) - b
\end{equation}
Using the union bound for $\PP(|Z|>\mathsf{q}(\epsilon))\le \epsilon$ and $\PP(|v|\ge C_\delta\|v\|)\ge C_\delta$ in 
\eqref{eq:l1_l2_noreg}, it follows that
\begin{align*}
    \E\bigl[I\{|Z|\le \mathsf{q}(\epsilon)\} |v|\bigr]\ge C_\delta\|v\| \cdot \E\bigl[I\{|Z|\le \mathsf{q}(\epsilon)\} I\{|v|\ge C_\delta \|v\|\} \bigr] \ge C_\delta \|v\| (C_\delta - \epsilon).
\end{align*}
Substituting this lower bound to \eqref{eq:l_bound_rho}, we have that for all $\epsilon\in (0,1)$, 
 \begin{align*}
    \ml(v)
    &\ge -  \sqrt{\epsilon} \|\loss\|_{\lip} \|v\| + a C_\delta  \|v\| (C_\delta  -\epsilon) -  2\|\loss\|_{\lip} \mathsf{q}(\epsilon)  -b\\
    &= \|v\| \bigl(aC_\delta ^2 - \sqrt{\epsilon}\|\loss\|_{\lip} - a C_\delta \epsilon\bigr) -2 \|\loss\|_{\lip}\mathsf{q}(\epsilon) - b\\
    &= \|v\| \bigl(aC_\delta ^2 - 2\sqrt{\epsilon}\|\loss\|_{\lip}\bigr) -2 \|\loss\|_{\lip}\mathsf{q}(\epsilon) - b. && \text{by $a\le \|\loss\|_{\lip}$, $C_\delta \le 1$ and $\epsilon \le \sqrt{\epsilon}$}.
 \end{align*}
 If we take $\epsilon\in(0,1)$ such that 
 $\sqrt{\epsilon} = aC_\delta^2/(4\|\loss\|_{\lip}) (\le 1/4)$, we have
 $$
\ml(v) \ge \|v\|\frac{aC_\delta^2}{2} - 2\|\loss\|_{\lip} \mathsf{q}\bigl(
    \frac{a^2C_\delta^4}{16\|\loss\|_{\lip}^2}
\bigr) - b.
 $$
 Taking $C_\delta' = \min(C_\delta^2/2, C_\delta^4/16) = C_\delta^4/16$ in the above display gives the desired lower bound of $\ml(v)$. 
\end{proof}

\section{Proof for \Cref{sec:reg}}
\subsection{Set up for infinite-dimensional optimization problem}
\begin{lemma}\label{lm:lg_basic_reg}
    Suppose $\loss:\R\to\R$ and $\reg:\R\to\R$ are convex and Lipschitz. Then the maps $(\ml, \mg)$ defined by
    \begin{align*}
        \ml: \mh \to \R, \quad &(v,w) \mapsto \delta\E[\loss(v+Z)-\loss(Z)] + \E[\reg(w+X)-\reg(X)]\\
        \mg: \mh \to \R, \quad &(v,w)\mapsto \mt(v,w) - \delta^{-1/2}\E[wH]
    \end{align*}
    with 
    $$
    \mt(v,w) := \bigl\{(\|w\|-\E[vG])_+^2 + \E[\proj(v)^2]\bigr\}^{1/2} \quad \text{with} \quad \proj(v) =v-\E[vG]G
    $$
    are convex, Lipschitz, and finite valued. 
    Furthermore, 
    $\mg$ is Fr\'echet differentiable at $(v,w)$ if $\mt(v,w)>0$ and $\|w\|>0$, in the sense that
    $$
    \mg(v+h,w+\eta)=\mg(v,w) + \E[\nabla_v\mg(v,w)h + \nabla_w\mg(v,w)\eta] + o(\|h\|+\|\eta\|), 
    $$
    where $\nabla_v\mg(v,w)$ and $\nabla_{w}\mg(v,w)$ are the Fr\'echet derivatives given by 
    \begin{align*}
        \nabla_v\mg(v, w) &= \frac{-(\|w\|-\E[vG])_+ G + v - \E[vG]G}{\mt(v,w)}, \\ 
        \nabla_w\mg(v, w)&=
        \frac{(1-\|w\|^{-1}\E[vG])_+w}{\mt(v,w)} - \delta^{-1/2}H
    \end{align*}
\end{lemma}
\begin{proof}
    The convexity of $\ml$ immediately follows from the convexity of $\loss$ and $\reg$. $\ml$ is finite valued and Lipschitz since $\ml(0,0)=0$ and 
\begin{align*}
&|\ml(v,w)-\ml(v',w')| \\
&\le \delta \E[|\loss(v+Z)-\loss(v'+Z)|] + \E[|\reg(w+X)-\reg(w'+X)|] &&\text{by the triangle ineq.}\\
&\le \delta \|\loss\|_{\lip} \E[|v-v'|] + \|\reg\|_{\lip} \E[|w-w'|] \\
&\le (\delta\|\loss\|_{\lip} + \|\reg\|_{\lip}) (\|v-v'\| + \|w-w'\|) &&\text{by Jensen's ineq.}\\
&\le \sqrt{2}(\delta\|\loss\|_{\lip} + \|\reg\|_{\lip}) \sqrt{\|v-v'\|^2 + \|w-w'\|^2} &&\text{by $a+b \le \sqrt{2(a^2+b^2)}$}
\end{align*}
for all $(v,w), (v',w')\in \mh$. 
    Below we prove the claims for $\mg$. 
    The function \( \mathcal T(v,w) \) is the composition
    of the Lipschitz functions
    \begin{align}
        &\mh\to\R^2, && (v,w)\mapsto (\|w\|-\E[vG], \|\proj(v)\|),
        \\
        &\R^2\to\R_{\ge 0}^2, && (a,b)\mapsto (a_+,b_+),
        \\
        &\R_{\ge 0}^2\to \R, && (a,b)\mapsto \|(a,b)\|_2=\{a^2+b^2\}^{1/2}
    \end{align}
    so that \( (v,w) \to \mt(v,w) \) is Lipschitz.
    This implies that
    $\mg$ is Lipschitz,
    and finite valued at every \( (v,w)\in\mh \) thanks to
    \( \mg(0,0)=0 \).
    Next, we show that $(v, w)\mapsto \mg(v,w)=\mt(v,w)-\delta^{-1/2}\E[Hw]$ is jointly convex. Since $\E[Hw]$ is linear in $w$, it suffices to show the convexity of $(v,w)\mapsto \mt(v,w)$.
    Let us define $a(v, w) := \|w\|_2 - \E[vG]$ and $b(v, w) := \|\proj(v)\|$ so that $\mt(v,w) =  \sqrt{a(v,w)_+^2 + b(v,w)_+^2}$. Denoting $v_t=tv' + (1-t)v'$ and $w_t =t w + (1-t)w'$ for all $t\in[0,1]$, the convexity of $(v,w)\mapsto a(v,w)$ and $(v,w)\mapsto b(v,w)$ yields
    $$
    \begin{pmatrix}
      a(v_t, w_t)\\
      b(v_t, w_t)
    \end{pmatrix} \le \begin{pmatrix}
      t a(v, w) + (1-t) a(v', w')\\
      t b(v, w) + (1-t) b(v', w')
    \end{pmatrix} = t \begin{pmatrix}
      a(v, w)\\
      b(v, w)
    \end{pmatrix} + (1-t)\begin{pmatrix}
      a(v',w')\\
      b(v',w')
    \end{pmatrix}, 
    $$
    where the inequality holds coordinate-wisely. If we apply the convex and increasing map
    $\|(\cdot)_+\|_2^2 : \R^2 \to \R$, $(a, b) \mapsto (a_+^2 + b_+^2)^{1/2}$ to the above inequality, we are left with 
    $$
    \sqrt{a(v_t, w_t)_+^2 + b(v_t, w_t)_+^2} \le t\sqrt{a(v,w)_+^2 + b(v,w)_+^2} + (1-t) \sqrt{a(v',w')_+^2 + b(v',w')_+^2}
    $$
    and $\mt(v_t,w_t)\le t\mt(v,w)+(1-t)\mt(v',w')$. This means that $(v,w)\mapsto \mt(v,w)$ is jointly convex, thereby completing the proof of the joint convexity of $(v,w)\mapsto \mg(v,w)$. The derivative easily follows from the Pythagorean theorem $\|\proj(v)\|^2 = \E[v^2]- \E[vG^2]$. 
\end{proof}

\begin{lemma}\label{lm:lagrange_reg}
    $(v_*,w_*)\in \mh$ solves the constrained optimization problem
    $$
    \min_{v,w\in\mh} \ml(v,w) \quad \text{subject to} \quad \mg(v,w)\le 0
    $$
    if and only if there exists an associated Lagrange multiplier $\mu_*\ge 0$ such that the KKT condition
    \begin{align}\label{eq:KKT_reg}
        -\mu_* \partial\mg(v_*, w_*)\cap \partial \ml(v_*, w_*)\ne \emptyset, \quad \mu_* \mg(v_*,w_*)=0, \quad \text{and} \quad \mg(v_*, w_*)\le 0
    \end{align}
    is satisfied. Furthermore, if $\argmin_x\loss(x)=\argmin_x \reg(x) =\{0\}$ and $\PP(Z\ne 0)>0$ then the Lagrange multiplier $\mu_*$ must be always strictly positive and the constraint $\mg(v, w)\le 0$ is binding, i.e., $\mg(v_{*}, w_*)=0$.  
\end{lemma}

\begin{proof}
Note that the set $\{(v,w)\in\mh: \mg(v,w) < 0\}$ is not empty since
$$
\mg(2G, H) = \{(\|H\|_2 - 2 \E[G^2])_+^2 + \E[\proj(2G)^2]\}^{1/2}-\delta^{-1/2} \E[H^2] = -\delta^{-1/2} <0. 
$$
Thus, the if and only if part follows from \Cref{lm:Lagrange}. 
It remains to show the positiveness $\mu_*>0$.  We proceed by contradiction. Suppose $\mu_*=0$. Then, $(v_*, w_*)$ is the minimizer of $\min_{v,w} \ml(v,w)$. Let us take $u_n = -I\{|\loss(Z)| \le n\}Z$ and $w_n = -I\{|\reg(X)| \le n\}X$ so that $v_n$ and $w_n$ are bounded in $L_2$ by the coercivity of $\loss$ and $\reg$ (see the proof of \Cref{lm:lagrange_noreg}). Note that $\ml(v_n, w_n)$ can be written as
$$
\ml(v_n, w_n) = \delta\E[I\{|\loss(Z)| \le n\} (\loss(0)-\loss(Z))] + \E[I\{|\reg(X)|\le n\} (\reg(0)-\reg(X))].
$$
Then the same argument in the proof of \Cref{lm:lagrange_noreg} yields $\E[\reg(X)] <+\infty$, $\E[\loss(Z)]<+\infty$, and $\lim_{n\to\infty} \mathcal{L}(v_n, w_n)=\delta \E[\loss(0)-\loss(Z)] + \E[\reg(0)-\reg(X)]$. Combined with $(v_*, w_*)\in \argmin_{(v,w)}\ml(v,w)$, we have
\begin{align*}
\delta \E[\loss(v_*+Z)-\loss(0)] + \E[\reg(w_*+X)-\reg(0)] \le 0.
\end{align*}
Thanks to $\{0\} = \argmin_x \loss(x) = \argmin_x \reg(x)$, the above display gives $v_*=-Z$ and $w_*=-X$.  Substituting this to the constraint $\mg(v_*, w_*)\le 0$, using the independence of $(G,Z)$ and $(H,Z)$, we are left with
\begin{align*}
& 0 \ge \mg(v_*, w_*) = \sqrt{(\E[(-X)^2]^{1/2} -\E[-WG])_+^2 +\E[\proj(-Z)^2]} - \delta^{-1/2} \E[-XH] = \sqrt{\E[X^2] + \E[Z^2]},
\end{align*}
so in particular $\E[Z^2]=0$, a contradiction with $\PP(Z \ne 0) >0$. Therefore, $\mu_*$ must be strictly positive. 
\end{proof}

\begin{lemma}\label{lm:vG_decrease}
    For any $v\in\mh$ and any $t\in [0, \pi/2]$, define $v_t\in\mh$ as 
    $$
    v_t(G, Z) := \E[v(G\cos t+ \tilde{G}\sin t, Z)|G, Z] \quad \text{where} \quad \tilde{G}\sim N(0,1) \ind (G, Z). 
    $$
    Then, we have
    $$
    \E[v_t G] = \E[v G] \cos t, \quad \|\proj(v_t)\| \le \|\proj(v)\|, \quad \E[\loss(v_t+Z)-\loss(v+Z)] \le 0. 
    $$
\end{lemma}
\begin{proof}
Let us define $G_t:=G\cos t +\tilde G\sin t$ and $\dot G_t = -G\sin t + \tilde G \cos t$ 
so that $v_t=\E[v(G_t, Z)|G, Z]$, \( G= G_t \cos t - \dot G_t\sin t \) and $G_t\ind \dot{G}_t$. Then have 
$$
\E[v_t(G,Z)G]=\E[v(G_t,Z)G]
= \E[v(G_t,Z)(G_t \cos t - \dot G_t \sin t)]
=\E[v(G_t,Z)G_t]\cos t.
$$
Since $(G_t, Z)=^d (G, Z)$, this gives $\E[v_t G]=\E[vG]\cos{t}$. Furthermore, \begin{align*}
    \E[\proj(v_t)^2] &= \E\Bigl[
    \bigl(v_t(G,Z)-\E[G v_t(G,Z)]G\bigr)^2
    \Bigr]\\
    &=
    \E\Bigl[
    \bigl(v_t(G,Z)-\E[Gv(G,Z)]G\cos t\bigr)^2
    \Bigr] &&\text{by} \E[v_t G]=\E[vG]\cos t
    \\&=
    \E\Bigl[
    \Bigl(
        \E\Bigl[
    v(G_t,Z)-\E[Gv(G,Z)]G_t
    \mid G,Z \Bigr]
    \Bigr)^2
    \Bigr]
        &&\text{by def. of $v_t$ and $\E[G_t|G,Z]=G\cos t$}
    \\&\le
    \E\Bigl[
    \bigl(
    v(G_t,Z)-\E[G v(G,Z)]G_t
    \bigr)^2
    \Bigr]
        &&\text{Jensen's ineq.}
    \\&= \E[\proj(v)^2] 
        &&\text{since $G_t=^d G$}. 
    \end{align*}
    Finally, using $v(G,Z)=^d v(G_t,Z)$, we find
\begin{align*}
    &\E[\loss(Z+v_t(G,Z))-\loss(Z+v(G,Z))]
    \\&
    =
    \E
    \{\loss(Z+v_t(G,Z))-\loss(Z+v(G_t,Z))
    \}
        &&\text{equality in distribution}
    \\&=
    \E\{
    \loss(Z+v_t(G,Z))-
    \E[
    \loss(Z+v(G_t,Z))
    \mid G,Z]
\} &&\text{iterated expectations}
    \\&\le
    \E\{
    \loss(Z+v_t(G,Z))-
    \loss(Z+
    \E[
    v(G_t,Z)
    \mid G,Z]
    )
\} &&\text{conditional Jensen's inequality}
    \\&=0&&\text{by definition of $v_t$}.
\end{align*}
This finishes the proof.
\end{proof}

\subsection{Proof of \Cref{lm:equivalence_potential_reg}}\label{proof:lm:equivalence_potential_reg}
\begin{lemma}\label{lm:M1_M2_equivalence}
    Fix $\alpha\ge 0$ and consider the optimization problem:
    \begin{align}\label{eq:M1_alpha}
        \min_{v,w} \mathcal{L}(v,w) \text{ subject to }
    \begin{cases}
    \|w\|\le \alpha , \\
     \|\alpha G - v\|\le \E[Hw]/\sqrt \delta.
    \end{cases}        
    \end{align}
    Then this optimization problem admits a minimizer $(v_\alpha, w_\alpha)$. Furthermore, this solution also solves the optimization problem with the relaxed condition
    \begin{align}\label{eq:M2_alpha}
    \min_{v,w} \mathcal{L} (v,w) \text{ subject to }
        \begin{cases}
        \|w\|\le \alpha , \\ \{(\alpha-\E[vG])_+^2+\|\proj(v)\|^2\}^{1/2}\le \E[Hw]/\sqrt \delta
        \end{cases}.
    \end{align}
\end{lemma}
\begin{proof}
    Note that the set of $(v,w)\in \mh$ satisfying the constraint \eqref{eq:M1_alpha} is closed and convex in $\mh$, as well as bounded since the Cauchy--Schwarz inequality gives $\|w\|\le \alpha$ and 
    $$
    \|v\|\le \|\alpha G-v\| + \|\alpha G\| \le \E[Hw]/\sqrt{\delta} + \alpha \le \alpha(1/\sqrt{\delta} + 1). 
    $$
    Then \cite[Proposition 11.15]{bauschke2017correction} implies that there exists a solution $(v_*, w_*)$ to (2). 
    Next, we show that the solution to \eqref{eq:M1_alpha} also solves \eqref{eq:M2_alpha}. If $\alpha=0$, the constraint in \eqref{eq:M1_alpha} gives $v=w=0$ so that $(0,0)$ is the unique solution. 
    On the other hand, the constraint of \eqref{eq:M2_alpha} with $\alpha=0$ is equivalent to $(v,w)=(-tG, 0)$ for $t\ge 0$. Since the map $\R_{\ge 0}\ni t\mapsto \ml(-tG, 0) = \delta\E[\loss(-tG+Z)-\loss(Z)]$ is increasing thanks the convexity of $\loss$ and Jensen's inequality, we find that $(v,w) = (0,0)$ also solves \eqref{eq:M2_alpha}. 
    Let us consider the case $\alpha>0$. Let $(v_\alpha, w_\alpha)$ be the solution to \eqref{eq:M1_alpha}. 
    Now we fix $(v,w)\in \mh$ satisfying the constraint in \eqref{eq:M2_alpha}. If $(v,w)$ further satisfy $\E[vG]\le \|w\|$, combined the first condition $\|w\|\le \alpha$, the term $(\alpha-\E[vG])$ inside the second constraint is nonnegative. Thus, the second constraint yields $\|\alpha G-v\| \le \E[Hw]/\sqrt{\delta}$ so that $(v,w)$ satisfies the all constraints of \eqref{eq:M1_alpha}. Otherwise, i.e., if $\E[vG] > \|w\|$, then \Cref{lm:vG_decrease} with $t=\arccos({\|w\|/\E[vG]})$ implies that 
    there exists a $\tilde{v}$ such that 
    $$
    \E[\tilde{v} G] = \|w\| ,\quad  \|\proj(\tilde{v})\|\le \|\proj(v)\|,  \quad \E[\loss(\tilde{v}+Z)-\loss(v+Z)] \le 0. 
    $$
    Thus, the above display implies that 
    the constraint in \eqref{eq:M1_alpha} is satisfied by the pair $(\tilde{v}, w)$ and $\ml(v,w)\ge \ml(\tilde{v}, w)$. For these reasons, we must have 
    \begin{align*}
        \ml(v,w) &\ge \text{ Objective {value} of \eqref{eq:M1_alpha}} =  \ml(v_\alpha, w_\alpha) 
    \end{align*}
    for all $(v,w)$ satisfying the constraint in \eqref{eq:M2_alpha}. 
    Since $(v_\alpha, w_\alpha)$ satisfies the constraint \eqref{eq:M2_alpha}, the above display implies that $(v_\alpha, w_\alpha)$ solves \eqref{eq:M2_alpha}. 
\end{proof}

Now let us define 
\begin{align*}
    \mathsf{M}(\alpha) &= \sup_{\beta > 0, \tau_h>0}\inf_{\tau_g > 0} \mathsf{D}(\alpha, \beta, \tau_h, \tau_g),\\
\text{where}  \quad \mathsf{D}(\alpha, \beta, \tau_h, \tau_g) &:= 
    \frac{\beta\tau_g}{2} + \delta \mathsf{L}(\alpha, \frac{\tau_g}{\beta})
 - \left\{\begin{array}{ll}
    \frac{\alpha\tau_h}{2} + \frac{\alpha\beta^2}{2\tau_h} - \mathsf{R}\Bigl(\frac{\alpha\beta}{\tau_h},\frac{\alpha}{\tau_h}\Bigr) & \alpha>0\\
    0 & \alpha=0
\end{array}\right..
\end{align*}
Recall that $\mathsf{L}$ and $\mathsf{R}$ are jointly convex functions over $\R\times \R_{>0}$ defined as 
\begin{align*}
    \forall c\in\R, \ \tau>0, \quad \begin{split}
        \mathsf{L}(c, \tau) &:= \E[\env_\loss(cG + Z; \tau)-\loss(Z)]\\
        \mathsf{R}(c, \tau) &:= \E[\env_\reg(cH + X; \tau)-\reg(X)]
    \end{split}
\end{align*}

\begin{lemma}\label{lm:M_M1_equivalence}
    $\mathsf{M}(\alpha)$ equals the objective value of \eqref{eq:M1_alpha} for all $\alpha\ge 0$. 
\end{lemma}

\begin{proof}
Note that $\mathsf{M}(0)=0$ by \Cref{lm:N_check_N}-(1), while the objective value of \eqref{eq:M1_alpha} is also $0$ since the constraint is reduced to $v=w=0$ when $\alpha=0$. Below we fix $\alpha>0$. 
For any $(v, w)\in \mh$ satisfying the constraint in \eqref{eq:M1_alpha} and for all $\beta > 0$, $\tau_h, \tau_g>0$, we have that
\begin{align*}
    \frac{\beta\tau_g}{2} + \delta \mathsf{L}(\alpha, \frac{\tau_g}{\beta})
    &= \frac{\beta\tau_g}{2} + \delta \E\Bigl[\min_{x\in\R} \Bigl(\frac{\beta}{2\tau_g}(x-\alpha G)^2 + \loss(x+Z)-\loss(Z)\Bigr)\Bigr]\\
    &\le \beta(\frac{\tau_g}{2} + \frac{\delta \|v-\alpha G\|^2}{2 \tau_g}) + \delta \E[\loss(v+Z)-\loss(Z)]
\end{align*}
and 
\begin{align*}
    -\frac{\alpha\tau_h}{2} - \frac{\alpha\beta^2}{2\tau_h} + \mathsf{R}\Bigl(\frac{\alpha\beta}{\tau_h},\frac{\alpha}{\tau_h}\Bigr)
    &=  -\frac{\alpha\tau_h}{2} - \frac{\alpha\beta^2}{2\tau_h} + \E\Bigl[\min_{x\in\R} 
    \Bigl(
        \frac{\tau_h}{2\alpha}(x-\frac{\alpha\beta}{\tau_h}H)^2 + \reg(x+X)-\reg(X)
    \Bigr)
    \Bigr] \\
    &\le -\frac{\alpha\tau_h}{2} - \frac{\alpha\beta^2}{2\tau_h} + \E\Bigl[
    \Bigl(
        \frac{\tau_h}{2\alpha}(w-\frac{\alpha\beta}{\tau_h}H)^2 + \reg(w+X)-\reg(X)
    \Bigr)
    \Bigr] \\
    &=\frac{\tau_h}{2\alpha}(-\alpha^2 + \|w\|^2) - \beta\E[Hw] + \E[\reg(w+X)-\reg(X)]\\
    &\le -\beta\E[Hw] + \E[\reg(w+X)-\reg(X)] \qquad \text{using } \|w\|\le \alpha
\end{align*}
Combining the above displays, we have 
\begin{align*}
    \mathsf{M}(\alpha) &\le \sup_{\beta > 0} \inf_{\tau_g>0} \beta(\frac{\tau_g}{2} + \frac{\delta \|v-\alpha G\|^2}{2 \tau_g}) + \delta \E[\loss(v+Z)-\loss(Z)] - \beta\E[Hw] + \E[\reg(w+X)-\reg(X)]\\
    &= \sup_{\beta > 0, \tau_h>0} \beta\sqrt{\delta} \|v-\alpha G\| -\beta\E[Hw] + \delta\E[\loss(v+Z)-\loss(Z)]+\E[\reg(w+X)-\reg(X)]\\
    &=\ml(v,w) + \sup_{\beta > 0} \beta\sqrt{\delta} (\|v-\alpha G\|-\E[Hw]/\sqrt{\delta})\\
    &\le \ml(v,w) \qquad \text{using } \|v-\alpha G\| \le \E[Hw]/\sqrt{\delta} 
\end{align*}
for all $(v,w)$ satisfying the constraint in \eqref{eq:M1_alpha}. This implies that $\mathsf{M}(\alpha)$ is less than the objective value of \eqref{eq:M1_alpha}. 

It remains to show the reverse inequality $\mathsf{M}(\alpha) \ge \text{\eqref{eq:M1_alpha}}$. Since \eqref{eq:M1_alpha}$=$\eqref{eq:M2_alpha} by \Cref{lm:M1_M2_equivalence}, it suffices to show  $\mathsf{M}(\alpha) \ge \text{\eqref{eq:M2_alpha}}$. Let $(v_\alpha, w_\alpha)$ be the solution to \eqref{eq:M2_alpha}.
Now we claim that there exists Lagrange multipliers $\mu_1 \ge 0$ and $\mu_2>0$ such that $(v_\alpha, w_\alpha)$ solves 
$$
\min_{v,w} \ml(v,w) + \mu_1 \mg_1(v,w) + \mu_2 \mg_2(v,w), \ \text{where} \ \begin{cases}
    \mg_1(v,w) := \|w\|^2-\alpha^2\\
    \mg_2(v,w) := \{(\alpha-\E[vG])_+^2+\|\proj(v)\|^2\}^{1/2}- \tfrac{\E[Hw]}{\sqrt \delta}
\end{cases} 
$$
As for the existence of nonnegative Lagrange multipliers $(\mu_1, \mu_2)$, by \Cref{lm:Lagrange}, it suffices to check $ \{\mg_1(v,w) < 0\} \cap \{\mg_2(v,w) < 0\} \neq \emptyset$. Indeed, $(v,w)=(2\alpha G, 2^{-1}\alpha H)$ satisfies $\mg_1(v,w)<0$ and $\mg_2(v,w)\le 0$. Let us show $\mu_2>0$. We proceed by contradiction. If $\mu_2$=0, then $(v_\alpha, w_\alpha)$ must solve
$$
v_\alpha \in\argmin_{v} \E[\loss(v+Z)-\loss(Z)], \quad w_\alpha \in \argmin_{w}\E[\reg(w+X)-\reg(X) + \mu_1 w^2].
$$
Then by the same argument in \Cref{lm:lagrange_reg} (by taking the truncated one $v_n=-Z I\{\loss(Z)\le n\}$ if necessary), we must have $v_\alpha = -Z$. The same argument applies to $w$, and we must have $w_\alpha = \prox[\mu_1^{-1} \reg] (X)-X$ (we write $w_\alpha=-X$ if $\mu_1=0$). Since $(v_\alpha, w_\alpha)$ satisfies the second constraint $\mg_2\le 0$, substituting $(v_\alpha, w_\alpha) = (-Z, \prox[\mu_1^{-1} \reg] (X)-X)$ and using the independence $v_\alpha \ind G$, $w_\alpha\ind H$, we obtain $0 \ge \mg_2(v_\alpha, w_\alpha) = \{\alpha^2 + \|Z\|^2\}^{1/2}$. This gives $\PP(Z=0)=1$, a contradiction with the assumption $\PP(Z\ne 0)>0$ in \Cref{as:reg}. This finishes the proof of $\mu_2>0$. 
Now, letting $\beta_*=\mu_2/\sqrt{\delta}>0$, we have the upper estimate:
\begin{align*}
    \text{\eqref{eq:M1_alpha}} &= \min_{v,w} \ml(v,w) + \mu_1(\|w\|^2-\alpha^2) + \beta_* \Bigl(\sqrt{\delta} \{(\alpha-\E[vG])_+^2+\|\proj(v)\|^2\}^{1/2}- \E[Hw]\Bigr), \\
    &\le  \inf_{v,w} \ml(v,w) + \mu_1(\|w\|^2-\alpha^2) + \beta_* (\sqrt{\delta} \|v-\alpha G\| - \E[Hw]) \quad \text{dropping $()_+$}\\
    &= \inf_{v,w} \inf_{\tau_g>0}  \ml(v,w) + \mu_1(\|w\|^2-\alpha^2) + \frac{\beta_*\tau_g}{2} + \frac{\beta_*\delta}{2\tau_g}\|v-\alpha G\|^2  - \beta_*\E[Hw]\\
    &= \inf_{\tau_g>0} \frac{\beta_*\tau_g}{2} + \delta \mathsf{L}(\alpha, \frac{\tau_g}{\beta_*}) + \inf_{w} \E\bigl[\reg(w+X)-\reg(X) - \beta_* w H\bigr] + \mu_1(\|w\|^2-\alpha^2)
\end{align*}
Note that for all $w\in\mh_X$ and $\tau_h>0$, we have 
\begin{align*}
\E[\reg(w+X) - \reg(X) - \beta_*w H] &=  - \frac{\tau_h\|w\|^2}{2\alpha} - \frac{\alpha\beta_*^2}{2\tau_h} + \E\Bigl[
\frac{\tau_h}{\alpha} (\frac{\alpha\beta_*}{\tau_h} H- w)^2 + \reg(w+X) - \reg(X)
\Bigr].
\end{align*}
Let us parametrize $w$ by $\tau_h>0$ as
\begin{align}\label{eq:w_tau_h}
    w(\tau_h) := \prox\Bigl[\frac{\alpha}{\tau_h}\reg\Bigr] \Bigl(X+\frac{\alpha\beta_*}{\tau_h} H\Bigr) - X \in \mh_X. 
\end{align}
Then we have 
$$
\E[\reg(w(\tau_h)+X) - \reg(X) - \beta_* w(\tau_h) H] =  - \frac{\tau_h\|w(\tau_h)\|^2}{2\alpha} - \frac{\alpha\beta_*^2}{2\tau_h} + \mathsf{R}(\frac{\alpha\beta_*}{\tau_h}, \frac{\alpha}{\tau_h})
$$
for all $\tau_h>0$. Combining this and the upper bound of $\text{\eqref{eq:M1_alpha}}$, we get 
$$
\text{\eqref{eq:M1_alpha}} \le
\inf_{\tau_g>0} \frac{\beta_*\tau_g}{2} + \delta \mathsf{L}(\alpha, \frac{\tau_g}{\beta_*}) 
- \frac{\tau_h\|w(\tau_h)\|^2}{2\alpha} - \frac{\alpha\beta_*^2}{2\tau_h} + \mathsf{R}(\frac{\alpha\beta_*}{\tau_h}, \frac{\alpha}{\tau_h}) +  \mu_1(\|w(\tau_h)\|^2-\alpha^2)
$$
Suppose that there exists a $\tau_h^\star >0$  such that $\|w(\tau_h^\star )\|^2=\alpha$. Then, taking $\tau_h=\tau_h^\star$ in the last display, 
\begin{align*}
\text{\eqref{eq:M1_alpha}}& \le \inf_{\tau_g>0} \frac{\beta_*\tau_g}{2} + \delta \mathsf{L}\Bigl(\alpha, \frac{\tau_g}{\beta_*}\Bigr) 
- \frac{\tau_h^\star \alpha}{2} - \frac{\alpha\beta_*^2}{2\tau_h^\star} + \mathsf{R}\Bigl(\frac{\alpha\beta}{\tau_h^\star}, \frac{\alpha}{\tau_h^\star}\Bigr) \le \mathsf{M}(\alpha),
\end{align*}
which concludes the proof. Therefore, it suffices to show the existence of $\tau_h=\tau_h^\star>0$ such that $w(\tau_h)$ defined in \eqref{eq:w_tau_h} satisfies $\|w(\tau_h)\|^2=\alpha^2$. Let us consider the concave function $\mathsf{N}(\tau_h)$:
$$
\mathsf{N}: \R_{>0} \to \R, \quad \tau_h \mapsto \min_{w\in\mh_X} \E\Bigl[\reg(w+X) - \reg(X) - \beta_*w H + \frac{\tau_h w^2}{2\alpha}\Bigr] -  \frac{\tau_h \alpha}{2}
$$
Since the minimization is achieved at $w=w(\tau_h)$, the derivative of $\mathsf{N}(\tau_h)$ is given by
\[
\partial_{\tau_h} \mathsf{N}(\tau_h)
 = \frac{1}{2\alpha} (\|w(\tau_h)\|^2 - \alpha^2).
\]
Furthermore, if we take $w=0$ in the minimization, we get the upper estimate $\mathsf{N}(\tau_h) \le - \frac{\tau_h\alpha}{2}$, which in particular means $\mathsf{N}(\tau_h) \to - \infty$ as $\tau_h\to + \infty$. Thus, if the derivative $\partial_{\tau_h} \mathsf{N}(\tau_h)
= \frac{1}{2\alpha} (\|w(\tau_h)\|^2 - \alpha^2)$ is positive for sufficiently small $\tau_h>0$, this means that there exists a finite stationary point $\tau_h^\star$ satisfying $\|w(\tau_h^\star)\|^2 = \alpha^2$. 

Below we show that the scaled norm $\tau_h \|w(\tau_h)\|$ converges to a positive value {as $\tau_h\to 0$}. 
Let us rewrite $w(\tau_h) = p(\epsilon):=\prox[\epsilon^{-1} \reg] (X+\epsilon^{-1}\beta_* H) - X$ by the re-parametrization $\tau_h \mapsto \epsilon = \alpha \tau_h$. Note
\begin{align*}
& p(\epsilon) \in \argmin_{w\in \R} 2^{-1}{(\epsilon^{-1}\beta_* H-w)^2} + \epsilon^{-1}\reg(w+X)
= \argmin_{w\in \R} 2^{-1} {(\beta_* H- \epsilon w)^2} + \epsilon \reg(w+X).
\end{align*}
Thus, we have
$$
\epsilon p(\epsilon) \in \argmin_{v\in\R} 2^{-1}(\beta_* H - v)^2 + \reg_\epsilon(v) = \prox[\reg_\epsilon](\beta_* H) \quad \text{ with } \quad \reg_\epsilon (\cdot) = \epsilon \reg(\cdot/\epsilon +X).
$$
Let us take $d_{-}$ and $d_{+}$ such that
$$
\lim_{t\to 0+} t \reg(1/t) \to d_+\quad  \text{and} \quad \lim_{t\to 0-} t \reg(1/t) \to d_{-} \quad \text{for some $d_-, d_+\in\R$}.
$$
Note that such $d_{-}$ and $d_{+}$ exist because $(\reg(x)-\reg(0))/(x-0)$ is monotone and bounded due to the convexity and Lipschitz condition of $\reg$. By the Lipschitz condition of $\reg$, $\reg_\epsilon$ converges to the quantile function $\reg_0$ defined below pointwisely:
$$
\reg_\epsilon(x) \to d_{+} x I\{x\ge 0\} + d_{-} x I\{x<0\} =: \reg_0(x). 
$$
Note that the map $v\mapsto 2^{-1}(\beta_* H - v)^2 + g(v)$ is convex and coercive for $g= \reg_\epsilon$ and $g=\reg_0$, so the pointwise convergence $\reg_\epsilon\to \reg_0$ implies the convergence of the corresponding minimizer, namely, $\epsilon p(\epsilon) \to \prox[\reg_0](\beta_* H)$. Noting $|\epsilon p(\epsilon)| \le \beta_* H + \|\reg_\epsilon\|_{\lip} = \beta_* H + \|\reg\|_{\lip}$, the dominated convergence theorem gives us that $\|\epsilon p(\epsilon)\|\to \|\prox[\reg_0](\beta_*H)\|$. Now we suppose the limit is $0$. Then, we must have $\prox[\reg_0](\beta_*H)=0$ pointwisely, which gives $\beta_* H \in \partial \reg_0(0)$ and $\beta_*|H|\le \|\reg_0\|_{\lip} <+\infty$, a contradiction with the unboundedness of $H\sim N(0,1)$. This finishes the proof of $\epsilon\|p(\epsilon)\| \to \|\prox[\reg_0](\beta_*H)\| > 0$. Recalling $\epsilon=\alpha\tau_h$ and $p(\epsilon) = w(\tau_h)$, this means  $\|w(\tau_h)\| \sim \frac{c}{\tau_h}$ as $\tau_h\to 0^{+}$ for a positive constant $c>0$ and hence $\partial_{\tau_h}\mathsf{N}(\tau_h) = \frac{1}{2\alpha}(\|w(\tau_h)\|^2 - \alpha^2)\to +\infty$ as $\tau_h\to 0^+$. 
\end{proof}

\begin{lemma}
    $\mathsf{M}(\alpha)$ is minimized at $\alpha_*\ge 0$ if and only if there exists a minimizer $(v_*, w_*)$ of 
    $$
    \min_{v,w\in\mh} \ml(v,w) \quad \text{subject to} \quad \mg(v,w)\le 0
    $$
    with $\|w_*\|=\alpha_*>0$.
\end{lemma}
\begin{proof}
By \Cref{lm:M1_M2_equivalence} and \Cref{lm:M_M1_equivalence}, $\mathsf{M}(\alpha)$ can be represented as
\begin{align*}
    \mathsf{M}(\alpha) = \min_{w:\|w\| \le \alpha}\min_{v\in \mathcal{V}(\alpha, w)} \ml(v,w)
\end{align*}
where $\mathcal{V} (\alpha, w)$ is the subset of $\mh_{Z}$ given by 
$$
\mathcal{V} (\alpha, w) := \left\{v\in\mh_{Z}:
   \{(\alpha - \E[vG])_+^2 + \|\proj(v)\|^2\}^{1/2} \le {\E[Hw]}/{\sqrt{\delta}}
\right\}
$$
By switching the order of $\min_{w:\|w\|\le \alpha}$ and $\min_{\alpha\ge 0}$, we have
$$
\min_{\alpha\ge 0} \mathsf{M}(\alpha) = \min_{w} \min_{\alpha \ge \|w\|} \min_{v\in\mathcal{V}(\alpha, w)} \ml(v,w)
$$
Since the sequence $\{\mathcal{V}(\alpha, w)\}_{\alpha\ge 0}$ is decreasing in the sense that $
\mathcal{V}(\alpha, w)\supset \mathcal{V}(\alpha', w)$ for all $\alpha \le \alpha'$, 
the map $[\|w\|,+\infty)\ni \alpha \mapsto \min_{v \in \mathcal{V}(\alpha, w)} \ml(v,w)$ is increasing for each $w$.  Therefore, the optimal $\alpha_*$ is given by $\alpha_*=\|w\|$ for each $w$, and substituting this to the last display gives us that
$$
    \min_{\alpha\ge 0} \mathsf{M}(\alpha) = \min_{v,w\in\mh} \ml(v,w) \quad \text{subject to} \quad  \{(\|w\| - \E[vG])_+^2 + \|\proj(v)\|^2\}^{1/2} \le {\E[Hw]}/{\sqrt{\delta}}.
$$
This concludes the proof. 
\end{proof}

\subsection{Proof of \Cref{lm:equivalence_system_reg}}\label{proof:lm:equivalence_system_reg}

\begin{lemma}\label{lm:remove_plus}
   Let $(v_*,w_*)$ be a solution to the optimization
    $$
    \min_{v,w\in\mh} \ml(v,w) \quad \text{subject to} \quad \mg(v,w)\le 0.
    $$
    If $\E[v_* G] > \|w_*\|$, then we can find another $v_{**}$ satisfying $\E[v_{**}G]=\|w_{**}\|$ 
    such that $(v_{**}, w_*)$ also solves the above optimization problem.
\end{lemma}
\begin{proof}
\Cref{lm:vG_decrease} with $t=\arccos(\|w_*\|/\E[v_*G])$ implies that there exists a ${v}_{**} \in\mh_{Z}$ such that
$$
\E[{v}_{**}{G}]=\|w_*\|, \quad \|\proj(v_{**})\|\le\|\proj(v_{**})\|,\quad \E[\loss(v_{**}+Z)-\loss(v+Z)]\le0
$$
Note that $\ml(v_{**},w_{*})\le \ml(v_{*}, w_*)$ thanks to the third inequality and 
\begin{align*}
    \mg(v_{**}, w_{*}) &= \|\proj(v_{**})\|-\E[Hw_{*}]/\sqrt{\delta}\\
    &\le \|\proj(v_{*})\| - \E[Hw_{*}]/\sqrt{\delta}\\
    &\le \sqrt{(\|w_{*}\|-\E[v_{*}G])_+^2 + \|\proj(v_{*})\|^2} - \E[Hw_*]/\sqrt{\delta}\\
    &=\mg(v_*,w_*) \le 0.
\end{align*}
Therefore, $(v_{**}, w_*)$ satisfies the condition in the optimization problem. 
\end{proof}

\begin{lemma}\label{lm:t_positive}
Suppose that one of the following conditions is satisfied:
\begin{itemize}
\item $\reg$ is differentiable
\item $\reg$ has a finite number of non-differentiable points and $X$ is unbounded in the sense that $\PP(|X|>M)>0$ for all $M>0$. 
\end{itemize}
Then, for any solution $(v_*, w_*)$ to the optimization problem $\min_{\mg(v,w)\le 0}\ml(v,w)$, if $\|w_*\| > 0$ and $\E[v_* G]\le \|w_*\|$, we must have $\mt(v_*, w_*) = \bigl\{(\|w_*\|-\E[v_*G])_+^2 + \E[\proj(v_*)^2]\bigr\}^{1/2} > 0$. 
\end{lemma}
\begin{proof}
We proceed by contradiction. Suppose $\mt(v_*, w_*)=0$. Combined with the assumption $\E[v_*G]\le \|w_*\|$, this gives $v_*= \|w_*\| G$. Furthermore, since the constraint $\mg\le 0$ is binding (\Cref{lm:lagrange_reg}), we have 
 $\mg(v_*, w_*)=\mt(v_*, w_*)-\E[H w_*]/\sqrt{\delta}=0$ and $\E[w_* H]=0$. Below, we denote $\alpha_{*} = \|w_{*}\|>0$ so that $v_{*}=\alpha_{*}G$.

{Step 1, Identify ${w_{*}}$:}
Let $\mathcal{H}_X^\perp =\{w\in \mh_X: \E[wH]=0\}$ be the orthogonal complement of $\mh_X$. Note that $w_*\in \mathcal{H}_X^\perp$ and $\mg(\|w\|G,w)=0$ for all $w\in\mh_X^\perp$. Since $(v_*, w_*)=(\|w_*\| G, w_*)$ solves $\min_{\mg(v,w)\le 0} \ml(v,w)$, we have 
\begin{equation}\label{eq:w_*_optimality}
w_{*} \in  \argmin_{w\in\mh_X^\perp} \ml(\|w\|G, w) = \delta\E[\loss(\|w\| G+Z)-\loss(Z)] + \E[\reg(w+X)-\reg(X)]. 
\end{equation}
Now we claim that $w\mapsto \E[\loss(\|w\| G+Z)-\loss(Z)]$ is convex and differentiable at $w_{*}$. Note that it is the composition of the two convex maps $f_2\circ f_1(w)$ where 
$$
f_1:\mh_X^\perp \to\R_{\ge 0}, \ w\mapsto \|w\|, \qquad f_2: \R_{\ge 0}\to \R, \  x \mapsto \E[\loss(xG+Z)-\loss(Z)].
$$
Then, $w\mapsto \E[\loss(\|w\| G+Z)-\loss(Z)]$ is convex since 
$f_2$ is nondecreasing
by conditional Jensen's inequality. For the differentiability, $f_1$ is differentiable at $w\ne 0$, while $f_2$ is differentiable at $x>0$ since $xG+Z$ has a continuous distribution and $\loss$ is Lipschitz. Thus, $w\mapsto \E[\loss(\|w\| G+Z)-\loss(Z)]$ is differentiable at 
$w_{*}(\ne 0)$ and the derivative at $w_{*}$ is given by
$$
\E\bigl[G \loss'(\|w_{*}\|G+Z)\bigr] \frac{w_{*}}{\|w_{*}\|} = \alpha_*^{-1} \E[G\loss'(\alpha_* G +Z)] w_{*}.
$$
Then, the first-order optimality condition for $w_{*}$ in \eqref{eq:w_*_optimality} gives
\begin{equation}\label{eq:xi_*}
- \xi_{*} w_{*} \in \partial \reg(w_{*}+X)  \quad \text{where}\quad 
\xi_{*} = \delta \alpha_*^{-1} \E[G\loss'(\alpha_* G +Z)]
\end{equation}
Now we verify that $\xi_{*}$ is strictly positive. 
For $\tilde{G}\sim N(0,1)$ being independent of $G$ and $X$, 
\begin{align*}
\delta^{-1}\alpha_* \xi_{*} =\E[G \loss'(\|w_{*}\| G + Z)]= 2^{-1} \E\bigl[\bigl\{\loss'(\alpha_{*} G + X) - \loss'(\alpha_{*} \tilde{G} + X)\bigr\}(G-\tilde{G})\bigr] =: 2^{-1}\E[D],
\end{align*}
where $D =\{\loss'(\alpha_{*} G + X) - \loss'(\alpha_{*} \tilde{G} + X)\}(G-\tilde{G})$ is nonnegative since $x\mapsto \loss'(\alpha_{*} x + X)$ is nondecreasing. Furthermore, $D$ is strictly positive on the event $\{\alpha_{*} G + X > c\} \cap \{\alpha_{*} \tilde{G} + X < -c\}$ for some $c>0$ such that ${\loss}'(-c) < 0 < {\loss}'(c)$, and we can always find such $c>0$ because $\loss$ is convex and $\{0\}=\argmin_x\loss(x)$.  
Since this event has a positive probability, we must have $\E[D] > 0$ and hence $\xi_{*} = 2^{-1}\delta \alpha_{*}^{-1}\E[D] > 0$. Multiplying the both sides in \eqref{eq:xi_*} by $\xi_{*}>0$, we have $-w_{*}\in \xi_{*}^{-1}\partial \reg(w_{*}+X)$, or equivalently,
$$
w_{*} = \prox[\xi_{*}^{-1}\reg](X) - X
$$
and $w_*$ is independent of $H$. 

{Step 2, Obtain a contradiction:}
First, we claim that the event 
$$
\Omega=\{\text{$\reg$ is differentiable at $w_{*}+X$}\}
$$ has a positive probability. 
Note that $|w_{*} + X| \ge |X|-|w_{*}| \ge  |X|-\xi_{*}^{-1} \|\reg\|_{\lip}$ by 
$w_{*} = \prox[\xi_{*}^{-1}\reg](X) - X$, while $\PP(|X|>M)>0$ for any $M>0$ and the set of nondifferentiable points of $\reg$ is finite by the assumption. Thus, taking $M$ larger than
$
(\xi_{*}^{-1} \|\reg\|_{\lip} + \max \{|x|: \text{$\reg$ is not differentiable at $x$}\})
$
gives an event of positive probability for which $\reg$ is continuous. 
Letting $p=\PP(\Omega)>0$ for this event, we define $(v_t, w_t)$ as
$$
v_t := \prox[t\loss][atG+v_{*} + Z] - (Z+v_{*}) \quad \text{with} \quad a=\E[\loss'(\alpha_*G+Z)], \qquad w_t := t b H\tfrac{I\{\Omega\}}{p}
$$
for all $t>0$, where $b$ is a nonnegative constant specified later.  
Note that $w_{*}$ and $I\{\Omega\}$ are independent of $H$. 
Now verify that $\ml(v_t+v_{*}, w_t+w_{*})-\ml(v_{*}, w_{*})<0$ and $\mg(v_t+v_{*}, w_t+w_{*})<0$ for a sufficiently small $t>0$. 
For the term involving the regularizer $\reg$, since $\reg$ is differentiable at $w_{*}+X$ on the event $\Omega$, the dominated convergence theorem gives
\begin{align*}
&\E[\reg(w_{t}+w_{*}+X) - \reg(w_{*}+X)]\\
&=\E\bigl[I\{\Omega\}\bigl(\reg(w_{*}+X+tbH/p)-\reg(w_{*}+X)\bigr)\bigr]&&\text{$w_t$ is $0$ on the event $\Omega^c$}\\
&=tbp^{-1}\E\bigl[I\{\Omega\} \reg'(w_{*}+X)H\bigr] + o(t) &&\text{by dominated convergence}\\
&= tbp^{-1} \cdot 0 + o(t) = o(t) &&\text{by the independence of $(H,X)$}. 
\end{align*}
For the perturbation involving $\loss$,  
noting $t^{-1}(taG-v_t) \in \partial \loss(v_t+v_*+Z)$, we have 
\begin{align*}
    \loss(v_t+v_*+Z)-\loss(v_*+Z) \le (aG - \tfrac{v_t}{t})v_t,
\end{align*}
and $t^{-1} v_t \to aG-\loss'(\alpha_* G+Z)$, almost surely.
Since $|t^{-1}v_t|$ is uniformly bounded by the square-integrable random variable $|aG| + \|\loss\|_{\lip}$, 
the dominated convergence theorem yields
\begin{align*}
    \E[G v_t/t] \to a- \E[G\loss'(\alpha_*G+Z)] = 0, \quad \|\proj(v_t/t)\|^2 \to \|\proj(\loss'(\alpha_* G+Z))\|^2.
\end{align*}
Therefore, we have 
$$
\E[\loss(v_t+v_*+Z)-\loss(v_*+W_*)] \le t\E[(\alpha_*G-\tfrac{v_t}{t})\tfrac{v_t}{t}] = - t \|\proj(\loss'(\alpha_* G+Z))\|^2 + o(t). 
$$
Putting the above displays together, we observe that the increment of the objective function is bounded from above as
\begin{align*}
\ml(v_{*}+v_t, w_t+w_{*})-\ml(v_{*}, w_{*}) &\le -t \delta \|\proj(\loss'(\alpha_* G+Z))\|^2 + o(t).
\end{align*}
For the constraint term, noting that $v_{*}=\alpha_{*} G$, $\|w_{*}\|=\alpha_{*}>0$, and $w_{*}$ is independent of $H$, and using $\E[v_tG]=o(t)$ and $\E[\proj(v_t)^2]=t^2 \E[|\proj(\loss'(\alpha_*G+Z))^2]+ o(t^2)$, we have
\begin{align*}
&\mg(v_t+v_{*}, w_t+w_{*})\\
&= \bigl\{(\|w_t+w_{*}\|-\|w_{*}\|- \E[v_t G])_+^2 + \E[\proj(v_t)^2]\bigr\}^{1/2} -\delta^{-1/2}\E[Hw_t]\\
&= \Bigl\{(\sqrt{\alpha_{*}^2 + t^2 b^2/p}-\alpha_{*} - o(t) )_+^2 + t^2 \|\proj(\loss'(\alpha_* G+Z))\|^2 + o(t^2)\Bigr\}^{1/2} - t {\delta}^{-1/2} b \\
&= \bigl\{
    o(t^2) + t^2 \|\proj(\loss'(\alpha_* G+Z))\|^2
\bigr\}^{1/2} - t\delta^{-1/2} b\\
&=  o(t) + t \bigl(\|\proj(\loss'(\alpha_* G+Z))\| - \delta^{-1/2} b\bigr).
\end{align*}
Thus, if we take $b=2\sqrt{\delta}\|\proj(\loss'(\alpha_* G+Z))\|$, we have $\mg(v_t+v_{*}, w_t+w_{*}) = -t \|\proj(\loss'(\alpha_* G+Z))\| + o(t)$. In summary, we have shown that
\begin{align*}
\ml(v_{*}+v_t, w_t+w_{*})&\le \ml(v_{*}, w_{*}) -t \delta \|\proj(\loss'(\alpha_* G+Z))\|^2 + o(t),\\ \mg(v_t+v_{*}, w_t+w_{*})&\le -t \|\proj(\loss'(\alpha_* G+Z))\| + o(t).
\end{align*}
as $t\to 0+$. Now we verify that $\|\proj(\loss'(\alpha_* G+Z))\|$ is strictly positive. If $\|\proj(\loss'(\alpha_* G+Z))\|=0$, then there exists some real $s$ such that  $\loss'(\alpha_*G+Z)=s G$, but we must have $s=0$ since $\loss'(\alpha_*G+Z)$ is bounded while $G$ is unbounded. This in turn gives $\loss'(\alpha_*G+Z)=0$, a contradiction since the support of $\alpha_{*}G +Z$ is $\R$ while $\loss$ is not constant. Thus, $\|\proj(\loss'(\alpha_* G+Z))\|$ is strictly positive. Then, we can find a sufficiently small $t=t'>0$ such that $ \ml(v_{*}+v_{t'}, w_{t'}+w_{*}) < \ml(v_{*}, w_{*}) $ and $\mg(v_{*}+v_{t'}, w_{t'}+w_{*}) <0$. This is a contradiction with the fact that $(v_*, w_*)$ is a minimizer. This concludes the proof.
\end{proof}

\begin{lemma}\label{lm:remove_plus_strict}
    If $(v_*, w_*)$ solves \eqref{eq:const_optim_reg} with $w_*\ne 0$, then we must have $\E[w_* G] < \|w_*\|$. 
\end{lemma}
\begin{proof}
We proceed by contradiction. Suppose $\E[w_* G] \ge \|w_*\|$. Then, by \Cref{lm:remove_plus}, we can find some $v_{**}\in\mh_Z$ such that $\E[v_{**}G]=\|w_*\|$ and $(v_{**}, w_*)$ is a minimizer. Then, we must have $\mt(v_{**}, w_{*})>0$ by  \Cref{lm:t_positive}. 
Let us take a Lagrange multiplier $\mu_*>0$ satisfying \eqref{eq:KKT_reg} so that 
$$
-\mu_* \partial_w \mg(v_{**},w_{*}) \cap \partial_w \ml(v_{**}, w_{*}) \ne \emptyset. 
$$
Here, by $\mt(v_*, w_*)>0$ and \Cref{lm:lg_basic_reg}, $(v,w)\mapsto \mg(v,w)$ is differentiable at $(v_{**}, w_{*})$ with the derivative with respect to $w$ given by 
$$
\nabla_{w}\mg(v_{**}, w_{*}) = \frac{\bigl(1-\|w_*\|^{-1}{\E[v_{**}G]}\bigr)}{{\mt(v_{**}, w_*)}}  w_* - \delta^{-1/2} H = -\delta^{-1/2} H,
$$
where we have used $\E[v_{**}G]=\|w_*\|$. Combined with the previous display, we are left with
$$
\mu_*  \delta^{-1/2} H \in \partial \ml_{w} (v_{**}, w_{*})
$$
However, since $\mu_*>0$ and $\reg$ is Lipschitz, this gives $|H|\le \mu_*^{-1} \delta^{1/2}\|\reg\|_{\lip}$ with probability $1$, but this is a contradiction since $H=^d N(0,1)$ is not bounded. This concludes that $\E[w_* G] < \|w_*\|$ must hold. 
\end{proof}

\begin{lemma}\label{lm:positive}
    If $(v_*, w_*)$ solves $\min_{\mg(v,w)\le 0}\ml(v,w)$ with $w_*\ne 0$, then it always satisfies 
    $$
    \E[v_* G] < \|w_*\|, \quad \mt(v_*, w_*) >0 \quad \text{and} \quad \|\proj(v_*)\|>0.
    $$
\end{lemma}

\begin{proof}
    $ \E[v_* G] < \|w_*\|$ and $\mt(v_*, w_*) >0$ immediately follows from 
\Cref{lm:t_positive} and \Cref{lm:remove_plus_strict}. Let us take a Lagrange multiplier $\mu_*>0$ satisfying \eqref{eq:KKT_reg} so that 
$$
-\mu_* \partial_v \mg(v_{*},w_{*}) \cap \partial_v \ml(v_{*}, w_{*}) \ne \emptyset. 
$$
By \Cref{lm:lg_basic_reg}, $v\mapsto \mg(v, w_*)$ is differentiable at $v=v_*$ with the derivative given by 
$$
\nabla_v\mg(v_*, w_*) 
= \frac{-(\|w_*\|-\E[v_*G]) G + v_* - \E[v_*G] G}{\mt(v_*,w_*)} = \frac{-\|w_*\|G + v_*}{\mt(v_*,w_*)}. 
$$
Since $v\mapsto \ml(v, w_*)$ is $\delta \|\loss\|_{\lip}$-Lipschitz, the above two displays give 
$$
|-\|w_*\| G + v_*| \le \mu_*^{-1}\mt(v_*, w_*) \delta \|\loss\|_{\lip}
$$
with probability $1$. Suppose $\|\proj(v_*)\|=0$. Then, there exists some $t\in\R$ such that $v_*=tG$. Substituting this to the previous display, we are left with $|t-\|w_*\|| \cdot |G| \le \mu_*^{-1}\mt(v_*, w_*) \delta \|\loss\|_{\lip}$. Since $G\sim N(0,1)$ is unbounded, we must have $t=\|w_*\|$ and $v_*=\|w_*\|G$. This in turn gives $\mt(v_*, w_*)=0$, a contradiction with $\mt(v_*, w_*)>0$. Thus, we must have $\|\proj(v)\|>0$.
\end{proof}

\begin{lemma}\label{lm:optimization_to_system_reg}
    Suppose that $(v_*, w_*)$ solves \eqref{eq:const_optim_reg} with $\|w_*\| >  0$. Let us take a 
     Lagrange multiplier satisfying the KKT condition \eqref{eq:KKT_reg} and define the positive scalars as follows 
    \begin{equation}\label{eq:scalar_*_reg}
        \alpha_* = \|w_*\|, \quad \beta_* = \frac{\mu_*}{\sqrt{\delta}}, \quad \kappa_* = \frac{\delta}{\mu_*}\mt(v_*, w_*), \quad \nu_* = \mu_* \frac{1-\|w_*\|^{-1}\E[v_*G]}{\mt(v_*, w_*)}
    \end{equation}
    (note that the positiveness of $\kappa_*$ and $\nu_*$ follows from \Cref{lm:positive}). 
    Then, $v_{*}$ and $w_{*}$ take the form of 
    $$v_* = \prox[\kappa_* \loss](\alpha_* G+Z)-Z\quad  \text{ and } \quad w_* = \prox[\nu_*^{-1} \reg](\nu_*^{-1}\beta_* H+X)-X
    $$
    and the above $(\alpha_{*}, \beta_*, \kappa_*, \nu_*)$ solve the nonlinear system of equations \eqref{eq:system_reg}; below we restate the system for convenience. 
\begin{align}
\alpha^2 &= \E \bigl[
    \bigl(
        \prox[\nu^{-1} \reg] (\nu^{-1}\beta H + X) - X
    \bigr)^2
    \bigr], \label{eq:system_reg_1}\\
    \beta^2 \kappa^2 &=  \delta  \cdot \E\bigl[\bigl(\alpha G + Z - \prox[\kappa\loss](\alpha G + Z)\bigr)^2\bigr], \label{eq:system_reg_2}
    \\
    \nu\alpha\kappa &= \delta \cdot \E\bigl[G\cdot \bigl(\alpha G + Z - \prox[\kappa](\alpha G + Z)\bigr)\bigr], \label{eq:system_reg_3}
    \\
    \kappa\beta &= \E\bigl[H \cdot \bigl(\prox[\nu^{-1}\reg] (\nu^{-1}\beta H + X)-X\bigr)\bigr]. \label{eq:system_reg_4}
  \end{align}
\end{lemma}
\begin{proof}
    As we have shown in the proof of \Cref{lm:remove_plus_strict} and \Cref{lm:positive}, the map $(v, w)\mapsto \mg(v,w)$ is Fr\'echet differentiable at $(v_*, w_*)$ with the derivatives given by 
    \begin{align*}
        \nabla_v\mg(v_*, w_*) 
        = \frac{-\|w_*\|G + v_*}{\mt(v_*,w_*)}, \quad 
        \nabla_w\mg(v_*,w_*) 
        = \frac{\bigl(1-\|w_*\|^{-1}{\E[v_*G]}\bigr)}{{\mt(v_*, w_*)}}  w_* - \delta^{-1/2} H. 
      \end{align*}
    Thus, by the definition of $(\alpha_*, \beta_*, \kappa_*, \nu_*)$, the KKT condition $-\mu_*\partial \mg(v_*, w_*)\cap \partial\ml(v_*, w_*)\ne 0$ in \eqref{eq:KKT_reg} reads to 
  \begin{align*}
    -(\delta/\kappa_*) \cdot (-\alpha_*G+v_*) \in \partial_v \ml(v_*,w_*), \quad - (\nu_*w_* -\beta_* H) \in \partial_w \ml(v_*,w_*).
  \end{align*}
    This implies that $(v_*, w_*)$ minimizes the convex function
  \begin{align*}
    \mh\ni (v,w)&\mapsto  \ml(v,w) + \frac{\delta}{2\kappa_*} \E\bigl[(v-\alpha_*G)^2\bigr] + \frac{\nu_*}{2}\E\Bigl[(w-\frac{\beta_*}{\nu_*} H)^2\Bigr]\\
    &=\delta \E\Bigl[
        \loss(v+Z)-\loss(Z) + \frac{\delta}{2\kappa_*}(v-\alpha_*G)^2
    \Bigr] + \E\Bigl[
        \reg(w+X)-\reg(X) + \frac{\nu_*}{2}(w-\frac{\beta_*}{\nu_*}H)^2
    \Bigr], 
  \end{align*}
  and hence we obtain the form of $(v_*, w_*)$:
    $$v_* = \prox[\kappa_* \loss](\alpha_* G+Z)-Z \quad  \text{ and } \quad w_* = \prox[\nu_*^{-1} \reg](\nu_*^{-1}\beta_* H+X)-X.
    $$
Let us prove that the positive scalars $(\alpha_*,\beta_*, \kappa_*, \nu_*)$ defined in \eqref{eq:scalar_*_reg} solve the nonlinear system of equations. 
The first equation \eqref{eq:system_reg_1} immediately follows from 
$\alpha_* = \|w_*\|$. 
 For the fourth \eqref{eq:system_reg_4}, the definition $\beta_*=\mu_*\delta^{-1/2}$, $\kappa_*=\mu_*^{-1}\delta \mt(v_*, w_*)$ and the binding condition $\mg(v_*, w_*)=0$ lead to 
\begin{align*}
  \kappa_*\beta_* = \delta^{1/2} \mt(v_*,w_*) = \E[w_*H]
\end{align*}
so that \eqref{eq:system_reg_4} is satisfied. Using $\kappa_*\beta_* = \delta^{1/2} \mt(v_*,w_*)$ in the above display, since $\mt(v,w)=\|\|w\| G - v\|$ when $\E[vG]\le \|w\|$ and $\|w_*\|=\alpha_*$, we have $\kappa_* \beta_* = \delta^{1/2} \mt(v_*,w_*) = \delta^{1/2} \|\alpha_* G - v_*\|$ so that the second equation \eqref{eq:system_reg_2} is satisfied. Finally, for the third equation \eqref{eq:system_reg_3}, it holds that 
\begin{align*}
  \nu_*\alpha_*\kappa_* &= \delta \alpha_* (1 - \|w_*\|^{-1} \E[v_*G]) && \text{ by $\kappa_* = \mu_*^{-1} \delta \mt(v_*, w_*)$ and $\nu_* = \mu_* \frac{1- \|w_*\|^{-1}\E[v_*G]}{\mt(v_*,w_*)},$}\\
  &= \delta (\alpha_* - \E[v_*G]) &&\text{ by $\|w_*\|=\alpha_*$}\\
  &= \delta \E[G(\alpha_* G-v_*)] && \text{ by $\E[G^2] = 1$}
\end{align*}
so that \eqref{eq:system_reg_3} is satisfied.  This finishes the proof. 
\end{proof}

\begin{lemma}\label{lm:system_to_optimization_reg}
    If $(\alpha_*, \beta_*, \kappa_*, \nu_*)\in\R_{>0}^4$ is a solution to the nonlinear system of equations \eqref{eq:system_reg_1}-\eqref{eq:system_reg_4}, then $(v_*, w_*)$ defined by 
\begin{align*}
    v_* &= \prox[\kappa_*\loss](\alpha_* G +Z)-Z, &&w_*=\prox[\nu_*^{-1}\reg] (\nu_*^{-1}\beta_* H +X)-X
\end{align*}
satisfies $\|w_*\|=\alpha_*>0$ and the KKT condition \eqref{eq:KKT_reg} with $\mu_{*}=\beta_*\sqrt{\delta}$, so that $(v_*, w_*)$ solves $\min_{\mg(v,w)\le 0}\ml(v,w)$ with $w_*\ne 0$. 
\end{lemma}
\begin{proof}
By the definition of $v_*$ and $w_*$, the nonlinear system of equations \eqref{eq:system_reg_1}-\eqref{eq:system_reg_4} reads to 
\begin{align*}
    \alpha_*^2  = \E[w_*^2],  \quad \beta_*^2\kappa_*^2 = \delta \E[(\alpha_* G - v_*)^2], \quad  \kappa_*\beta_* = \E[H w_*], \quad \nu_*\alpha_*\kappa_* = \delta \E[G (\alpha_* G-v_*)].
  \end{align*}
Then, it easily follows from the above display that 
  \begin{equation*}
 \|w_*\| > 0, \quad 1 - \|w_*\|^{-1} \E[v_*G]=\frac{\nu_*\kappa_*}{\delta}>0, \quad 
  \mt(v_*, w_*)=\frac{\beta_*\kappa_*}{\sqrt{\delta}} \quad \text{and} \quad \mg(v_*, w_*)=0. 
  \end{equation*}
Recall that the map $(v,w)\mapsto\mg(v,w)$ is differentiable if  $\mt(v, w)>0$ and $\|w\|>0$ (see \Cref{lm:lg_basic_reg}). Noting $\|w_*\|>\E[v_*G]$, the derivative formula gives 
\begin{align*}
    -\sqrt{\delta} \beta_*\cdot  \nabla_v\mg(v_*, w_*) 
        &= -\sqrt{\delta} \beta_*  \cdot \frac{-\|w_*\|G + v_*}{\mt(v_*,w_*)} = \delta \kappa_*^{-1} { (\alpha_* G - v_*)}\\
        -\sqrt{\delta}\beta_* \cdot \nabla_w\mg(v_*,w_*) 
        &= -\sqrt{\delta}\beta_* \cdot \Bigl(\frac{\bigl(1-\|w_*\|^{-1}{\E[v_*G]}\bigr)}{{\mt(v_*, w_*)}}  w_* - \delta^{-1/2} H.\Bigr) = -\nu_*w_* + \beta_* H. 
\end{align*}
  Since $(v_*, w_*)$ satisfies $\kappa_*^{-1}  (\alpha_* G-v_*) \in \partial \loss(v_*+Z)$ and ${\beta_*} H - \nu_* w_* \in \partial \reg(w_* + X)$ by the definition of proximal operator, using $\partial \reg(w+X)\subset \partial_w \ml(v,w)$ and $\delta \partial \loss(v+Z) \subset \partial_v \ml(v,w)$, we are left with
  $$
  -\sqrt{\delta}\beta_* \nabla_v \mg(v_*,w_*) \subset \partial \ml(v_*,w_*), \quad   - \sqrt{\delta}\beta_*  \mg(v_*,w_*) \subset \partial_w \ml(v,w). 
  $$
  This means that the KKT condition is satisfied with $\mu_* = \sqrt{\delta}\beta_*>0$. 
\end{proof}

\subsection{Proof of \Cref{lm:unique_reg}}\label{proof:lm:unique_reg}

\begin{lemma}\label{lm:proportional_reg}
Suppose $(v_*, w_*)$ and $(v_{**}, w_{**})$ solve $\min_{\mg(v,w)\le 0}$. If  $w_{*}$ and $w_{**}$ are both nonzero, then they must be proportional, i.e.,  $(v_{**}, w_{**})=t(v_{*}, w_{*})$ for some $t>0$. 
\end{lemma}
\begin{proof}
By \Cref{lm:optimization_to_system_reg},  $(v_*, w_*)$ takes the form of 
$$
v_{*}=\prox[\kappa_*\loss](\alpha_*G+Z)-Z, \quad w_* = \prox[\nu_*^{-1}\reg](\nu_*^{-1}\beta_*H +X)-X.
$$
where $(\alpha_*, \beta_*, \kappa_*, \nu_*)$ is the associated solution to the nonlinear system \eqref{eq:system_reg_1}-\eqref{eq:system_reg_4}. Using the form of $(v_*, w_*)$, the nonlinear system reads to 
\begin{align}\label{eq:4_equation_prox}
    \alpha_*^2  = \E[w_*^2],  \quad \beta_*^2\kappa_*^2 = \delta \E[(\alpha_* G - v_*)^2], \quad  \kappa_*\beta_* = \E[H w_*], \quad \nu_*\alpha_*\kappa_* = \delta \E[G (\alpha_* G-v_*)].    
\end{align}
By the definition of proximal operator, we have $\kappa_*^{-1}(\alpha_* G - v_*)\in \partial\loss(v_*+Z)$ and $\nu_{*}(\nu_*^{-1}\beta_* H - w_*)\in \partial \reg(w_*+X)$, which implies that, 
  pointwise, the scalar realization of $(v_*, w_*)$ minimizes the convex function
\begin{align*}
  (x,y)\mapsto \delta \Bigl[\loss(x+Z) + \frac{{(x-\alpha_* G)^2} - {(x-v_*)^2}}{2\kappa_*}\Bigr] +  \reg(y+X) + \frac{\nu_*}{2}\Bigl\{\Bigl(y-\frac{\beta_*}{\nu_*} H\Bigr)^2 - {(y-w_*)^2}\Bigr\}
\end{align*}
since the derivatives of $x\mapsto (x-v_*)^2$ and $y\mapsto (y-w_*)^2$ are zeros at $(x,y)=(v_*, w_*)$.
Thus, again pointwise, we have
\begin{align*}
  &\delta \Bigl[\loss(v+Z) + \frac{{(v-\alpha_* G)^2} - {(v-v_*)^2}}{2\kappa_*}\Bigr] +  \reg(w+X) + \frac{\nu_*}{2}\Bigl\{\Bigl(w-\frac{\beta_*}{\nu_*} H\Bigr)^2 - {(w-w_*)^2}\Bigr\}\\
  &\ge \delta \Bigl[\loss(v_*+Z) + \frac{{(v_*-\alpha_* G)^2}}{2\kappa_*}\Bigr] +  \reg(w_*+X) + \frac{\nu_*}{2}\Bigl(y-\frac{\beta_*}{\nu_*} H\Bigr)^2
\end{align*}
for all $(v,w)\in \mh$. 
Taking expectation, we are left with 
\begin{align*}
    \ml(v,w) - \ml(v_*, w_*) &= \delta\E[\loss(v+Z)-\loss(v_*+Z)] + \E[\reg(w+X)-\reg(w_*+X)] \\
    &\ge  \frac{\delta}{2\kappa_*} \E\bigl[
        (\alpha_* G-v_*)^2 - (\alpha_* G-v)^2 + (v_*-v)^2
     \bigr] \\
     &+ \frac{\nu_*}{2} \E\Bigl[
        \Bigl(w_*-\frac{\beta_*}{\nu_*} H\Bigr)^2 - \Bigl(w-\frac{\beta_* H}{\nu_*}\Bigr)^2 + (w-w_*)^2
      \Bigr].   
\end{align*} 
Now we simplify the RHS using \eqref{eq:4_equation_prox}. For the first term involving $(v, v_*)$, expanding the square $(v-v_*)^2 = (\alpha_* G - v_* - (\alpha_* G - v))^2$ and using  $\delta \E[(\alpha_* G-v_*)^2]=\beta_*^2\kappa_*^2$ in \eqref{eq:4_equation_prox}, we have
\begin{align*}
    \frac{\delta}{2\kappa_*} \E\bigl[(\alpha_*G-v_*)^2 -(\alpha_* G-v)^2 + (v_*-v)^2\bigr] &= \frac{\delta}{2\kappa_*}\E\bigl[ 2(\alpha_*G-v_*)^2 - 2(\alpha_* G - v_*)(\alpha_*G-v)\bigr]\\
    &= \frac{\delta}{2\kappa_*}\Bigl(2\frac{\beta_*^2\kappa_*^2}{\delta}
    - 2\E[(\alpha_*G - v_*)(\alpha_*G-v)]
    \Bigr)\\
    &=\beta_*^2\kappa_* -  \kappa_{*}^{-1}\delta \E\bigl[(\alpha_* G - v_*)(\alpha_*G-v)\bigr].
\end{align*}
For the second term involving $(w,w_*)$, expanding the squares so that $w^2$ and $(\nu_*^{-1}\beta_*H)^2$ are cancelled out, using $\E[w_*^2]=\alpha_*^2$, $\E[w_* H]=\kappa_*\beta_*$ in \eqref{eq:4_equation_prox}, we have
\begin{align*}
    \frac{\nu_*}{2} \E\Bigl[
    \Bigl(w_*-\frac{\beta_*}{\nu_*} H\Bigr)^2 - \Bigl(w-\frac{\beta_*}{\nu_*}H\Bigr)^2 + (w-w_*)^2
  \Bigr] &= \frac{\nu_*}{2} \E\Bigl[
    2w_*^2 - 2\frac{\beta_*}{\nu_*} w_*H +  2\frac{\beta_*}{\nu_*} wH - 2ww_*
  \Bigr]\\
  &=\nu_*\alpha_*^2 - \kappa_*\beta_*^2 + \beta_*\E[ wH] - \nu_*\E[ww_*].
\end{align*}
Putting the above displays together, noting that $\kappa_*\beta_*^2$ is canceled out, we are left with 
\begin{align}\label{eq:L_Xi_bound}
  \ml(v,w)-\ml(v_*,w_*) \ge - (\delta/\kappa_*) \E\bigl[(\alpha_* G-v_*)(\alpha_* G-v)\bigr]
+ \nu_* \alpha_*^2 - \nu_* \E[w_*w] + \beta_* \E[wH].
\end{align}
for all $(v,w)\in\mh$. Using the equations \eqref{eq:4_equation_prox} and the decomposition of the inner product
\begin{equation}
    \E[v_1 v_2]= \E[\proj(v_1)\proj(v_2)] + \E[Gv_1]\E[Gv_2] \quad \text{with} \quad \proj(v)=v -\E[vG]G, 
    \label{proj_decomposition}
\end{equation}
the terms on the RHS of \eqref{eq:L_Xi_bound} (except for the right-most term $\beta_*\E[wH]$) are bounded from below as 
\begin{align*}
    &\phantom{ {} = } - (\delta/\kappa_*) \E\bigl[(\alpha_* G-v_*)(\alpha_* G-v)\bigr]
+ \nu_* \alpha_*^2 - \nu_* \E[w_*w]\\
&= - (\delta/\kappa_*)\bigl(
    \E[\proj(v_*)\proj(v)]
     +(\alpha_*-\E[Gv_*])
      (\alpha_*-\E[Gv])
      \bigr)
    + \nu_*\alpha_*^2 - \nu_* \E[w_*w]
&& \text{by \eqref{proj_decomposition}}\\
&= - (\delta/\kappa_*) \bigl(
    \E[\proj(v_*)\proj(v)]
     +\alpha_* \nu_*(\kappa_*/\delta)
      (\alpha_*-\E[Gv])
      \bigr)
    + \nu_*\alpha_*^2 - \nu_* \E[w_*w] &&\text{by $\nu_*\alpha_*\kappa_* = \delta\E[G(\alpha_* G - v_*)]$} \\
&= -(\delta/\kappa_*) \bigl(
    \E[\proj(v_*)\proj(v)]
     -\alpha_* \nu_*(\kappa_*/\delta)\E[Gv]
      \bigr) - \nu_*\E[w_*w]\\
&\ge -(\delta/\kappa_*) \bigl(
        \E[\proj(v_*)\proj(v)]
         -\alpha_* \nu_*(\kappa_*/\delta)\E[Gv]
          \bigr) - \nu_*\|w_*\||w\| && \text{CS for $\E[w_*w]\le \|w_*\|\|w\|$}\\
&=-(\delta/\kappa_*) \bigl(
    \E[\proj(v_*)\proj(v)]
     +\alpha_* \nu_*(\kappa_*/\delta) (\|w\|-\E[vG])
      \bigr) && \text{using } \|w_*\|=\alpha_*  \\
& = - (\delta/\kappa_*) \bigl(
    \E[\proj(v_*)\proj(v)]
     + (\|w_*\|-\E[Gv_*])_+(\|w\| - \E[Gv])
      \bigr) && \text{by \eqref{eq:4_equation_prox}} \\
&\ge - (\delta/\kappa_*) \bigl(
    \E[\proj(v_*)\proj(v)]
     + (\|w_*\|-\E[Gv_*])_+(\|w\| - \E[Gv])_+
  \bigr) && \text{using } y_+x_+\ge y_+ x.
\end{align*}
In the right column where justifications are provided,
CS stands for the Cauchy--Schwarz inequality.
Therefore, since $\mt(v, w) = \{(\|w\| - \E[vG])_+^2 + \E[\proj(v)^2]\}^{1/2}$ by definition,
\begin{align*}
    &\phantom{ {} = } \ml(v,w)-\ml(v_*,w_*) \\
     &\ge - (\delta/\kappa_*) \bigl[
    \E[\proj(v_*)\proj(v)]
     + (\|w_*\|-\E[Gv_*])_+(\|w\| - \E[Gv])_+
      \bigr] + \beta_*\E[wH]\\
    &\ge -(\delta/\kappa_*) \bigl(\|\proj(v_*)\|\|\proj(v)\| + (\|w_*\|-\E[Gv_*])_+(\|w\| - \E[Gv])_+ \bigr) + \beta_*\E[wH] &&\text{CS for $\E[\proj(v_*)\proj(v)]$}\\
    &\ge - (\delta/\kappa_*) {\mt(v_*, w_*)}\cdot {\mt(v, w)} +\beta_*\E[wH] &&\text{CS}
    \\
    &= - \beta_*  {\mt(v, w)} + \beta_*\E[wH] &&\text{${\mt(v_*,w_*)}=\beta_*\kappa_*/\sqrt{\delta}$}\\
    &=-\beta_*\mg(v,w). 
\end{align*}
for all $(v,w)\in\mh$.  
Applying this inequality with $(v, w)=(v_{**}, w_{**})$, 
 since $\ml(v_{**}, w_{**})=\ml(v_*, w_*)$ and $\mg(v_{**}, w_{**})=0$, we are left with 
 $$
0 = \ml(v_{**}, w_{**}) -\ml(v_*, w_*)\ge -\beta_*\mg(v_{**}, w_{**}) = 0
 $$
so that the Cauchy--Schwarz inequality holds with equality. The equality case gives that 
\begin{align*}
w_{**} &=tw_{*} &&\text{for some real $t\ge 0$}, \\
\proj(v_{**}) &= u \proj(v_{*}) &&\text{for some real $u\ge 0$}, \\
\begin{pmatrix}
    \|\proj(v_{**})\|\\
    (\|w_{**}\|-\E[G v_{**}])_+ 
\end{pmatrix} &= s \begin{pmatrix}
    \|\proj(v_{*})\|\\
    (\|w_{*}\|-\E[G v_{*}])_+
\end{pmatrix} &&\text{for some real $s\ge 0$}. 
\end{align*}
Since $\|\proj(v_{*})\|>0$ by \Cref{lm:positive}, the second and the third equations give $u=s>0$. Furthermore, we must have $t=s$ since 
\begin{align*}
t \E[H w_*] &=
\E[H w_{**}] &&\text{by $w_{**}=t w_*$}\\
&= \sqrt{\delta} \sqrt{(\|w_{**}\|-\E[v_{**}G])_+^2 + \|\proj(v_{**})\|^2} &&\text{by $\mg(v_{**}, w_{**})=0$}\\
&= s \sqrt{\delta} \sqrt{(\|{w}_*\|-\E[{v}_*G])_+^2 + \|\proj({v}_*)\|^2} && \text{by $\begin{pmatrix}
\|\proj(\tilde v)\|\\
(\|\tilde w\|-\E[G \tilde v])_+ 
\end{pmatrix} = s \begin{pmatrix}
\|\proj(v_*)\|\\
(\|w_*\|-\E[G v_*])_+
\end{pmatrix}$}\\
&= s \E[H w_*] &&\text{by $\mg(v_*,w_*)=0$}
\end{align*}
and $\E[H w_*]=\kappa_*\beta_*>0$. Substituting $t=s=u> 0$ to the previous display, noting $\|w_{*}\| > \E[v_*G]$ and $\|w_*\|>\E[v_*G]$ by \Cref{lm:positive}, we obtain
$(v_{**},w_{**}) = t(v_*, w_*)$. 
\end{proof}

\begin{proof}[Proof of \Cref{lm:unique_reg}-(1)]
\end{proof}

\begin{lemma}[Uniqueness of $w_*$]
    If $(v_*, w_*)$ solves $\min_{\mg(v,w)\le0}\ml(v,w)$, then $w_*$ is unique. 
\end{lemma}
\begin{proof}
    We proceed by contradiction. Suppose there exists another minimizer $(v_{**}, w_{**})$ such that $w_{*}\ne w_{**}$. Since at least $w_*$ or $w_{**}$ are nonzero, we may assume $w_{*}\ne 0$ without loss of generality.  Then, taking the intermediate point $(\tilde{v}_{**},\tilde{w}_{**})=((v_{*}+v_{**})/2, (w_{*}+w_{**})/2)$ if necessary, we may also assume that both of $w_*$ and $w_{**}$ are nonzero. By \Cref{lm:optimization_to_system_reg}, $(v_*, w_*)$ take the form of
\begin{align*}
v_* &= \prox[\kappa_*\loss](\alpha_* G +Z)-Z, &&w_*=\prox[\nu_*^{-1}\reg] (\nu_*^{-1}\beta_* H +X)-X
\end{align*}
where $(\alpha_*, \beta_*, \kappa_*, \nu_*)$ are the associated positive scalars, and the same things hold for $(v_{**}, w_{**})$ with $(\alpha_{**}, \beta_{**}, \kappa_{**}, \nu_{**})$. 
By \Cref{lm:proportional_reg}, $(v_{*}, w_{*})=t(v_{**}, w_{**})$ for some $t>0$. Using $\|w_*\| = \alpha_*$ and $\|w_{**}\|=\alpha_{**}$, we have 
$$
{\alpha}_{**} = t \alpha_{*} \quad \text{and}\quad \prox[\kappa_{**}\loss](\alpha_{**}G+Z)-Z = t\cdot \bigl(\prox[\kappa_{*}\loss](\alpha_{*}G+Z)-Z \bigr)
$$
Then the same argument in the proof of \Cref{lm:unique_v_*} gives $t=1$ and $w_{*}=w_{**}$, which is a contradiction with $w_{*}\ne w_{**}$. Thus, $w_*$ is unique. 
\end{proof}

\begin{lemma}[Uniqueness of $v_*$ and Lagrange multiplier $\mu_*$]
    If $w_*\ne 0$ then $v_{*}$  and the Lagrange multiplier $\mu_{*}$ satisfying the KKT condition \eqref{eq:KKT_reg} are also unique.  
\end{lemma}
\begin{proof}
The argument in the previous proof gives $v_{*}=v_{**}$. Let $\mu_{*}$ and $\mu_{**}$ be Lagrange multipliers satisfying the KKT condition. Then, by \Cref{lm:optimization_to_system_reg}, we must have 
$$
v_{*} = \prox[\kappa_{*}\loss](\alpha_{*} G +Z) - Z = \prox[\kappa_{**}\loss](\alpha_* G + Z) - Z. 
$$
where $\alpha_*=\|w_*\|$, $\kappa_{*}=\delta \mt(v_{*}, w_{*})/\mu_{*}>0$ and $\kappa_{**}=\delta\mt(v_*, w_*)/\mu_*>0$. Then the same argument in the proof of \Cref{lm:unique_lagrange_noreg} gives $\kappa_{*}=\kappa_{**}$ and $\mu_{*}=\mu_{**}$. This finishes the proof. 
\end{proof}

\subsection{Proof of \Cref{lm:nonzero_reg}}\label{proof:lm:nonzero_reg}

We proceed by contradiction. Suppose $w_*=0$. Then the constraint $\mg(v_*, w_*)\le 0$ with $w_*=0$ gives $v_*=0$ so that $(0, 0)$ solves the constrained optimization problem. 
Taking the Lagrange multiplier $\mu>0$ satisfying the KKT condition \eqref{eq:KKT_reg}, we have 
$(0, 0)\in \argmin_{(v,w)\in\mh} \ml(v,w)+\mu\mg(v,w)$. Noting $\ml(0,0)=\mg(0,0)=0$, this reads to 
\begin{align}\label{eq:contradiction}
0 \le \ml(v, w) + \mu \mg(v,w) \quad \text{for all $(v,w)$}.     
\end{align}
Removing the positive part $(\cdot)_+$ inside $\mt(v,w)=\{(\|w\|-\E[vG])_+^2 + \|\proj(v)\|^2\}^{1/2}$, 
\begin{align*}
    \ml(v,w)+\mu\mg(v,w) &\le \E[\delta(\loss(v+Z)-\loss(Z)) + \reg(w+X)-\reg(X)]\\
    &+ \mu \sqrt{\|w\|^2 + \|v\|^2 - 2\|w\| \E[vG]} -\mu \delta^{-1/2} \E[Hw]
\end{align*}
Now, for all $(s, t)\in \R_{>0}^2$, let us take  $(v_s, w_t)\in\mh$ as
$$
v_s = \prox[s a \loss](sG+Z)-Z, \quad w_t = \prox[t b \reg](t H+X)-X, 
$$
where $(a, b)$ are positive constants to be specified later. 
By the definition of  proximal operator, we have $(as)^{-1} (sG-v_s) \in \partial \loss(v_s+Z) $ and $(bt)^{-1}(tH-w_t)\in \partial \reg(w_t+X)$ so that 
\begin{align*}
    \loss(v_s+Z)-\loss(Z) \le \frac{sG-v_s}{as} v_s, \quad 
   \reg(w_t+X)-\reg(X) \le \frac{tH-w_t}{bt} w_t
\end{align*}
Then, if we take $b^{-1} = \mu \delta^{-1/2}$, noting that $\E[Hw_t]$ is cancelled out, 
\begin{align}\label{eq:A_v_t}
\ml(v_s, w_t) + \mu\mg(v_s, w_t) \le \E\bigl[\frac{\delta}{a} \bigl(G v_s -\frac{v_s^2}{s}\bigr) - \frac{\mu}{\sqrt{\delta}t}w_t^2\bigr] + \mu \sqrt{\|w_t\|^2 +\|v_s\|^2 - 2 \|w_t\| \E[v_sG]}. 
\end{align}
Let us identify the limit of $v_s/s$ and $w_t/t$ as $s,t\to 0+$. By the same argument in the proof of \Cref{lm:nonzero_noreg}, $(G-v_s/s) \in a\partial \loss(v_s+Z) $ and  $(H-w_t/t)\in b \partial \reg(w_t+X)$ give 
\begin{align*}
    \lim_{s\to 0+} v_s/s &= (G-a\max\partial \loss(Z))_+ - (a\min\partial \loss(Z)-G)_+\\
    \lim_{t\to 0+} w_t/t &= (H-b\max\partial \reg(X))_+ - (b\min\partial \reg(X)-H)_+.
\end{align*}
Moreover, $w_t/t$ and $v_s/s$ are uniformly bounded from above as $|v_s/s| \le |G| + a \|\loss\|_{\lip}$ and $|w_t/t| \le |H| + b \|\reg\|_{\lip}$ by the Lipschitz condition of $\loss$ and $\reg$, and hence 
the dominated convergence theorem yields
\begin{align*}
    \lim_{s \to 0+} \|v_s/s\| &= \|\dist(G, a\partial \loss(Z))\| =: \mathsf{d}_\loss(a)\\
    \lim_{t\to 0+} \|w_t/t\| &= \|\dist(H, b\partial \reg(X))\| =: \mathsf{d}_\reg(b)\\
    \lim_{s\to 0+}\E[Gv_s/s] &= \E[G(G-a\max\partial \loss(Z))_+ - G(a\min\partial \loss(Z)-G)_+] =: \mathsf{k}(a)
\end{align*}
Thus, if we take the limit $s, t\to 0+$ with $t/s=\mathsf{d}_\reg(b)^{-1}$, \eqref{eq:A_v_t} is reduced to 
$$
s^{-1} \bigl(\ml(v_s, w_t)+\mu\mg(v_s, w_t)\bigr) \le  \frac{\delta \mathsf{k}(a)}{a}  - \frac{\delta \mathsf{d}_\loss^2(a)}{a} - \frac{\mu \mathsf{d}_\reg(b)}{\sqrt{\delta}}+ \mu\sqrt{1 + \mathsf{d}_\loss^2(a) - 2 \mathsf{L}(\alpha)} + o(1)
$$
Using the identity $\sqrt{x} = \inf_{\tau>0} (x/(2\tau) + \tau/2)$ for the right-most term and taking $\tau>0$ such that $\mu /\tau = \delta/a$, $\delta \mathsf{k}(a)/a$ on the left-most term is canceled out and the leading term on the RHS is bounded from above by 
$$
    - \frac{\delta \mathsf{d}_\loss^2(a)}{a} - \frac{\mu \mathsf{d}_\reg(b)}{\sqrt{\delta}} + \frac{\delta}{2a} (1+\mathsf{d}_\loss^2 (a)) + \frac{\mu^2 a}{2\delta } =  \frac{\delta}{2a} \Bigl(-\mathsf{d}_\loss^2(a) - \frac{2\mu a \mathsf{d}_\reg(b)}{\delta \sqrt{\delta}} 
    + \frac{\mu^2 a^2}{\delta^2}
    +1 \Bigr).
$$
Taking $a=\sqrt{\delta} \mathsf{d}_\reg(b)/\mu$, the RHS  is reduced to 
$$
\frac{\delta}{2a} \bigl(-\mathsf{d}_\loss^2(a)- \frac{2\mathsf{d}_{\reg}^2(b)}{\delta} + \frac{\mathsf{d}_\reg^2(b)}{\delta} 
+1 \bigr) =  - \frac{\delta}{2a} \bigl(\mathsf{d}_\loss^2(a)+ \delta^{-1} \mathsf{d}_\reg^2(b) -1 \bigr).
$$
In summary, we have shown that there exist some positive constants $(a, b, r)$ such that $v_s = \prox[s a \loss](sG+Z)-Z$ and $w_t = \prox[t b \reg](t H+xX)-X$ with $t/s=r$ satisfy 
$$
\ml(v_s, w_t) + \mu\mg(v_s, w_t) < - s \cdot \frac{\delta}{2a} \bigl(\mathsf{d}_\loss^2(a)+ \delta^{-1} \mathsf{d}_\reg^2(b) -1 \bigr) + o(s)
$$
Here, noting $\mathsf{d}_\loss^2(a)\ge \inf_{t> 0} \E[\dist(G, t\partial\loss(Z))^2]$ and $\mathsf{d}_\reg^2(b) \ge \inf_{t>0} \E[\dist(H, t\partial \reg(X))^2]$, 
the leading term is bounded from below as 
$$
\tfrac{\delta}{2a} \bigl(\mathsf{d}_\loss^2(a)+ \delta^{-1} \mathsf{d}_\reg^2(b) -1 \bigr)
\ge \tfrac{\delta}{2a} (\inf_{t>0} \E[\dist(G, t\partial\loss(Z))^2] + \delta^{-1} \inf_{t>0} \E[\dist(H, t\partial \reg(X))^2] - 1),
$$
and the RHS is strictly positive
thanks to the assumption $\delta <\delta_{\mathsf{perfect}}$. Hence, we can find a sufficiently small $s', t'>0$ such that $\ml(v_s, w_t) + \mu\mg(v_s, w_t) < 0$, a contradiction with \eqref{eq:contradiction}. Therefore, we must have $w_*\ne 0$. 

\subsection{Proof of \Cref{lm:coercive_reg}}\label{proof:lm:coercive_reg}

\begin{lemma}\label{lm:l1_l2}
    There exists a constant $\tilde{C}_\delta\in (0,1]$ such that
    $$
    \PP\Bigl(|v|+|w|\ge \tilde{C}_\delta \bigl(\E[v^2]^{1/2}+\E[w^2]^{1/2}\bigr)\Bigr) \ge \tilde{C}_\delta \quad \text{for all $(v,w)\in\mh$ such that } \mg(v,w)\le 0
    $$
    \end{lemma}
\begin{proof}
    It suffices to show that there exists a constant $C_\delta>0$ such that
    \begin{equation}\label{eq:l1_l2_reg}
        \E[|v|^2]^{1/2} + \E[w^2]^{1/2} \le C_\delta (\E[|v|] + \E[|w|]) \quad \text{for all $(v,w)\in\mh$ such that $\mg(v,w)\le 0$}, 
    \end{equation}
    since the lower bound of the tail probability easily follows from the Paley--Zygmund inequality (see the argument in \Cref{lm:coercive_noreg_detail}). 

    First, let us consider the case of ${\E[w^2]^{1/2} \le \E[vG]}$. In this case, the constraint $\mg(v,w)\le 0$ reads to $\E[v^2] - \E[vG]^2 \le \delta^{-1} \E[Hw]^2$. Using the Cauchy--Schwarz inequality for $\E[wH]$,  we are left with 
        $$
       \E[v^2]^{1/2} \le \sqrt{\E[vG]^2 + \delta^{-1} \E[Hw]^2} \le \sqrt{\E[vG]^2 + \delta^{-1} \E[w^2]} \le \sqrt{1+\delta^{-1}} |\E[vG]|,
        $$
        For some $M>0$ to be specified later, the Cauchy--Schwarz inequality for $\E[|v| \cdot |G|I\{|G|>M\}]$ yields
        \begin{align*}
        |\E[vG]| \le \E\bigl[|v| |G| I\{|G|\ge M\}\bigr] + M \E[|v|] 
        &\le \E[v^2]^{1/2} \cdot \E\bigl[|G|^2I\{|G|\ge M\}\bigr]^{1/2} + M\E[|v|].
        \end{align*}
        If we take $M=M_\delta>0$ such that $\sqrt{1+\delta^{-1}} \E[|G|^2 I\{|G|>M_\delta \}]^{1/2} \le 1/2$, combined with $\E[v^2]^{1/2} \le \sqrt{1+\delta^{-1}} |\E[vG]|$, we obtain
        $$
        \E[v^2]^{1/2} \le \sqrt{1+\delta^{-1}}\Bigl( \E[v^2]^{1/2} \cdot \E\bigl[|G|^2I\{|G|\ge M_{\delta}\}\bigr]^{1/2} + M_{\delta}\E[|v|]\Bigr)
        \le \tfrac{1}{2}\E[v^2]^{1/2} + \sqrt{1+\delta^{-1}} M_\delta \E[|v|],
        $$
        so that $\E[v^2]^{1/2} \le 2 \sqrt{1+\delta^{-1}} M_\delta \E[|v|]$. Thus, combined with $\E[w^2]^{1/2} \le \E[vG]$, we obtain
        $$
        \E[v^2]^{1/2} + \E[w^2]^{1/2}
        \le \E[v^2]^{1/2} + \E[vG] \le 2 \E[v^2]^{1/2} \le 4 \sqrt{1+\delta^{-1}} M_\delta \E[|v|],
        $$
        which finishes the proof of \eqref{eq:l1_l2_reg} for the case ${\E[w^2]^{1/2}\le \E[vG]}$. 
        
        Next, we consider the other case ${\E[w^2]^{1/2}\ge \E[vG]}$. In this case the constraint $\mg(v,w)\le 0$ reads to 
        $$
        \E[w^2] + \E[v^2] \le 2 \E[w^2]^{1/2} \E[vG] + \delta^{-1}\E[Hw]^2.
        $$
        By the same argument based on the Cauchy--Schwarz inequality, we can find positive constants $M$ and $M_\delta$ such that
        $$
        |\E[vG]| \le \E[v^2]^{1/2}/4 + M \E[|v|],\quad 
        \delta^{-1/2} |\E[Hw]| \le \E[w^2]^{1/2}/4 + M_\delta \E[|w|]. 
        $$
        Substituting these bounds to the RHS of the previous display, using $(a+b)^2 \le 2(a^2+b^2)$, we are left with
        $$
        \E[w^2] + \E[v^2] \le 2^{-1} \E[w^2]^{1/2} \E[v^2]^{1/2} + 2M\E[w^2]^{1/2}\E[|v|] + 2^{-1}\E[w^2] + 2M_\delta \E[|w|]^2,
        $$
        or equivalently,
        $$
        \text{LHS} \coloneqq
        2^{-1}\E[w^2] + \E[v^2] - 2^{-1} \E[w^2]^{1/2}\E[v^2]^{1/2} \le 2M \E[w^2]^{1/2} \E[|v|] + 2M_\delta \E[|w|]^2
            \eqqcolon\text{RHS}.
        $$
        Noting that 
        \begin{align*}
        \text{LHS} &\ge 2^{-1} (\E[w^2]+ \E[v^2]-\E[w^2]^{1/2}\E[v^2]^{1/2})\\
        &\ge 
        4^{-1} (\E[w^2]^{1/2}+\E[v^2]^{1/2})^2  && \text{using $2(a^2+b^2-ab)\ge (a+b)^2$,} \\
        \text{RHS} &\le 2M \E[w^2]^{1/2}\E[|v|] + 2M_\delta \E[w^2]^{1/2} \E[|w|] &&\text{by Jensen's ineq. for $\E[|w|]\le \E[w^2]^{1/2}$},\\
        &\le 2(M+ M_\delta) \E[w^2]^{1/2}(\E[|v|]+\E[|w|])\\
        &\le 2(M+ M_\delta) (\E[w^2]^{1/2}+\E[v^2]^{1/2}) (\E[|v|]+\E[|w|]), 
        \end{align*}
        multiplying the RHS and the LHS by $4 (\E[|v|]+\E[|w|])^{-1}$, we obtain
        \begin{align*}
        \E[v^2]^{1/2}+\E[w^2]^{1/2} \le 8(M+ M_\delta) (\E[|v|]+\E[|w|]).
        \end{align*}
        This finishes the proof of \eqref{eq:l1_l2_reg} for the case $\E[w^2]^{1/2}>\E[vG]$, thereby completing the proof of \Cref{lm:l1_l2}. 
    \end{proof}

\begin{lemma}\label{lm:coercive_reg_detail}
Suppose $\loss$ and $\reg$ are convex, Lipschitz, Lipschitz, and coercive in the sense of \eqref{eq:coercive_condition}, i.e.,
$\loss(\cdot)-\loss(0) \ge a_\loss|\cdot| - b_\loss$ and $\reg(\cdot)-\reg(0)\ge a_\reg|\cdot| - b_\reg$
for some $(a_\loss, b_\loss)\in\R_{>0}\times \R_{\ge 0}$ and  $(a_\reg, b_\reg)\in\R_{>0}\times \R_{\ge 0}$. Then, there exists a  positive constant $C_\delta$ depending on $\delta$ only  such that for all $(v,w)\in\mh$ satisfying $\mg(v,w)\le 0$, $\ml(v,w)$ is bounded from below as 
$$
\ml(v,w ) \ge C_\delta (\delta a_\loss\wedge a_\reg) (\|v\|+\|w\|) - 2\bigl(\delta \|\loss\|_{\lip} +\|\reg\|_{\lip}\bigr) \mathsf{q}\Bigl(
\frac{C_\delta(\delta a_\loss\wedge a_\reg)^2}{(\delta \|\loss\|_{\lip}+\|\reg\|_{\lip})^2}
\Bigr) -(\delta b_\loss + b_\reg), 
$$
where $\mathsf{q}: \R_{>0} \to \R_{>0}$ is the nonincreasing map $\epsilon \mapsto $ 
$
\mathsf{q}(\epsilon)= \inf\{q>0: \PP(|Z|> q)\le\epsilon\} + \inf\{q>0: \PP(|X|> q)\le \epsilon\}. 
$
\end{lemma}
\begin{proof}
Recall that we have shown in \eqref{eq:l_bound_rho} that
\begin{align*}
\E[\loss(v+Z)-\loss(Z)] \ge -\sqrt{\epsilon} \|\loss\|_{\lip} \|v\| +a_\loss \E\bigl[
I\{|Z|\le \mathsf{q}_Z(\epsilon)\} |v|
\bigr] - 2\|\loss\|_{\lip} \mathsf{q}_Z(\epsilon) - b_\loss
\end{align*}
for all $\epsilon>0$, where $\mathsf{q}_Z(\epsilon) = \inf\{q>0:\PP(|Z|>q)\le \epsilon\}$. The same argument for $(\reg, X)$ yields
\begin{align*}
\E[\reg(w+X)-\loss(X)] \ge -\sqrt{\epsilon} \|\reg\|_{\lip} \|w\| +a_\reg \E\bigl[
I\{|X|\le \mathsf{q}_X(\epsilon)\} |w|
\bigr] - 2\|\reg\|_{\lip} \mathsf{q}_X(\epsilon) - b_\reg,
\end{align*}
where $\mathsf{q}_X(\epsilon) = \inf\{q>0:\PP(|X|>q)\le \epsilon\}$.
Putting the above displays together, noting $\mathsf{q} = \mathsf{q}_X + \mathsf{q}_Z$, $\ml(v,w)$ is bounded from below as 
\begin{align}
\ml(v,w) &\ge -\sqrt{\epsilon} (\delta\|\loss\|_{\lip}+\|\reg\|_{\lip})(\|v\|+\|w\|)\nonumber \\
&\quad+ \underbrace{(\delta a_\loss \wedge a_\reg)\E\bigl[I\{|Z|\le \mathsf{q}_Z(\epsilon)\}I\{|X|\le \mathsf{q}_X(\epsilon)\} (|v|+|w|)\bigr]}_{\Xi(\epsilon)} \nonumber 
\\
&\quad -2 (\delta \|\loss\|_{\lip}+\|\reg\|_{\lip}) \mathsf{q}(\epsilon) - \delta b_\loss - b_\reg\nonumber \\
&= -\sqrt{\epsilon} (\delta\|\loss\|_{\lip}+\|\reg\|_{\lip})(\|v\|+\|w\|) + \Xi(\epsilon) -2 (\delta \|\loss\|_{\lip}+\|\reg\|_{\lip}) \mathsf{q}(\epsilon) - \delta b_\loss - b_\reg.  \label{eq:L_vw_lowerbound_Xi}
\end{align}
By \Cref{lm:l1_l2}, there exists a positive constant $C_\delta \in (0,1]$ such that $\PP(|v|+|w| \ge C_\delta (\|v\|+\|w\|))\ge C_\delta$. Then, the expectation in $\Xi(\epsilon)$ is bounded from below as 
\begin{align*}
&\E\bigl[I\{|Z|\le \mathsf{q}_Z(\epsilon)\}I\{|X|\le \mathsf{q}_X(\epsilon)\} (|v|+|w|)\bigr] \\
&\ge C_\delta (\|v\|+\|w\|) \E\Bigl[
    I\{|Z|\le \mathsf{q}_Z(\epsilon)\} I\{|X|\le \mathsf{q}_X(\epsilon)\} I\{|v|+|w| \ge C_\delta (\|v\|+\|w\|)\}
\Bigr] \\
&\ge C_\delta (\|v\|+\|w\|) \Bigl(\PP\bigl(|v|+|w| \ge C_\delta (\|v\|+\|w\|)\bigr) - \PP\bigl(|Z|>\mathsf{q}_Z(\epsilon)\bigr) - \PP\bigl(|X|>\mathsf{q}_X(\epsilon)\bigr)\Bigr) &&\text{by the union bound}
\\
&\ge C_\delta  (\|v\|+\|w\|) (C_\delta - 2\epsilon)\\
&\ge (C_\delta^2 - 2\epsilon) (\|v\|+\|w\|)   && \text{by $C_\delta \le 1$}. 
\end{align*}
Thus, for all $\epsilon\in (0,1)$,
\begin{align*}
\Xi(\epsilon) &\ge (\delta a_\loss \wedge a_\reg)(C_\delta^2 - 2\epsilon) (\|v\|+\|w\|)  \\
&\ge
\bigl\{(\delta a_\loss \wedge a_\reg) C_\delta^2 - 2 \epsilon (\delta \|\loss\|_{\lip}\wedge \|\reg\|_{\lip}) \bigr\}(\|v\|+\|w\|) && \text{by $a_\loss\le \|\loss\|_{\lip}$ and $a_\reg\le \|\reg\|_{\lip}$} \\
&\ge
\bigl\{(\delta a_\loss \wedge a_\reg) C_\delta^2 - 2 \sqrt{\epsilon} (\delta \|\loss\|_{\lip}\wedge \|\reg\|_{\lip}) \bigr\}(\|v\|+\|w\|). && \text{$\sqrt{\epsilon} \ge \epsilon$ since $\epsilon\in(0,1)$}
\end{align*}
Combined with \eqref{eq:L_vw_lowerbound_Xi}, we are left with 
$$
\ml(v,w) \ge \bigl\{(\delta a_\loss \wedge a_\reg) C_\delta^2 - 3 \sqrt{\epsilon} (\delta \|\loss\|_{\lip}\wedge \|\reg\|_{\lip}) \bigr\}(\|v\|+\|w\|) -2 (\delta \|\loss\|_{\lip}+\|\reg\|_{\lip}) \mathsf{q}(\epsilon) - \delta b_\loss - b_\reg
$$
for all $\epsilon\in(0,1)$. Taking $\epsilon = \bigl\{\tfrac{C_\delta^2(\delta a_\loss\wedge a_\reg)}{6(\delta \|\loss\|_{\lip}+\|\reg\|_{\lip})}\bigr\}^2 \le \tfrac{1}{36} < 1$, we get 
\begin{align*}
\ml(v,w) \ge  \frac{(\delta a_\loss \wedge a_\reg)C_\delta^2}{2} (\|v\|+\|w\|) - 2(\delta \|\loss\|_{\lip}+ \|\reg\|_{\lip})\mathsf{q}\Bigl(
\frac{C_\delta^4(\delta a_\loss\wedge a_\reg)^2}{36(\delta \|\loss\|_{\lip}+\|\reg\|_{\lip})^2}
\Bigr) - \delta b_\loss -b_\reg, 
\end{align*}
which finishes the proof. 
\end{proof}

\section{Phase transition based on convex geometry}\label{sec:convex_geometry}
\begin{lemma}\label{lm:limit_dist_cone}
    Let $h:\R\to\R$ be a convex and Lipschitz function with $\argmin_{x} h(x)=\{0\}$. 
    Let $\bg\in \R^m$ and $\bw\in \R^{m}$ be independent random vectors with iid marginals
    \begin{align*}
        \bg = (g_i)_{i=1}^{m} \iid G =^d N(0,1), \quad \bw = (w_i)_{i=1}^{m} \iid W 
    \end{align*}
    for some distribution $W$ with $\PP(W\ne 0)>0$. Then we have 
    $$
    \tfrac{1}{m}\E\Bigl[\dist\bigl(\bg, \cone(\partial {h}(\bw)) \bigr)^2|\bw\Bigr] \to^p \inf_{t>0} \E\bigl[\dist(G, t\partial h(W))^2\bigr] \quad \text{as} \quad m\to+\infty,
    $$
    where $\partial h(\bw) = \bigtimes_{i=1}^{m} \partial h (w_i) \subset \R^m$ and $\cone(\partial h(\bw)) = \cup_{t\ge 0}\partial(t h (\bm{z}))$. 
\end{lemma}

\begin{proof}
By the explicit gradient identities \cite[(B.7)-(B.9)]{amelunxen2014living},
the Euclidean norm of the gradient of 
\(\bm g \mapsto \text{dist}(\bg, \cone(\partial {h}(\bw)) )^2 \) is bounded by \( 2\|\bm g\|_2 \). Thus, conditionally on \( \bw \), the Gaussian Poincar\'e inequality \cite[Theorem 3.20]{boucheron2013concentration} yields
$$
\E[\text{dist}(\bg,  \cone(\partial {h}(\bw)) )^2|\bw] = {\dist}(\bg, \cone(\partial {h}(\bw)))^2 + O_P(\sqrt{m}).
$$
By the identity ${\dist}(\bg, \cone(\partial {h}(\bw)))^2=\inf_{t>0} \dist(\bg, t\partial h(\bw))^2$ thanks to $\cone(\partial h(\bw))=\cup_{t>0} t\partial h(\bw)$, the above display reads to 
\begin{align*}
    \tfrac{1}{m}\E[\text{dist}(\bg,  \cone(\partial {h}(\bw)) )^2|\bw] &= \inf_{t>0} J_n(t) + O_P(m^{-1/2})
\end{align*}
where $J_m:(0,+\infty)\to\R$ is the stochastic convex function defined by
$$
t\mapsto J_m(t) \coloneqq  \tfrac 1 m \text{dist}(\bm g, t \partial h(\bw))^2.
$$
By the Law of Large Number and observing that the square distance is separable, we have the pointwise convergence
$$
J_m(t) \to^P \E[\dist(G, t\partial h(W))^2]=:J(t)
$$
for each $t>0$. Now we claim that the map $t\mapsto J(t)$ is coercive. By the triangle inequality, $\dist(G, t\partial h(W))$ is bounded from below as 
\begin{align*}
    \dist(G, t\partial h(W)) = \inf_{s \in \partial h(W)} |ts-G| \ge t \cdot \inf_{s\in\partial h(W)} |s| - |G| = b(W) \cdot t - |G|, 
\end{align*}
where $b(W)=\inf_{s\in\partial h(W)}|s|$. Note that $b(W)$ is bounded since $h$ is Lipschitz, and strictly positive on the event $\{W\ne 0\}$ since $\{0\}=\argmin_x h(x)$. Therefore, $J(t)$ is bounded from below by 
\begin{align*}
    J(t) &\ge \E[I\{W\ne 0,\  t > b(W)^{-1}|G| \} (b(W) \cdot t  - |G|)^2]\\
    &= t^2 \E[b(W)^2 I\{W\ne 0, \ t > b(W)^{-1}|G| \}] + O(t)\\
    &= t^2 \E[b(W)^2I\{W\ne 0\}] + o(t^2) && \text{using dominated convergence}
\end{align*}
as $t\to+\infty$. Since $\E[b(W)^2I\{W\ne 0\}] >0$, we have $J(t)\to+\infty$ as $t\to+\infty$. Thus, combined with the pointwise convergence $J_m(t)\to^P J(t)$, \cite[Lemma B.1]{thrampoulidis2018precise} gives $\inf_{t>0}J_m(t)\to^P \inf_{t>0} J(t)$. This finishes the proof. 
\end{proof}

\subsection{Proof of \Cref{prop:phase_transition_noreg}}\label{proof:prop:phase_transition_noreg}
By the assumption of linear response $\bm{y} = \bm{A}\bm{x}_0 + \bm{z}$ and the KKT condition for the unregularized M-estimator $\hat{\bm{x} }\in \argmin_{\bm{x}} \sum_{i=1}^n \loss(y_i - \bm{e}_i^\top \bm{A}\bm{x})$, 
the perfect recovery $\hat{\bm{x}} = \bm{x}_0$ holds if and only if 
$$
\bm{A}^\top \partial \loss(\bm{z}) =\bm{0}_p
$$
where $\partial \loss(\bm z) = \bigtimes_{i=1}^n \partial\loss(z_i) \subset \R^n$. 
Since the components of \( \bm z \) are iid with the same
distribution as \( Z \), by the assumption $\PP(Z\ne 0) > 0$, the event \(\{\bm z \ne \bm 0\} = \cup_{i=1}^n \{z_i\ne 0\}\) holds with probability approaching one
exponentially fast. In this event, since \( \loss \) is minimized only at $0$, we have
\({{\partial\loss}(\bm z) \not\ni \bm 0}\) and perfect recovery
happens if and only if
\begin{equation}
    \label{perfect_noreg}
    \ker(\bm{A}^\top)\cap \cone(\partial {\loss}(\bm{z}))  \ne \{0\}. 
\end{equation}
Since \( \bm A \) has iid normal entries independent of \( \bm z \),
the subspace \( \ker(\bm{A}^\top) \) is a rotationally invariant random subspace
in \( \R^n \) with dimension $n-p$. Therefore, \cite[Theorem I]{amelunxen2014living}
implies that for \( \eta\in (0,1/2) \), conditionally on \( \bm z \), it holds that
\begin{equation}\label{phase_recovery_noreg_prelim}
\begin{aligned}
    \statdim( \cone(\partial {\loss}(\bm{z})) ) + n- p \ge n + n^{1/2+\eta} & \Rightarrow
    \PP( \eqref{perfect_noreg} \mid \bm z) \to 1,
    \\
    \statdim(\cone(\partial {\loss}(\bm{z})) ) + n- p \le n - n^{1/2+\eta} & \Rightarrow
    \PP( \eqref{perfect_noreg} \mid \bm z) \to 0
\end{aligned}
\end{equation}
almost everywhere. 
Here, $\statdim(\cone(\partial {\loss}(\bm{z})))$ is the \emph{statistical dimension}
    of $\cone(\partial {\loss}(\bm{z}))$, which can be written explicitly as, thanks to \cite[Proposition 3.1]{amelunxen2014living}, 
\begin{equation*}
    \statdim(\cone(\partial {\loss}(\bm{z}))) = n - \E[\text{dist}(\bg, \cone(\partial {\loss}(\bm{z})))^2|\bm z] \quad  \text{for}\quad \bm{g}\sim \N(\bm{0}_n, I_n), \quad \bm{g}\ind \bm{\bm{z}}
\end{equation*} 
Combined with the convergence $n^{-1} \E[\text{dist}(\bg, \cone(\partial {\loss}(\bm{z})))^2|\bm z] \to^p \inf_{t>0} \E[\dist(G, t\partial\loss(Z))^2]$ by \Cref{lm:limit_dist_cone}, we have 
\begin{equation}
    \tfrac 1 n \statdim(\cone(\partial\loss(\bm{z}))) \to^P 1-\inf_{t>0} \E[\dist(G, t \partial \loss(Z))^2]. 
    \label{eq:limit_statdim_noreg}
\end{equation}
Dividing by \( n \) the left-hand side of the implications
in \eqref{phase_recovery_noreg_prelim} and noting $\lim \tfrac{n}{p}=\delta$, we obtain 
\begin{align*}
    \textstyle
    1  - \inf_{t>0}\E[\dist(G,t\partial \loss(Z))^2] - \delta^{-1} >0 
    &\Rightarrow 
        \PP(\text{perfect recovery \eqref{perfect_noreg}}) \to 1,\\
    \textstyle
    1  - \inf_{t>0}\E[\dist(G,t\partial \loss(Z))^2] - \delta^{-1} < 0 
    &\Rightarrow
        \PP(\text{perfect recovery \eqref{perfect_noreg}})  \to 0.
\end{align*}
This finishes the proof. 
\subsection{Proof of \Cref{prop:phase_transition_reg}}\label{proof:prop:phase_transition_reg}
By the KKT conditions, the perfect recovery happens
for some \( \lambda>0 \)
if and only if
$$
\bm A^\top \partial  \loss(\bm z) \cap 
\bigl(
{\lambda}
\partial  \reg (\bm x_0) \bigr) \ne \emptyset \quad \text{for some} \quad \lambda > 0. 
$$
Note that the event \(\{\bm z \ne \bm 0 \}\) holds exponentially large probability thanks to $\PP(Z\ne 0)>0$. In this event, since \( \loss \) is minimized only at $0$, we have
\({{\partial\loss}(\bm z) \not\ni \bm 0}\) and perfect recovery
happens for some \( \lambda>0 \) if and only if
\begin{equation}
\bm A^\top \text{cone}(\partial \loss(\bm z))
\cap 
\text{cone}(\partial \reg(\bm x_0))
\ne \{\bm{0}_p\}.
\label{eq:condition_perfect_recovery_reg_cones}
\end{equation}
Theorem 1.1 in
\cite{han2022gaussian} extends results of \cite{amelunxen2014living}
and gives the following phase transition
for \eqref{eq:condition_perfect_recovery_reg_cones}.
Conditionally on \( (\bm z,\bm x_0) \),
\begin{equation}\label{phase_recovery_reg_prelim}
\begin{split}
\sqrt{\statdim(\text{cone}(\partial \loss(\bm z)))}
-
\sqrt{p-\statdim(\text{cone}(\partial \reg(\bm x_0)))}
> n^{1/4}
&\Rightarrow
    \PP( \eqref{eq:condition_perfect_recovery_reg_cones} \mid
    \bm x_0, \bm z ) \to 1,
\\
\sqrt{\statdim(\text{cone}(\partial \loss(\bm z)))}
-
\sqrt{p-\statdim(\text{cone}(\partial \reg(\bm x_0)))}
< n^{1/4}
&\Rightarrow
    \PP( \eqref{eq:condition_perfect_recovery_reg_cones} \mid
    \bm x_0, \bm z ) \to 0
\end{split}    
\end{equation} 
where the convergence of conditional probabilities hold
almost surely. Here, by the same argument in the proof of \eqref{eq:limit_statdim_noreg}, it holds that 
\begin{align*}
\tfrac 1 n \statdim(\text{cone}(\partial \loss(\bm z)))
&\to^P 1- \inf_{t>0}\E[\dist(G,t\partial \loss(Z))^2],
\\ 
\tfrac 1 p \statdim(\text{cone}(\partial \reg(\bm x_0)))
&\to^P 1- \inf_{s>0}\E[\dist(H,s\partial \reg(X))^2],
\end{align*}
Thus, dividing by \( \sqrt n \) the left-hand side of the implications
in \eqref{phase_recovery_reg_prelim} and noting $\lim \tfrac{n}{p}=\delta$, we are left with
\begin{align*}
1- \inf_{t>0}\E[\dist(G,t\partial \loss(Z))^2]
-
\delta^{-1}
\inf_{s>0}\E[\dist(H,s\partial \reg(X))^2]
> 0
&\Rightarrow
    \PP( \eqref{eq:condition_perfect_recovery_reg_cones}) \to 1,
\\
1- \inf_{t>0}\E[\dist(G,t\partial \loss(Z))^2]
-
\delta^{-1}
\inf_{s>0}\E[\dist(H,s\partial \reg(X))^2]
< 0
&\Rightarrow
    \PP( \eqref{eq:condition_perfect_recovery_reg_cones}) \to 0.
\end{align*}
This establishes the phase transition for perfect recovery. 

\section{Explicit upper bound of solution $\alpha_*$}\label{sec:upper_bound_alpha}

\subsection{Unregularized M-estimation}\label{subsec:upper_bound_alpha_noreg}
The coercivity of the map $v\mapsto \ml(v)$ over the constraint $\{v\in\mh:\mg(v)\le 0\}$ (\Cref{lm:coercive_noreg_detail}) and the representation $\alpha_*=\|v_*\|/\sqrt{1-\delta^{-1}}$ lead to the following corollary. 
\begin{corollary}\label{cor:bound_noreg}
    The solution 
    $\alpha_*$ to \eqref{eq:system_noreg} is bounded from above as 
    \begin{align*}
        \alpha_*  \le  
            \frac{\|\loss\|_{\lip}}{ac_\delta } \mathsf{q}\Bigl(
                c_\delta \frac{a^2}{\|\loss\|_{\lip}^2}
            \Bigr) + \frac{b}{a c_\delta } 
    \end{align*}
    where $\mathsf{q}(x) = \inf\{q>0: \PP(|Z|>q)\le x\}$ and $(a, b)\in\R_{>0}\times\R_{\ge 0}$ is any pair satisfying \eqref{eq:coercive_condition}. 
\end{corollary}

\begin{exam}[Quantile loss]\label{exam:quantile}
Let us consider the quantile loss
$$
\rho_q(x) := q(x)_+ - (1-q) (x)_{-}\quad \text{for $q\in (0,1)$}, 
$$
and let $a_*(q)$ be the solution to the nonlinear system \eqref{eq:system_noreg} with $\loss$ being the quantile loss $\rho_q$.
Noting $\|\rho_q\|_{\lip}=\max(q, 1-q) \le 1$ and $\rho_q(\cdot)\ge \min(q, 1-q)|\cdot|$, using that $\mathsf{q}(\cdot)$ is nondecreasing, \Cref{cor:bound_noreg} gives
$$
\alpha_*(q) \le \frac{\|\rho_q\|_{\lip}}{c_\delta \{ q\wedge (1-q)\}} \mathsf{q}\bigl(
    c_\delta \frac{q^2\wedge(1-q)^2}{\|\rho_q\|_{\lip}^2}
\bigr) \le  \frac{\mathsf{q}\bigl(c_\delta\cdot (q^2 \wedge (1-q)^2)\bigr)}{c_\delta \cdot (q\wedge (1-q))}
$$
for all $q\in(0,1)$.
\end{exam}

\subsection{Regularized M-estimation}\label{subsec:upper_bound_alpha_reg}
The coercivity of the map $(v, w) \mapsto \ml(v, w)$ over the constraint $\{(v, w)\in\mh:\mg(v, w)\le 0\}$ (\Cref{lm:coercive_reg_detail}) and the representation $\alpha_*=\|w_*\|$ yield the corollary below.
\begin{corollary}\label{cor:bound_reg}
    The solution $\alpha_*$ to the nonlinear system \eqref{eq:system_reg} is bounded from above as
    \begin{align*}
        \alpha_* \le \frac{b_\loss + b_\reg}{c_\delta(a_\loss \wedge a_\reg)} + \frac{\|\loss\|_{\lip} + \|\reg\|_{\lip}}{c_\delta(a_\loss \wedge b_\loss)} \mathsf{q}\Bigl(\frac{c_\delta (a_\loss^2 \wedge b_\loss^2)}{\|\loss\|_{\lip}^2 + \|\reg\|_{\lip}^2}\Bigr)
    \end{align*}
    where $\mathsf{q}(t) :=\inf\{q>0: \PP(|Z|\ge q) \le t\} + \inf\{q>0: \PP(|X|\ge q) \le t\}$, and $(a_\loss, b_\loss)\in \R_{>0}\times \R_{\ge 0}$ are any constants such that $\loss(\cdot)-\loss(0)\ge a_\loss|\cdot|-b_\loss$, and $(a_\reg, b_\reg)\in \R_{>0}\times \R_{\ge 0}$ are any constants such that $\reg(\cdot)-\reg(0)\ge a_\reg|\cdot|-b_\reg$, 
  \end{corollary}
  
  \begin{exam}[L1 loss and L1 regularizer]
    When the loss is  L1 loss $\loss(x)=|x|$ and the regularizer $\reg$ is the L1 loss $\lambda |x|$ with some regularization parameter $\lambda>0$, taking $b_\loss=b_\reg$ and $a_\loss=1$, $a_\reg=\lambda$ in the above display, we obtain that
    \begin{align*}
    \alpha_*(\lambda) \le \frac{\mathsf{q}(c_\delta \min(1, \lambda^2))}{c_\delta \min(1,\lambda)} 
    \end{align*}
  \end{exam}

\section{Miscellaneous useful facts}\label{sec:misc}
\begin{lemma}\label{lm:convex_coercive}
    If a convex function $f:\R\to\R$ satisfies $\{0\}=\argmin_x f(x)$, then $f$ is coercive in the sense that
    there exists some constants $a_f>0$ and $b_f\ge 0$ such that 
    \begin{equation}\label{eq:coercive_condition}
        f(x) -f(0) \ge a_f|x|-b_f \quad \text{for all $x\in\R$}. 
    \end{equation}
\end{lemma}

\begin{proof}
    We assume $f(0)=0$ without loss of generality. For any choice 
    of subdifferential $\mathsf{d}(\cdot)\in\partial f(\cdot)$, evaluating the subdifferential at $1$ and $-1$, we have
    $$
    f(x)\ge f(1) + (x-1) \mathsf{d}(1)\quad \text{and} \quad f(x) \ge f(-1) + (x+1) \mathsf{d}(-1), 
    $$
    for all $x\in\R$. Noting $\mathsf{d}(-1) < 0 < \mathsf{d}(1)$ by the monotonicity of the subdifferential and $\{0\}=\argmin_x f(x)$, the previous display yields 
    \begin{align*}
        f(x) &\ge \min(|\mathsf{d}(1)|, |\mathsf{d}(-1)|) |x| - \max (|\mathsf{d}(1) - f(1)|, |\mathsf{d}(-1) + f(-1)|). 
    \end{align*}
    If we take  $a_f=\min(|\mathsf{d}(\pm1)|)>0$ and $b_f=\max(|\mathsf{d}(\pm1)|-\pm f(\pm 1)|) \ge 0$ in the above display, we obtain the coercivity of $f$. 
\end{proof}

\begin{prop}\label{prop:lipschitz_necessary}
    If the convex $\loss:\R\to\R$ is not Lipschitz then there exists a noise distribution $Z$ such that the solution $\alpha_*$ to the nonlinear system \eqref{eq:system_noreg} is not finite. 
  \end{prop}
\begin{proof}
For the proximal operator 
\begin{equation*}
    \mathsf{p}(\cdot):=\prox[\kappa\loss](\cdot),
\end{equation*}
note that $\mathsf{p} (x)x \ge \mathsf{p} (x)^2$ since $\mathsf{p} (x)=\prox[\kappa\loss](x)$ is firmly non-expansive and $\mathsf{p} (0)=0$ by $\{0\} = \argmin_x\loss(x)$. This gives that $x-\mathsf{p} (x)\ge 0$ if $x>0$. Moreover, $x-\mathsf{p} (x)$ is nondecreasing since $\kappa^{-1}(x-\mathsf{p} (x))$ is equal to the derivative $\env_\loss'(x; \kappa)$ of the Moreau envelope $\env_\loss(x; \kappa)=(x-\mathsf{p} (x))^2/(2\kappa)+\loss(\mathsf{p} (x))$, which is convex in $x$. Then, if $(\alpha, \kappa)$ is a solution to the nonlinear system \eqref{eq:system_noreg}, noting $(\alpha G + Z)>Z>0$ on the event $\{G>0, Z>0\}$, 
the first equation in \eqref{eq:system_noreg} gives that
$\alpha^2$ is bounded from below as 
\begin{align*}
    \alpha^2 
    &\ge {\delta} \E\bigl[
    I\{Z> 0\} I\{G>0\} 
    \bigl({(\alpha G + Z)-\prox[\kappa\loss](\alpha G +Z)}\bigr)^2 \bigr] \\
    &\ge {\delta} \E[
     I\{Z>0\} I\{G>0\}
    (Z-\prox[\kappa\loss](Z))^2].
\end{align*}
Here, $Z-\prox[\kappa\loss](Z)$ is no smaller than $\min(\tfrac{Z}{2}, \kappa\max\partial\loss(\tfrac{Z}{2}))$ since if $Z-\prox[\kappa\loss](Z)<\tfrac{Z}{2}$ then we have $\tfrac{Z}{2} < \prox[\kappa\loss](Z)$, which in turn gives 
\begin{equation*}
    Z-\prox[\kappa\loss](Z) \ge
    \kappa\min\partial \loss (\prox[\kappa\loss](Z))
    > \kappa\max\partial\loss(Z/2)
\end{equation*}
by the monotonicity of the subdifferential.
Thus, using the independence of $(G,Z)$ and $\PP(G>0)=1/2$,
we have shown that
$\alpha$ is bounded from below as 
$$
\alpha^2 \ge \tfrac \delta 2 \min(1, \kappa) \E\bigl[\min\bigl(\tfrac{W_+}{2}, \max\partial\loss(\tfrac{W_+}{2})\bigr)^2\bigr].
$$
If \( \loss \) is not Lipschitz on \( \R_{\ge 0} \)
then \( \max \partial \loss(x)\to+\infty \) as \( x\to+\infty \).
The minimum of two nondecreasing functions is still nondecreasing,
so \( r(x)=\min\{x, \max\partial \loss(x)\} \) is also nondecreasing
and \( r(x)\to+\infty \) as \( x\to+\infty \).
We can thus always construct a distribution for \( Z \)
such that the right-hand side of the previous display is infinite.
For instance, if \( w_k>0 \) are such that \( r(w_k)\ge k \),
we define the point masses
\( \PP(Z=w_k) \asymp 1/k^2 \) for \( k\in \mathbb{N}\). 
\end{proof}

\begin{lemma}\label{lm:Lagrange}
    Let $\mh$ be a Hilbert space and let $\ml:\mh\to\R$ and $\mg_{1}, \dots, \mg_m:\mh\to\R$ be finite-valued convex functions over $\mh$. Suppose that exists a point $v_0\in \mh$ such that $\cap_{i=1}^m \{\mg_i(v_0) < 0\}\ne \emptyset$. Then the claims $(1)$-$(2)$ below are equivalent:
    \begin{enumerate}
        \item $v_*\in\mh$ is a solution to the constrained optimization problem
        $$
        \min_{v\in\mh} \ml(v) \quad \text{subject}\quad \mg_i(v)\le 0 \quad  (i=1, \dots, m)
        $$
        \item There exists some nonnegative scalars $ (\mu_i)_{i=1}^m \in \R_{\ge 0}^m$ and some subdifferential $(u_i)_{i=1}^m \in \partial \mg_i$ such that 
        \begin{align*}
           - \sum_{i=1}^m \mu_i u_i\in \partial \ml(v_*) \quad\text{and}\quad \begin{dcases}
            \mg_i(v_*)\le 0\\
            \mu_i \mg_i = 0 
           \end{dcases} \quad \text{for all} \quad i\in [m]. 
        \end{align*}
    \end{enumerate}
\end{lemma}
\begin{proof}
    By \cite[Proposition 27.21]{bauschke2017correction}, it suffices to show that the Slater condition below is satisfied:
    \begin{align*}
        \begin{cases}
            \{v\in\mh: \mg_i(v)\le 0\} \subset \text{int}\  \text{dom} \mg_i \\
            \text{dom} \ml \cap \bigcap_{i=1}^m \{v\in\mh: \mg_i(v) < 0\} \ne \emptyset
        \end{cases}
    \end{align*}
    where $\text{dom} f =\{v\in\mh: f(v)<+\infty\}$ is a domain for a function $f$. Indeed,  $\ml$ and $\mg$ are finite-valued so that $\text{dom} \mg_i=\text{dom} \ml = \mh$ for all $i \in[n]$. On the other hand, $\bigcap_{i=1}^m \{v\in\mh: \mg_i(v) < 0\}$ is nonempty by assumption. This means that the Slater condition is satisfied. 
\end{proof}
\end{document}